\documentclass{article}
\usepackage{fullpage,enumitem}
\usepackage[margin=1in]{geometry}
\usepackage[dvipsnames]{xcolor}
\usepackage{cite}
\usepackage{hyperref}
\hypersetup{
   colorlinks=true,
   linkcolor=red,
   citecolor=blue,
   hypertexnames=false
}

\usepackage{comment}

\usepackage[T1]{fontenc}
\usepackage[nopatch,expansion=false]{microtype}
\usepackage[english]{babel}
\usepackage{amsmath,amssymb,amsthm,amsbsy,amsfonts,algorithm,algorithmic}
\usepackage{times}
\usepackage{colortbl, booktabs, cleveref}
\usepackage{float,graphicx}

\newcommand{\tmop}[1]{\ensuremath{\operatorname{#1}}}

\newtheorem{proposition}{Proposition}[section]
\newtheorem{lemma}{Lemma}[section]
\newtheorem{theorem}{Theorem}[section]
\newtheorem{definition}{Definition}[section]
\newtheorem{assumption}{Assumption}[section]
\theoremstyle{remark}
\newtheorem{remark}{Remark}[section]

\crefname{proposition}{Proposition}{Propositions}
\crefname{lemma}{Lemma}{Lemmas}
\crefname{theorem}{Theorem}{Theorems}
\crefname{definition}{Definition}{Definitions}
\crefname{assumption}{Assumption}{Assumptions}
\crefname{remark}{Remark}{Remarks}
\crefname{algorithm}{Algorithm}{Algorithms}
\crefname{section}{Section}{Sections}
\crefname{figure}{Figure}{Figures}
\crefname{equation}{}{}

\newcommand{\vct}[1]{\mathbf{#1}}
\newcommand{\mtx}[1]{\mathbf{#1}}
\newcommand{\set}[1]{\mathcal{#1}}

\def \vlambda {{\boldsymbol{\lambda}}}
\def \valpha  {{\boldsymbol{\alpha}}}
\def \vbeta   {{\boldsymbol{\beta}}}
\def \vgamma  {{\boldsymbol{\gamma}}}

\def \mGamma  {{\boldsymbol{\Gamma}}}
\def \mPhi    {{\boldsymbol{\Phi}}}

\def \vb {\vct{b}}

\def \vd {\vct{d}}

\def \vu {\vct{u}}
\def \vv {\vct{v}}
\def \vw {\vct{w}}
\def \vx {\vct{x}}
\def \vy {\vct{y}}
\def \vz {\vct{z}}

\def \mA {\mtx{A}}
\def \mB {\mtx{B}}
\def \mC {\mtx{C}}
\def \mD {\mtx{D}}
\def \mF {\mtx{F}}
\def \mH {\mtx{H}}
\def \mJ {\mtx{J}}
\def \mL {\mtx{L}}
\def \mM {\mtx{M}}
\def \mN {\mtx{N}}
\def \mP {\mtx{P}}
\def \mQ {\mtx{Q}}
\def \mR {\mtx{R}}
\def \mS {\mtx{S}}
\def \mT {\mtx{T}}
\def \mU {\mtx{U}}
\def \mV {\mtx{V}}
\def \mW {\mtx{W}}
\def \mX {\mtx{X}}
\def \mY {\mtx{Y}}
\def \mZ {\mtx{Z}}

\def \calA {\set{A}}
\def \calL {\set{L}}
\def \calT {\set{T}}

\DeclareMathOperator*{\minimize}{minimize}
\DeclareMathOperator*{\argmin}{argmin}
\def \st { \;\; \mbox{s.t.} \;\; }
\def \nn       {\nonumber}

\def \eye      {\mathbf{I}}
\def \vzero    {\vct{0}}
\def \mzero    {\mtx{0}}
\def \lg       {\left\langle}
\def \rg       {\right\rangle}
\def \Null     {\tmop{null}}
\def \kr       {\circledast}
\def \R        {\mathbb{R}}

\usepackage[toc, page]{appendix}

\title{Second-Order KKT Guarantees for \\ Bregman ADMM in Nonconvex and Non-Lipschitz Optimization}

\author{Shuang Li, Zhihui Zhu, and Qiuwei Li\thanks{SL is with the Department of Electrical and Computer Engineering, Iowa State University, Ames, Iowa 50014, USA. ZZ is with the Department of Computer Science and Engineering, Ohio State University, Columbus, Ohio 43210, USA. QL is with DAMO Academy, Alibaba Group US, Bellevue, Washington 98004, USA.}}

\date{\today}

\begin{document}
\setlength{\emergencystretch}{2em}
\hfuzz=1pt

\allowdisplaybreaks
\maketitle

\begin{abstract}
We analyze Bregman ADMM for nonconvex linearly constrained problems under two-sided relative smoothness, a condition that replaces the standard Lipschitz gradient assumption with a Hessian comparison relative to a Bregman kernel. This setting covers polynomial objectives arising in matrix and tensor models for which a global Lipschitz-gradient constant need not exist. We show that on an invariant open state-space domain, one iteration of Bregman ADMM defines a smooth primal--dual fixed-point map whose strict-saddle KKT points are unstable fixed points; consequently, from random initialization the iterates converge to a strict saddle with probability zero. Combined with existing first-order convergence results, this yields almost-sure second-order stationarity of limiting KKT points. We extend the analysis to a multi-block star consensus formulation for distributed optimization. The technical novelty lies in a determinant reduction with a Bregman-specific symmetrization and scaling step in the two block spectral argument, together with a null space cancellation exploiting the star graph structure in the consensus case. Numerical experiments on distributed matrix factorization illustrate the theory, and a symmetric tensor factorization example demonstrates the broader Bregman proximal splitting idea beyond the separable consensus setting.
\end{abstract}

\section{Introduction}
We consider the following optimization problem
\begin{equation}\begin{split}
	\minimize_{\vx\in\R^{n},\vy\in\R^{m}}&~ f(\vx,\vy)\doteq f_{1}(\vx)+f_{2}(\vy)
	 \st \mA\vx+\mB\vy=\vb,
\end{split}\label{eqn:problem:2}
\end{equation}
where $\mA\in\R^{p\times n}$, $\mB\in\R^{p\times m}$, and $\vb\in\R^p$.
The functions $f_1:\R^n\rightarrow \R$ and $f_2:\R^m\rightarrow \R$ can be nonconvex and their gradients need not be globally Lipschitz continuous.
The augmented Lagrangian is
\begin{equation}\begin{split}
\calL(\vx,\vy,\vlambda)=&f_1(\vx)+f_2(\vy)+\lg \vlambda, \mA\vx+\mB\vy-\vb\rg +\frac{\rho}{2}\|\mA\vx+\mB\vy-\vb\|_{2}^{2},
\end{split}\label{eqn:Lagrange:2}
\end{equation}
where $\vlambda\in\R^{p}$ is the Lagrangian dual variable and $\rho>0$ is the augmented penalty parameter. Throughout the paper, we treat $\rho$ as a fixed constant (i.e., it does not vary with the iteration index).

A standard approach for \eqref{eqn:problem:2} is the alternating direction method of multipliers (ADMM; see \Cref{alg:admm})~\cite{glowinski1975approximation,gabay1976dual}. Each iteration alternates between minimizing over $\vx$ and over $\vy$, followed by a dual ascent step for $\vlambda$.
\begin{algorithm}[H]
\caption{ADMM}
\label{alg:admm}
\begin{algorithmic}[1]

\STATE {\bf Initialization:}  $(\vy^{0},\vlambda^0)$

\STATE {\bf For} $k=1,2, \dots$,
 generate $(\vx^{k},\vy^{k},\vlambda^k)$ by
\vspace{-0.2em}
\begin{equation}
\begin{aligned}
\vx^{k}=& \argmin_{\vx} \calL(\vx,\vy^{k-1},\vlambda^{k-1}),
\\[-.1em]
\vy^{k}=& \argmin_{\vy} \calL(\vx^{k},\vy,\vlambda^{k-1}),	
\\[-.1em]
\vlambda^{k}=& \vlambda^{k-1}+\rho(\mA\vx^{k}+\mB\vy^{k}-\vb).
\end{aligned}
\label{eqn:admm}
\end{equation}
\vspace{-0.8em}
\end{algorithmic}
\end{algorithm}
  
Common variants include proximal ADMM~\cite{xu2007proximal,gonccalves2017convergence}, which adds proximal regularization in the primal updates, and linearized ADMM \cite{hong2018gradient}, which replaces the primal minimizations by alternating gradient-type steps.
These splitting schemes are well suited to structured models where separability can be exploited, with applications spanning low-rank matrix recovery~\cite{xu2012alternating}, tensor factorization, consensus optimization, and beyond; see~\cite{boyd2011distributed} for an overview.

\paragraph{First-order convergence and complexity for nonconvex ADMM.}
For nonconvex linearly constrained problems and related consensus formulations, convergence of ADMM to first order stationary points has been studied under various regularity conditions. \cite{hong2016convergence} establishes convergence for certain nonconvex consensus and sharing formulations under regularity assumptions on the objective and subproblems and sufficiently large augmented Lagrangian penalty parameters, while \cite{wang2014convergence} analyzes a Bregman modification of ADMM for nonconvex composite problems and proves convergence to stationary points of the associated augmented Lagrangian under suitable assumptions. \cite{wang2015global} further proves global convergence of ADMM for nonconvex nonsmooth problems by combining conditions that ensure lower boundedness, sufficient descent and a subgradient bound, together with the Kurdyka {\L}ojasiewicz (K\L{}) inequality. More recently, \cite{barber2024convergence} develops finite iteration convergence guarantees under restricted strong convexity without requiring smoothness or differentiability, and \cite{mancino2023adapd} proposes communication efficient decentralized primal dual methods for smooth nonconvex consensus optimization.

\paragraph{Second-order stationarity and strict-saddle avoidance.}
Moving beyond first-order stationarity is useful because many structured nonconvex problems have benign landscapes in which all second-order stationary points are globally optimal; see, e.g., low-rank matrix recovery~\cite{ge2016matrix} and tensor decomposition~\cite{ge2015escaping}. For ADMM-type methods, \cite{hong2018gradient} uses a dynamical-systems viewpoint and the stable manifold theorem~\cite[Theorem III.7]{shub2013global} to show that linearized ADMM avoids strict saddles almost surely under a Lipschitz-gradient assumption.
Beyond ADMM, strict-saddle avoidance has been extended to broader classes of first-order and proximal-type algorithms. In particular, \cite{panageas2019vanishingsteps} shows that a stable-manifold argument continues to apply under vanishing stepsizes (covering, among others, mirror-descent-type updates), and \cite{davis2022proximal} establishes that proximal methods avoid (active) strict saddles for weakly convex nonsmooth objectives.
We note that \cite{davis2022proximal} works in the nonsmooth weakly convex regime and introduces the notion of \emph{active strict saddle}, which accounts for the combinatorial structure of the subdifferential; their setting and ours are largely complementary: we treat smooth (but non-Lipschitz-smooth) objectives within ADMM splitting, whereas they handle nonsmooth objectives via proximal point iterations.

\paragraph{Bregman geometry and relative smoothness.}
A key limitation of Lipschitz-gradient analyses is that $\|\nabla^2 f\|$ may be unbounded in many applications of interest (e.g., polynomial objectives arising in matrix/tensor models), so the standard smoothness constant does not exist globally. Relative smoothness and Bregman first-order methods \cite{bauschke2016descent,bolte2018first} provide a principled alternative by replacing Euclidean geometry with a problem-adapted distance-generating function $h$ (see also the Bregman divergence \cite{censor1981iterative,bregman1967relaxation}). Bregman-divergence-based ADMM methods have been studied for first-order convergence, e.g., \cite[Theorem III.8]{wang2014convergence} and \cite[Theorem 3.9]{wang2018convergence}; we discuss specific Bregman ADMM variants next.

\paragraph{Bregman ADMM variants.}
Early Bregman variants of ADMM were proposed to better match the problem geometry and to simplify subproblems by replacing Euclidean proximal terms with Bregman divergences; see, e.g., the Bregman ADMM of \cite{wang2014bregman}. More recent work has developed first-order convergence theory in nonconvex and multi-block settings. For instance, \cite{liu2024improvedbregmanadmm} studies an improved Bregman ADMM with convergence-rate guarantees, \cite{pham2024bregmanproxlinearizedadmm} considers a Bregman proximal-linearized scheme for separable sums coupled by a difference of functions term, and \cite{cui2025mpsbadmm} analyzes a distributed multi-block variant for nonconvex and nonsmooth sharing problems; see also \cite[Theorem 3.9]{wang2018convergence} for a representative first-order result. What is missing from this line of work is a mechanism excluding convergence to strict saddles. Our contribution fills this gap: under relative (bi-)smoothness, strict saddles are unstable fixed points of the induced smooth map, and random initialization therefore rules out convergence to them almost surely. Combined with standard first-order convergence assumptions, this yields second-order stationarity of limiting KKT points.

At the same time, the extension from the Euclidean and Lipschitz setting of \cite{hong2018gradient} to Bregman ADMM is not a trivial substitution of proximal terms. The Euclidean analysis in \cite{hong2018gradient} also reduces its instability proof, after a suitable transformation, to a real symmetric spectral argument. In the Bregman setting, however, replacing the Euclidean proximal term by a Bregman divergence introduces metric blocks from the kernel Hessians, so the determinant reduction and sign change argument must be rederived in the Bregman geometry. In the two block case, this requires a diagonal rescaling that symmetrizes the reduced spectral problem and a correspondingly rescaled strict saddle direction. In the consensus case, the proof is again not a direct lift of the Euclidean argument: the star consensus structure creates a hub and peripheral coupling pattern, and the key sign change relies on a null space cancellation of the consensus penalty, since the equal $\vx$ test direction lies in the kernel of the star graph Laplacian. These are the main new technical ingredients behind the strict saddle result in the Bregman setting.

%\SL{At the same time, the extension from the Euclidean and Lipschitz setting of \cite{hong2018gradient} to Bregman ADMM is not a trivial substitution of proximal terms. The Euclidean analysis in \cite{hong2018gradient} also reduces its instability proof, after a suitable transformation, to a real symmetric spectral argument. In the Bregman setting, however, replacing the Euclidean proximal term $\frac{1}{2\eta}\|\cdot\|^2$ by a Bregman divergence $\frac{1}{\eta}D_h$ introduces metric blocks from the kernel Hessians, so the determinant reduction no longer follows the same algebraic form. The two block proof therefore requires a Schur reduction specific to the Bregman setting and a $\mu$ dependent diagonal similarity transform (with scaling $s(\mu)$, see \Cref{pf:thm:badmm}), together with a $\mu$ dependent test direction $\vy(\mu)=(d_x,s(\mu)d_y)$ whose components converge to the strict saddle direction as $\mu\to0^+$. In the consensus case, the proof is again not a direct lift of the Euclidean argument: the star consensus structure creates a hub and peripheral coupling pattern, and the key sign change relies on a null space cancellation of the consensus penalty, since the equal $\vx$ test direction lies in the kernel of the star graph Laplacian. These are the main new technical ingredients behind the strict saddle result in the Bregman setting.}

\paragraph{Contributions.}
Our main contributions are as follows.
\begin{itemize}[leftmargin=*,topsep=0pt,partopsep=0pt,itemsep=0pt,parsep=2pt]
\item \textbf{Strict-saddle avoidance for Bregman ADMM beyond Lipschitz smoothness.}
We analyze Bregman ADMM (\Cref{alg:badmm}) under relative (bi-)smoothness with respect to suitable Bregman kernels. By viewing one iteration as a smooth primal--dual fixed-point map and showing that strict saddles are unstable fixed points, we prove that with random initialization the probability that the iterates converge to a strict saddle is zero, without requiring a Lipschitz gradient assumption. The two block proof requires a determinant reduction, symmetrization, and scaling step specific to the Bregman setting in the spectral argument; see the discussion around \Cref{sec:analysis}.

\item \textbf{Second-order stationarity of limit points as a corollary of strict-saddle avoidance.}
Under additional assumptions guaranteeing first-order convergence to a KKT point (e.g., the boundedness and K\L{} and subanalyticity conditions in \cite[Theorem III.8]{wang2014convergence} and \cite[Theorem 3.9]{wang2018convergence}), our instability result implies that limiting KKT points are almost surely second-order stationary. In the Euclidean special case when the two Bregman kernels are chosen as $h_1(\vx)=\tfrac{1}{2}\|\vx\|_2^2$ and $h_2(\vy)=\tfrac{1}{2}\|\vy\|_2^2$, \Cref{alg:badmm} reduces to proximal ADMM. A recent work of \cite{gao2026note} establishes strict saddle avoidance and almost sure second order convergence for Euclidean proximal ADMM. Relative to this Euclidean proximal ADMM result, our contribution is to extend the strict saddle avoidance mechanism to the Bregman setting under relative smoothness, which covers polynomial objectives without a global Lipschitz constant, and to develop the corresponding multi block consensus extension.

\item \textbf{Multi-block consensus extensions with the same second-order stationarity guarantee.}
We extend the framework to a global consensus formulation and prove the same strict-saddle avoidance guarantees for Consensus Bregman ADMM; see \Cref{sec:mult}. Here the key new ingredient is a star-consensus structural reduction, together with a hub/non-hub null-space argument that isolates the negative-curvature direction despite the consensus penalty.

\item \textbf{Practical implications for polynomial objectives and distributed factorization problems.}
Using known relative-(bi-)smoothness constructions for polynomial objectives, our theory applies to a broad class of non-Lipschitz models arising in matrix/tensor optimization. We report numerical experiments on distributed matrix factorization and symmetric tensor factorization in \Cref{sec:application}.
\end{itemize}
The paper is organized as follows. \Cref{sec:prelim} reviews stationarity notions and tools from Bregman geometry and relative smoothness. \Cref{sec:bregadmm} presents Bregman ADMM and the strict-saddle avoidance result. \Cref{sec:mult} extends the algorithm and guarantees to a global consensus setting. \Cref{sec:analysis} first presents the single-block core mechanism through the proof for the Bregman augmented Lagrangian method (Bregman ALM), and then gives proof sketches explaining how the argument extends to the two-block Bregman ADMM and consensus settings. \Cref{sec:application} reports numerical experiments. The appendix contains the full two-block Bregman ADMM and consensus proofs.

\paragraph{Notation.}
We write bold lowercase letters (e.g., $\vx,\vy,\vz$) for vectors, bold uppercase letters (e.g., $\mA,\mB,\mH$) for matrices, and $\calT$ for tensors.
$\R^n$ denotes $n$-dimensional Euclidean space.
$\|\cdot\|_2$ is the Euclidean (vector) norm and $\|\cdot\|_F$ the Frobenius (matrix/tensor) norm.
$\langle\cdot,\cdot\rangle$ denotes the standard inner product.
$\Null(\mM)$ denotes the null space of $\mM$, and $\eye$ denotes an identity matrix of compatible dimension.
$\nabla f$ and $\nabla^2 f$ denote the gradient and Hessian of $f$; for functions of two variables we write $\nabla^2_{\vx\vx}f, \nabla^2_{\vy\vy}f$ for the diagonal Hessian blocks and $\nabla^2_{\vx\vy}f,\nabla^2_{\vy\vx}f$ for the cross Hessian blocks.
For a matrix $\mM$, $\lambda_i(\mM)$ denotes an eigenvalue indexed by $i$; when $\mM$ is real symmetric, $\lambda_{\min}(\mM)$ denotes its smallest eigenvalue. $\mM\succ 0$ (resp.\ $\mM\succeq 0$) means $\mM$ is positive definite (resp.\ positive semidefinite).
For matrices $\mM_1,\ldots,\mM_k$, $\mathrm{diag}(\mM_1,\ldots,\mM_k)$ denotes the block-diagonal matrix with diagonal blocks $\mM_1,\ldots,\mM_k$.
$Dg$ denotes the Jacobian matrix (total derivative) of a differentiable map $g$, and $g^k$ its $k$-fold composition; $\nabla_\vx f$ denotes the partial gradient with respect to $\vx$.
We write $[J]\doteq\{1,\ldots,J\}$ for the index set, and use $\doteq$ throughout to mean ``is defined as''.

\section{Preliminaries}
\label{sec:prelim}

\begin{definition}\label{def:critical:point}
Consider the following equality-constrained program with $\mA\in\R^{p\times n}$, $\vb\in\R^p$, and twice continuously differentiable $f$:

\[
\minimize_{\vx\in\R^n} f(\vx) \st \mA\vx=\vb.
\]
\begin{enumerate}[leftmargin=*]
\item $\vx$ is a (first-order) stationary point if $\mA\vx=\vb$ and $\nabla f(\vx)+\mA^{\top}\vlambda=\vzero$ for some multiplier $\vlambda\in\R^{p}$; 

\item $\vx$ is a second-order stationary point if it is a first-order stationary point and $\vd^{\top}\nabla^{2} f(\vx)\vd\ge 0$ for all feasible directions $\vd\in \Null(\mA)$;

\item $\vx$ is a strict saddle if it is a first-order stationary point but not a second-order stationary point, i.e., $\vd^{\top}\nabla^{2} f(\vx)\vd < 0$ for some $\vd\in \Null(\mA)$.
\end{enumerate}
\end{definition}

\begin{definition}[Multi-block stationarity]\label{rem:stacked-definition}
The notions in \Cref{def:critical:point} extend to any linearly constrained problem $\min_{\vz}F(\vz)$ subject to $\mC\vz=\vb$ with twice continuously differentiable $F$ by viewing all decision variables as a single block $\vz$ and replacing $\mA$ with $\mC$.
In particular, for Problem~\eqref{eqn:problem:2}, set $\vz=(\vx,\vy)$, $F(\vz)=f_1(\vx)+f_2(\vy)$, and $\mC=[\mA\ \mB]$.
\begin{enumerate}[leftmargin=*]
\item $(\vx^\star,\vy^\star)$ is a first order stationary point if and only if $\mA\vx^\star+\mB\vy^\star=\vb$, $\nabla f_1(\vx^\star)+\mA^\top\vlambda^\star=\vzero$, and $\nabla f_2(\vy^\star)+\mB^\top\vlambda^\star=\vzero$ for some $\vlambda^\star$.

\item $(\vx^\star,\vy^\star)$ is a second order stationary point if and only if it is a first order stationary point and $\vd_x^\top\nabla^2 f_1(\vx^\star)\vd_x+\vd_y^\top\nabla^2 f_2(\vy^\star)\vd_y\ge0$ for all $(\vd_x,\vd_y)$ satisfying $\mA\vd_x+\mB\vd_y=\vzero$.

\item $(\vx^\star,\vy^\star)$ is a strict saddle if and only if it is a first order stationary point and $\vd_x^\top\nabla^2 f_1(\vx^\star)\vd_x+\vd_y^\top\nabla^2 f_2(\vy^\star)\vd_y<0$ for some $(\vd_x,\vd_y)$ satisfying $\mA\vd_x+\mB\vd_y=\vzero$; cf.\ \eqref{KKT:1:badmm} and \eqref{KKT:2:badmm} in \Cref{pf:thm:badmm}.
\end{enumerate}
The same convention applies to the consensus formulation after stacking all local variables and writing the agreement constraints as a linear system.
We refer to any primal dual point satisfying the feasibility and first order conditions in the relevant stacked formulation as a KKT point.
\end{definition}

\paragraph{Bregman divergence and relative smoothness.}
We recall the Bregman divergence induced by a strongly convex kernel $h$ \cite{bregman1967relaxation,censor1981iterative}.

\begin{definition}
\label{def:breg}
The Bregman divergence induced by a differentiable, strongly convex function $h$ is defined as
\begin{align}
D_h(\vx,\vy)=h(\vx)- h(\vy) - \langle\nabla h(\vy),\vx-\vy\rangle
\label{def:bregman:distance}
\end{align}
\end{definition}

Geometrically, $D_h(\vx,\vy)$ is the gap between $h(\vx)$ and the first-order Taylor approximation of $h$ at $\vy$, so it measures discrepancy in the geometry induced by the kernel rather than by the Euclidean norm.
When $h(\vx) = \frac{1}{2}\|\vx\|_{2}^2$, the Bregman divergence reduces to the standard squared Euclidean distance $D_h(\vx,\vy) = \frac{1}{2}\|\vx - \vy\|_{2}^2$. More generally, Bregman divergences underpin mirror-descent and Bregman-proximal schemes: \cite{bauschke2016descent} introduces the descent lemma under relative smoothness, \cite{bolte2018first} extends the framework to nonconvex objectives, and \cite{li2019provable} constructs polynomial kernels ensuring relative smoothness for broad classes of problems. These ideas also lead to Bregman variants of ADMM \cite{wang2014bregman,wang2014convergence,wang2018convergence}. One key analytic advantage is that the smoothness of $f$ can be measured \emph{relative} to $h$.

\paragraph{Examples used later.}
We record several kernels and the corresponding Bregman divergences used below.
\begin{itemize}[leftmargin=*]
\item \textbf{Quadratic kernel (Euclidean).} If $h(\vx)=\frac{1}{2}\|\vx\|_2^2$, then $D_h(\vx,\vy)=\frac{1}{2}\|\vx-\vy\|_2^2$.

\item \textbf{Polynomial kernels used for relative smoothness.}
For the kernel in \Cref{lem:h:exist1}, where $d\ge 2$ is the polynomial degree,
\[
h(\vx)=\frac{1}{d}\|\vx\|_2^{d}+\frac{1}{2}\|\vx\|_2^2+1,\qquad
\nabla h(\vx)=\|\vx\|_2^{d-2}\vx+\vx,
\]
so $D_h(\vx,\vy)$ can be computed by substituting these expressions into \eqref{def:bregman:distance}.

\item \textbf{Bi-variable/product kernels and bi-Bregman divergence.}
For the product-type bi-convex kernel used in \Cref{lem:h:exist2}, $h(\vx,\vy)=h_x(\vx)h_y(\vy)$, the bi-Bregman divergences simplify to scalar multiples of the corresponding one-variable Bregman divergences:
\begin{align*}
D_h^x(\vx_1,\vx_2;\vy)&=h_y(\vy)\,D_{h_x}(\vx_1,\vx_2),\\
D_h^y(\vy_1,\vy_2;\vx)&=h_x(\vx)\,D_{h_y}(\vy_1,\vy_2).
\end{align*}
This structure is used in the multi-block extensions of \Cref{sec:mult}.
\end{itemize}

\begin{definition}
\label{defi:adaptive:lipschitz}
A $\mathcal{C}^2$ function $f(\vx)$ is relatively smooth with constant $L_f$ with respect to a $\mathcal{C}^2$ strongly convex function $h(\vx)$ if
\begin{align}
\label{eqn:general:lipschitz}
L_f \nabla^{2}h(\vx)\pm \nabla^{2}f(\vx)\succeq 0,\quad \forall \vx\in\R^n.
\end{align}
\end{definition}
Our convention in \eqref{eqn:general:lipschitz} is a \emph{two-sided} Hessian comparison: besides the standard upper bound $\nabla^2 f(\vx)\preceq L_f\nabla^2 h(\vx)$, we also assume the lower bound $\nabla^2 f(\vx)\succeq -L_f\nabla^2 h(\vx)$. The lower bound is what guarantees strong convexity of the Bregman-proximal subproblems used later in the paper.
When $h(\vx) = \frac{1}{2}\|\vx\|_{2}^2$, \eqref{eqn:general:lipschitz} reduces to the standard Lipschitz-smoothness condition $\|\nabla^{2}f(\vx)\|\le L_f$. Under suitable standard assumptions in addition to relative smoothness, Bregman gradient and proximal gradient methods admit descent and convergence guarantees, including convergence to first order stationary points in the nonconvex setting \cite{bauschke2016descent,bolte2018first}, thereby avoiding the need for a global Lipschitz constant in the gradient.
This framework is particularly useful for polynomial models, which are central to the applications below. In particular, any finite degree polynomial objective is relatively smooth under a suitable polynomial kernel $h$ \cite{li2019provable}; the corresponding two variable construction yields relative bi-smoothness, as stated next.
\begin{proposition}[\cite{li2019provable}]\label{lem:h:exist1}
We call $f(\vx)$ a $d$th-degree polynomial if the maximum total degree among its monomials of the form $x_1^{a_1}\cdots x_n^{a_n}$ is $d$.
Let $f(\vx)$ be a $d$th-degree polynomial with $d\ge 2$ and $h(\vx)=
\frac{1}{d}\|\vx\|_{2}^{d}+\frac{1}{2}\|\vx\|_{2}^{2}+1$. Then $f$ is $L_f$-relatively smooth to $h$ for some $L_f>0$ depending on the coefficients of $f$.
\end{proposition}
\begin{remark}[Two-sided condition]\label{rem:twosided}
The relative smoothness convention in \cite{li2019provable} is already the two sided Hessian comparison $L_f\nabla^2h(\vx)\pm\nabla^2 f(\vx)\succeq0$. Thus \Cref{lem:h:exist1} is a specialization of the polynomial kernel construction in \cite{li2019provable}. If one starts instead from the standard one sided upper bound convention, the lower bound can also be obtained by applying the same polynomial kernel argument to $-f$ and taking $L_f\gets\max(L_f,L_{-f})$. The same reasoning applies to the bi variable case (\Cref{lem:h:exist2}).
\end{remark}

Since \eqref{eqn:problem:2} involves two primal blocks, we also use the following bi-smoothness conditions.
\begin{definition}
A twice continuously differentiable kernel $h(\vx,\vy)$ is called \emph{strongly bi-convex} on its domain if there exists $m_h>0$ such that
\begin{align*}
\nabla^2_{\vx\vx}h(\vx,\vy)&\succeq m_h\eye,\\
\nabla^2_{\vy\vy}h(\vx,\vy)&\succeq m_h\eye, ~ \forall  \vx,\vy.
\end{align*}
\end{definition}
\begin{definition} The bi-Bregman divergence induced by a strongly bi-convex function $h$ is defined as
\begin{align}
D_h^x(\vx_1,\vx_2;\vy)=&h(\vx_1,\vy)- h(\vx_2,\vy) - \langle\nabla_{\vx} h(\vx_2,\vy),\vx_1-\vx_2\rangle,
\\
D_h^y(\vy_1,\vy_2;\vx)=&h(\vx,\vy_1)- h(\vx,\vy_2) - \langle\nabla_{\vy} h(\vx,\vy_2),\vy_1-\vy_2\rangle.
\end{align}
When $h(\vx,\vy)=(\|\vx\|_2^2+\|\vy\|_2^2)/2$, we have $D_h^x(\vx_1,\vx_2;\vy)=\|\vx_1-\vx_2\|_2^2/2$ and  $D_h^y(\vy_1,\vy_2;\vx)=\|\vy_1-\vy_2\|_2^2/2$, respectively.
\label{def:bi:bregman:distance}
\end{definition}

\begin{definition}
\label{defi:adaptive:lipschitz:coordinate}
A twice continuously differentiable function $f(\vx,\vy)$ is relatively bi-smooth with constants $(L_f^x,L_f^y)$ with respect to a strongly bi-convex function $h(\vx,\vy)$ if
\begin{equation}
\begin{split}
\label{eqn:general:lipschitz:coordinate}
L_f^x \nabla^{2}_{\vx\vx}h(\vx,\vy) \pm \nabla_{\vx\vx}^{2}f(\vx,\vy)&\succeq 0,\\
L_f^y \nabla^{2}_{\vy\vy}h(\vx,\vy) \pm \nabla_{\vy\vy}^{2}f(\vx,\vy) &\succeq 0, ~ \forall  \vx,\vy.
\end{split}
\end{equation}
\end{definition}
As in the one-variable case, \eqref{eqn:general:lipschitz:coordinate} is a two-sided assumption: it controls each diagonal Hessian block of $f$ both above and below by the corresponding diagonal Hessian block of $h$. Throughout the paper, ``relative smoothness'' and ``relative bi-smoothness'' refer to these two-sided conditions unless stated otherwise.

\begin{proposition}[\cite{li2019provable}]\label{lem:h:exist2}
We call $f(\vx,\vy)$ a $(d_1,d_2)$th-degree polynomial if it is a $d_1$th-degree polynomial in $\vx$ and a $d_2$th-degree polynomial in $\vy$.
Let $f(\vx,\vy)$ be a $(d_1,d_2)$th-degree polynomial with $d_1,d_2\ge 2$ and $h(\vx,\vy)=
(\frac{1}{d_1}\|\vx\|_{2}^{d_1}+\frac{1}{2}\|\vx\|_{2}^{2}+1)(\frac{1}{d_2}\|\vy\|_{2}^{d_2}+\frac{1}{2}\|\vy\|_{2}^{2}
+1)$. Then $f$ is $(L_f^x,L_f^y)$-relatively bi-smooth with respect to $h$ for some $L_f^x,L_f^y>0$.
\end{proposition}

\section{Bregman ADMM}
\label{sec:bregadmm}
We consider Bregman ADMM for the two-block linearly constrained problem~\eqref{eqn:problem:2}. Under the one variable relative smoothness condition of \Cref{defi:adaptive:lipschitz}, applied separately to each block, together with the kernel strong convexity, stepsize, and invariant domain assumptions stated in \Cref{thm:badmm}, the probability that the iterates converge to a strict saddle is zero. Combined with first order convergence guarantees (e.g., \cite[Theorem III.8]{wang2014convergence}), this yields almost sure convergence to second order stationary points.

Bregman ADMM alternates between two Bregman-proximal primal updates and a dual ascent step. The divergences $D_{h_1}$ and $D_{h_2}$ act as geometry-adapted regularizers that stabilize each block subproblem and allow the analysis to rest on the relative smoothness condition of \Cref{defi:adaptive:lipschitz} rather than global Lipschitz continuity of $\nabla f$. When $h_1(\vx)=\tfrac{1}{2}\|\vx\|_2^2$ and $h_2(\vy)=\tfrac{1}{2}\|\vy\|_2^2$, Bregman ADMM reduces to proximal ADMM.

The algorithm is stated in \Cref{alg:badmm}.
\begin{algorithm}[H]
\caption{Bregman ADMM}
\label{alg:badmm}
\begin{algorithmic}[1]
\STATE {\bf Input:}  Bregman divergence kernels $h_1,h_2$  and parameters $\eta,\rho$.
\STATE {\bf Initialization:}  $(\vx^{0},\vy^{0},\vlambda^0)$
\STATE {\bf For} $k=1,2, \dots$, recursively generate $(\vx^{k},\vy^{k},\vlambda^k)$ by
\begin{equation}
\begin{aligned}
\vx^{k}=& \argmin_{\vx} \calL(\vx,\vy^{k-1},\vlambda^{k-1})+\frac{1}{\eta}D_{h_{1}}(\vx,\vx^{k-1}),
\\
\vy^{k}=& \argmin_{\vy} \calL(\vx^{k},\vy,\vlambda^{k-1})+\frac{1}{\eta}D_{h_{2}}(\vy,\vy^{k-1}),
\\
\vlambda^{k}=& \vlambda^{k-1}+\rho(\mA\vx^{k}+\mB\vy^{k}-\vb).
\end{aligned}
\label{eqn:badmm}
\end{equation}
\end{algorithmic}
\end{algorithm}

First-order convergence of Bregman ADMM has been analyzed for the two-block scheme in \cite[Theorem III.8]{wang2014convergence} and for the multi-block scheme in \cite[Theorem 3.9]{wang2018convergence}. These works use Bregman distance regularization, with constraints written as $\mA\vx=\mB\vy$, $\mA\vx+\mB\vy=\vzero$, or more generally $\sum_i\mA_i\vx_i=\vzero$. Up to block order, sign, and affine shift conventions, our updates coincide with these Bregman ADMM forms by using the scaled Bregman regularizers $\tfrac{1}{\eta}D_{h_1}$ and $\tfrac{1}{\eta}D_{h_2}$ and replacing the homogeneous constraint by $\mA\vx+\mB\vy=\vb$. Under \Cref{ass:domain} below, we complement these results by showing that the iteration defines a smooth fixed-point map $g$ whose strict saddles are unstable fixed points; consequently, the set of initializations that converge to strict saddles has Lebesgue measure zero.

Here and throughout, \emph{random initialization} means that the initial algorithmic state, for instance $(\vx^0,\vy^0,\vlambda^0)$ for \Cref{alg:badmm}, is drawn from a distribution on an open state space domain $\Omega$ that is absolutely continuous with respect to Lebesgue measure.

\begin{assumption}[Open-domain self-map condition]\label{ass:domain}
For the algorithm under consideration, the objective and Bregman kernels are $\mathcal{C}^2$ on an open primal domain, and the induced fixed-point map $g$ is well defined on an open state-space domain $\Omega$ satisfying $g(\Omega)\subseteq\Omega$.
\end{assumption}

\begin{remark}[Sufficient conditions for \Cref{ass:domain}]\label{rem:domainX}
\Cref{ass:domain} is the technical condition needed for the implicit function theorem and the dynamical systems argument used below to apply on the whole state space. Two common sufficient scenarios are:
\begin{itemize}[leftmargin=*]
\item \emph{Global full domain case.} If $f_j,h_j\in\mathcal{C}^2$ on the whole primal space and the Bregman proximal subproblems are globally well defined, for instance under the strong convexity, two sided relative smoothness, and stepsize assumptions of \Cref{thm:badmm}, then one may take $\Omega$ to be the full primal dual state space; the self map and domain invariance parts of \Cref{ass:domain} are automatic.
\item \emph{Invariant open domains.} In problems with natural open primal domains (e.g., when the kernels are Legendre-type so the primal variables stay in the interior of $\mathrm{dom}\,h_j$), the Bregman-proximal subproblems typically return minimizers in the same interior; taking $\Omega$ to be the corresponding product with the unrestricted dual space then yields $g(\Omega)\subseteq\Omega$.
\end{itemize}
If one also wants to conclude almost sure convergence to second order stationary points from first order convergence, then one sufficient route is to assume boundedness of the generated sequence, possibly obtained from coercivity in applications, together with K\L{} and subanalyticity type assumptions and the other regularity assumptions required by first order Bregman ADMM convergence results such as \cite[Theorem III.8]{wang2014convergence}.
\end{remark}

\begin{theorem}\label{thm:badmm}
Assume \Cref{ass:domain} holds for the fixed-point map induced by \Cref{alg:badmm}, and assume also that each kernel $h_j$ is strongly convex. Let $f_j$ be $L_j$-relatively smooth w.r.t.\ $h_j$ in the sense of \Cref{defi:adaptive:lipschitz} (applied to the $j$-th block), $j=1,2$, and set $\eta<\min_j1/L_j$. With random initialization, the probability that the iterates of \Cref{alg:badmm} converge to a strict saddle is zero.
\end{theorem}

We stress that \Cref{thm:badmm} does not assert convergence of the iterates; it states that the set of initializations leading to convergence to a strict saddle has Lebesgue measure zero.
Convergence itself requires separate conditions; once convergence to a KKT point, i.e., a primal-dual point satisfying the feasibility and first-order conditions in \Cref{rem:stacked-definition}, is established, the theorem excludes strict saddles from the set of possible limits.

The proof of \Cref{thm:badmm} is given in \Cref{pf:thm:badmm}. We emphasize that $\mathcal{C}^2$ regularity alone does \emph{not} guarantee either invariance of the state-space domain or boundedness of the iterates; \Cref{rem:domainX} gives representative sufficient conditions.
Under additional assumptions guaranteeing first order convergence, for example assumptions analogous to \cite[Theorem III.8]{wang2014convergence}, which include boundedness of the generated sequence and the regularity and K\L{} assumptions needed for convergence to a stationary point of the augmented Lagrangian, \Cref{thm:badmm} implies that the limiting KKT points are almost surely second order stationary.
By \Cref{lem:h:exist1}, the strict saddle avoidance conclusion of \Cref{thm:badmm} applies to finite degree polynomial objectives $f_1$ and $f_2$ with corresponding polynomial kernels. If, in addition, a first order convergence guarantee is available for the resulting Bregman ADMM iterates, then any limiting KKT point is almost surely second order stationary. For the polynomial kernel $h(\vx)=\tfrac{1}{d}\|\vx\|_2^d+\tfrac{1}{2}\|\vx\|_2^2+1$, the global full domain scenario in \Cref{rem:domainX} applies because the objectives and kernels are $\mathcal{C}^2$ and the subproblems are globally strongly convex under the stepsize condition; see \Cref{rem:global:sc}.

\begin{remark}[Global strong convexity of Bregman-proximal subproblems]\label{rem:global:sc}
For the polynomial kernel $h(\vx)=\tfrac{1}{d}\|\vx\|_2^d+\tfrac{1}{2}\|\vx\|_2^2+1$ of \Cref{lem:h:exist1}, we have $\nabla^2 h(\vx)\succeq \eye_n$ for all $\vx\in\R^n$ (since the $\|\vx\|_2^d$ term contributes a positive semidefinite Hessian and the quadratic term contributes $\eye_n$). By two-sided relative smoothness, $\nabla^2 f(\vx)\succeq -L_f\nabla^2 h(\vx)$ for all $\vx$, so the Bregman-proximal Hessian satisfies
\[
\nabla^2 f(\vx)+\frac{1}{\eta}\nabla^2 h(\vx)\succeq\Bigl(\frac{1}{\eta}-L_f\Bigr)\nabla^2 h(\vx)\succeq\Bigl(\frac{1}{\eta}-L_f\Bigr)\eye_n\succ 0
\]
\emph{uniformly} over $\R^n$ whenever $\eta<1/L_f$. Thus each subproblem is globally strongly convex with a unique minimizer, and the fixed-point map $g$ is well defined on the entire primal-dual space.
\end{remark}
For non-polynomial objectives satisfying relative smoothness, \Cref{ass:domain} must be verified on a case-by-case basis; the Legendre-type scenario in \Cref{rem:domainX} provides a useful template.
Since \Cref{alg:badmm} reduces to proximal ADMM when $h_1(\vx)=\tfrac{1}{2}\|\vx\|_2^2$ and $h_2(\vy)=\tfrac{1}{2}\|\vy\|_2^2$, \Cref{thm:badmm} covers proximal ADMM as a special case.

\begin{remark}[On single-variable vs.\ bi-variable relative smoothness]
Since the objective in \eqref{eqn:problem:2} is separable, $f(\vx,\vy)=f_1(\vx)+f_2(\vy)$, the cross-Hessian $\nabla^2_{\vx\vy}f$ vanishes and relative smoothness of each $f_j$ in its own variable (\Cref{defi:adaptive:lipschitz}) suffices for \Cref{thm:badmm}. For the consensus formulation in \Cref{sec:mult}, each local objective $f_j(\vx_j,\vy_j)$ may couple $\vx_j$ and $\vy_j$, so the stronger bi-smoothness condition (\Cref{defi:adaptive:lipschitz:coordinate}) is used in \Cref{thm:cadmm}.

\end{remark}

\section{Extension to Multi-block Case}
\label{sec:mult}

The two-block analysis already contains the essential mechanism. We now show how the same approach extends to a global consensus model, which covers the distributed settings most relevant for our applications.

\subsection{Global consensus problem}

Consider the following unconstrained centralized problem
\begin{equation}
\minimize_{\vx,\{\vy_j\}_{j=1}^J}  \sum_{j=1}^J f_j(\vx,\vy_j),
\label{eq:origcent}
\end{equation}
where $f_j$ are possibly nonconvex functions, which do not necessarily have a Lipschitz continuous gradient.
One can distribute this problem across a network of $J$ nodes in a ``star topology'', where $J-1$ agents are connected to and can communicate with a central node (which we denote by $1$), but do not communicate directly among themselves~\cite{hong2018gradient}. This results in the following {\em global consensus problem}
\begin{align}
\minimize_{\{\vx_{j},\vy_{j}\}_{j=1}^J} \sum_{j=1}^{J}f_{j}(\vx_{j},\vy_{j})  \st
\vx_{j}=\vx_{1},~~\forall j\ge 2.
\label{eqn:problemC}
\end{align}
Problem \eqref{eqn:problemC} is a generalized consensus formulation: only the shared variables $\vx_j$ are constrained to agree, while the local variables $\vy_j$ remain agent-specific.
The constraints $\vx_j=\vx_1$ encode a star-topology (global consensus) model in which agent $1$ plays the role of a hub and all other agents enforce agreement with $\vx_1$.
More general consensus topologies (e.g., edge-wise constraints $\vx_i=\vx_j$ over a connected graph) can be written as linear consensus constraints and reformulated into a global consensus form by introducing an auxiliary global variable, but extending the proof to a specific reformulation requires checking the induced fixed-point-map structure in addition to strong convexity of the resulting Bregman-proximal subproblems. For clarity, we focus on the star/global consensus case. 

\subsection{Landscape of global consensus problem}
\Cref{thm:cons} below connects the geometric landscapes of the unconstrained centralized problem \eqref{eq:origcent} and its equality-constrained distributed counterpart~\eqref{eqn:problemC}.

\begin{theorem}[Corollary 1 in \cite{li2019geometry}]
\label{thm:cons}
$[\vx, ~ \vy_1,~\cdots,~\vy_J]$ is a first-order/second-order stationary point of problem~\eqref{eq:origcent} iff $[\vx, ~\vx,~\cdots,~\vx,~ \vy_1,~\cdots,~\vy_J]$ is a first-order/second-order stationary point of problem~\eqref{eqn:problemC}. Moreover, if problem~\eqref{eq:origcent} satisfies the strict saddle property (every first order stationary point is either a local minimizer or a strict saddle) and has no spurious local minima, then for every second-order stationary point $[\vx, ~\vx,~\cdots,~\vx,~ \vy_1,~\cdots,~\vy_J]$  of~\eqref{eqn:problemC}, $[\vx, ~ \vy_1,~\cdots,~\vy_J]$ is a global minimizer of~\eqref{eq:origcent}.
\end{theorem}

\Cref{thm:cons} links the landscapes of \eqref{eq:origcent} and \eqref{eqn:problemC}. In particular, obtaining a second-order stationary point of the distributed consensus problem \eqref{eqn:problemC} yields a second-order stationary point of the centralized objective \eqref{eq:origcent}; in many structured matrix/tensor models, such points are globally optimal under a strict-saddle or no-spurious-minima landscape \cite{ge2016matrix,ge2015escaping}. Most available algorithms for global consensus problems, however, provide guarantees only for first-order stationarity. Under a Lipschitz gradient assumption, \cite{hong2018gradient} proves that linearized ADMM avoids strict saddles almost surely from random initialization; this assumption can fail for polynomial objectives arising in matrix/tensor factorizations, including distributed matrix factorization \cite{zhu2019distributed}.

\subsection{Convergence of Consensus Bregman ADMM}
To address these issues, we extend Bregman ADMM (see \Cref{alg:badmm}) to the multi-block case. We write the augmented Lagrangian for \eqref{eqn:problemC} as
\begin{equation}\begin{split}
&\calL(\vx_{1}, \vy_{1},\cdots,\vx_{J},\vy_{J}, \vlambda_{2},\cdots,\vlambda_{J})=
\sum_{j=1}^{J}f_{j}(\vx_{j},\vy_{j})
+ \sum_{j=2}^{J} \lg \vlambda_j, \vx_{j}-\vx_{1}\rg
+\frac{\rho}{2}\sum_{j=2}^{J} \|\vx_{j}-\vx_{1}\|_{2}^{2}.
\end{split}
\label{eqn:Lagrange}
\end{equation}
Applying Bregman-proximal updates to the augmented Lagrangian \eqref{eqn:Lagrange} yields \Cref{alg:cadmm} below. The agents are updated sequentially (Gauss-Seidel ordering): for $j\ge 2$, the $\vx_j$-update uses the already-computed $\vx_1^k$. More precisely, the ellipsis in $\calL(\vx_1^k,\vy_1^k,\ldots,\vx_j,\vy_j^{k-1},\ldots)$ means that all blocks with indices $\ell<j$ take their already-updated iteration-$k$ values, while blocks with $\ell>j$ retain their iteration-$(k{-}1)$ values; in particular, for $j=1$ no preceding blocks exist, so the Lagrangian is evaluated at $(\vx_1,\vy_1^{k-1},\vx_2^{k-1},\ldots)$. This ordering is also reflected in the proof in \Cref{sec:thm:cadmm}, whose Jacobian factorization follows the same update sequence.

\begin{algorithm}[ht]
\caption{Consensus Bregman ADMM}
\label{alg:cadmm}
\begin{algorithmic}[1]
\STATE {\bf Input:}  Bregman divergence kernels $h_j,j\in[J]$  and parameters $\eta,\rho$.

\STATE {\bf Initialization:}  $(\vx_{1}^{0},\vy_1^0,\cdots,\vx_{J}^{0},\vy_{J}^{0},\vlambda_{2}^{0},\cdots,\vlambda_{J}^{0})$

\STATE {\bf For} $k=1,2,\ldots$
\STATE {\bf For} $j=1,2,\ldots,J$
{
\begin{equation}
\begin{aligned}
     \vx_{j}^{k}&=\argmin_{\vx_{j}} \frac{1}{\eta}D^{x}_{h_{j}}(\vx_{j},\vx_{j}^{k-1};\vy_j^{k-1})+\calL(\vx_{1}^{k},\vy_{1}^{k},\cdots,\vx_{j},\vy_{j}^{k-1},\cdots,\vx_{J}^{k-1},\vy_{J}^{k-1},\vlambda_{2}^{k-1},\cdots,\vlambda_{J}^{k-1}),
\\
       \vy_{j}^{k}&=\argmin_{\vy_{j}} \frac{1}{\eta}D^{y}_{h_{j}}(\vy_{j},\vy_{j}^{k-1};\vx_j^{k})+\calL(\vx_{1}^{k},\vy_{1}^{k},\cdots,\vx_{j}^{k},\vy_{j},\cdots,\vx_{J}^{k-1},\vy_{J}^{k-1},\vlambda_{2}^{k-1},\cdots,\vlambda_{J}^{k-1}).
\end{aligned}
\label{eqn:algorithm}
\end{equation}
}
\STATE{\bf End For}
\STATE  $\begin{bmatrix}\vlambda_{2}^{k} & \cdots & \vlambda_{J}^{k}   \end{bmatrix}=
\begin{bmatrix} \vlambda_{2}^{k-1}& \cdots & \vlambda_{J}^{k-1}   \end{bmatrix} +\rho
\begin{bmatrix}\vx_{2}^k-\vx_{1}^k & \cdots & \vx_{J}^k-\vx_{1}^k  \end{bmatrix} $.
\STATE{\bf End For}
\end{algorithmic}
\end{algorithm}

\begin{theorem}
\label{thm:cadmm}
Assume \Cref{ass:domain} holds for the fixed-point map induced by \Cref{alg:cadmm}, and assume also that each kernel $h_j$ is strongly bi-convex. Let $f_j$ be $(L_j^x,L_j^y)$-relatively bi-smooth w.r.t.\ $h_j$, $j\in[J]$, and set $\eta< \min_{j\in[J]}\min(1/L_j^x,1/L_j^y)$. With random initialization, the probability that the iterates of \Cref{alg:cadmm} converge to a strict saddle is zero.
\end{theorem}

The proof of \Cref{thm:cadmm} is given in  \Cref{sec:thm:cadmm}. As with \Cref{thm:badmm}, the theorem does not assert convergence but rather excludes strict saddles from the set of possible limits; convergence to a KKT point must be established separately.
Combined with \Cref{thm:cons} and additional assumptions guaranteeing first order convergence, for example assumptions analogous to \cite[Theorem 3.9]{wang2018convergence}, which include boundedness or coercivity, regularity, and K\L{} and subanalyticity type conditions for convergence to a stationary point of the augmented Lagrangian, \Cref{thm:cadmm} implies that any limiting KKT point of \Cref{alg:cadmm} corresponds almost surely to a second-order stationary point of the centralized problem \eqref{eq:origcent}. Numerical illustrations are given in \Cref{sec:application}.

\section{\texorpdfstring{Strict Saddle Avoidance via Fixed Point Instability}{Strict Saddle Avoidance via Fixed Point Instability}}
\label{sec:analysis}
This section provides the high-level ideas underlying the proofs of \Cref{thm:badmm} and \Cref{thm:cadmm}; the detailed proofs are presented in \Cref{pf:thm:badmm} and \Cref{sec:thm:cadmm}, respectively. Specifically, \Cref{sec:sketch:twoblock} and \Cref{sec:sketch:consensus} develop the corresponding proof sketches.
To this end, we first present a fully expanded proof of the single block Bregman augmented Lagrangian method (Bregman ALM) prototype in \Cref{thm:bmm}, whose role is to make the common fixed point template transparent. 
%The main two block and consensus proof ideas are then summarized in \Cref{sec:sketch:twoblock,sec:sketch:consensus}.
The argument highlights four ingredients used again in the two block and consensus analyses: construction of a smooth fixed point map, nonsingularity of its Jacobian, identification of KKT points with fixed points, and instability of strict saddle fixed points. 
 
We begin with the following single block equality constrained problem:
\begin{align}
\minimize_{\vx\in\R^{n}} f(\vx)\ \st \mA\vx=\vb
\label{eqn:problem:1}
\end{align}
where $\mA\in\R^{p\times n}$ and $f$ is a nonconvex function.
The augmented Lagrangian form of \eqref{eqn:problem:1} is
\begin{align}
\calL(\vx,\vlambda)=f(\vx)+\lg \vlambda, \mA\vx-\vb\rg +\frac{\rho}{2}\|\mA\vx-\vb\|_{2}^{2},
\label{eqn:Lagrange:1}
\end{align}
where $\vlambda\in\R^{p}$ is the Lagrangian dual variable and $\rho>0$ is the augmented penalty parameter. In the single-block case, Bregman ADMM reduces to the Bregman augmented Lagrangian method (Bregman ALM; see \Cref{alg:bmm}).

\begin{algorithm}[H]
\caption{Bregman Augmented Lagrangian Method}
\label{alg:bmm}
\begin{algorithmic}[1]
\STATE {\bf Input:}  Bregman divergence kernel $h$  and parameters $\eta,\rho$.

\STATE {\bf Initialization:} $(\vx^{0},\vlambda^0)$

\STATE {\bf For} $k=1,2,\ldots$
\begin{equation}
\begin{aligned}
\vx^{k}=& \argmin_{\vx} \calL(\vx,\vlambda^{k-1})+\frac{1}{\eta}D_{h}(\vx,\vx^{k-1}),
\\
\vlambda^{k}=& \vlambda^{k-1}+\rho(\mA\vx^{k}-\vb).
\end{aligned}
\label{eqn:bmm}
\end{equation}
\end{algorithmic}
\end{algorithm}

Our analysis follows the dynamical systems viewpoint of Lee et al.~\cite{lee2017first}: an iterative method defines a smooth fixed point map on the algorithmic state. Once strict saddles are shown to be unstable fixed points of this map, local invertibility of the fixed point map allows the stable manifold argument of \cite{lee2017first} to imply that the set of initializations converging to such points has Lebesgue measure zero; see also the stable manifold theorem \cite{shub2013global}.

\begin{definition}\label{def:unstable:fixed:point}
Let $g$ be a $\mathcal{C}^{1}$ mapping from $\Omega$ to $\Omega$. Its set of unstable fixed points is
\begin{align}
\calA_{g}=\{\vz\in\Omega:\ g(\vz)=\vz,\ \max_i |\lambda_i(Dg(\vz))|>1\}.
\end{align}
\end{definition}

\begin{theorem}[\cite{lee2017first}]\label{thm:jason}
Let $g$ be a $\mathcal{C}^{1}$ mapping from $\Omega$ to $\Omega$ and suppose $\det(Dg(\vz))\neq 0$ for all $\vz\in\Omega$ (equivalently, $g$ is a local diffeomorphism on $\Omega$). Then the set of initial points that converge to unstable fixed points has zero Lebesgue measure:
\[
\mathrm{Leb}\Big(\big\{\vz^{0}\in\Omega:\ \lim_{k\to\infty} g^{k}(\vz^{0})\in\calA_{g}\big\}\Big)=0.
\]
\end{theorem}

Under the random initialization convention stated in \Cref{sec:bregadmm}, \Cref{thm:jason} shows that the set of initial points that converge to an unstable fixed point has probability zero. In our applications, we verify that strict saddles correspond to such unstable fixed points, so strict-saddle avoidance follows.

\subsection{\texorpdfstring{Prototype: strict saddle avoidance for Bregman ALM}{Prototype: strict saddle avoidance for Bregman ALM}}

\begin{theorem}\label{thm:bmm}
Assume \Cref{ass:domain} holds for the fixed-point map induced by \Cref{alg:bmm}, and assume also that the kernel $h$ is strongly convex. Let $f$ be $L_f$-relatively smooth w.r.t.\ $h$, and set $\eta<1/L_f$. With random initialization, the probability that the iterates of \Cref{alg:bmm} converge to a strict saddle is zero.
\end{theorem}

Note that \Cref{thm:bmm} (and likewise Theorems~\ref{thm:badmm} and~\ref{thm:cadmm}) does not assert that the iterates converge; it states that the set of initializations leading to convergence to a strict saddle has measure zero. Convergence itself requires separate conditions (e.g., the K\L{}/boundedness assumptions cited earlier); once convergence to a KKT point is established, the theorem excludes strict saddles from the set of possible limits.

We apply \Cref{thm:jason} to the fixed-point map induced by \Cref{alg:bmm}. The proof proceeds in four steps:
\begin{enumerate}[leftmargin=*]
\item Construct the fixed-point map $g$ (so that iterates satisfy $\vz^{k}=g(\vz^{k-1})$).
\item Compute the Jacobian $Dg$ (typically by differentiating optimality conditions via the implicit function theorem).
\item Show $\det(Dg)\neq 0$ on the domain, i.e., the fixed-point map is locally invertible everywhere.
\item Show that any strict saddle corresponds to an unstable fixed point of $g$.
\end{enumerate}

\subsubsection{Constructing the fixed-point map $g$}
To simplify notation, we rewrite \eqref{eqn:bmm} as

\begin{align*}
(\vx^k,\vlambda^{k-1})&=g_{1}(\vx^{k-1},\vlambda^{k-1}),\\
(\vx^k,\vlambda^{k})&=g_{2}(\vx^k,\vlambda^{k-1}),
\end{align*}
which implies that
\begin{equation}
\begin{aligned}
(\vx^{k},\vlambda^{k})&=(g_{2}\circ g_{1})(\vx^{k-1},\vlambda^{k-1})\doteq g(\vx^{k-1},\vlambda^{k-1}).
\label{eqn:g:bmm}
\end{aligned}
\end{equation}

\paragraph{Fixed points and KKT points.}
For \eqref{eqn:problem:1}, the algorithmic state is $\vz=(\vx,\vlambda)$. A point $(\vx^\star,\vlambda^\star)\in\Omega$ is a fixed point of the fixed-point map $g$ if and only if it satisfies the KKT conditions
\[
\mA\vx^\star=\vb,\qquad \nabla f(\vx^\star)+\mA^\top\vlambda^\star=\vzero.
\]
At a fixed point, the $g_1$ update returns $\vx^+=\vx$, where $\vx^+$ denotes the updated primal variable, so the Bregman gradient term $\nabla h(\vx^+)-\nabla h(\vx)$ vanishes; the $\vx$ update optimality condition and the dual update equation then reduce to the KKT system. Conversely, if $(\vx^\star,\vlambda^\star)$ satisfies the KKT conditions, then it satisfies the optimality conditions of \eqref{eqn:bmm}; since the $\vx$-subproblem is strongly convex on the primal domain, its minimizer is unique and therefore $\vx^+=\vx^\star$, i.e., $(\vx^\star,\vlambda^\star)$ is a fixed point.

\subsubsection{Computing the Jacobian matrix $Dg$}
Using the chain rule, we have
$
Dg=Dg_{2}Dg_{1}.
$
We first compute $Dg_{2}$. For notational convenience, denote $(\vx,\vlambda^{+})=g_{2}(\vx,\vlambda)$. Since $g_2$ is determined by
\begin{align}
\vlambda^{+}= \vlambda+\rho(\mA\vx-\vb),
\label{eqn:opt:block2:bmm}
\end{align}
we have
\begin{align}
Dg_{2}(\vx,\vlambda)=
\begin{bmatrix}
\eye_{n} &  \\
\nabla_{\vx}\vlambda^+& \nabla_{\vlambda}\vlambda^+
\end{bmatrix}
=
\begin{bmatrix}
\eye_{n} &\\
\rho \mA &  \eye_{p}
\end{bmatrix}.
\label{Dg2:bmm}
\end{align}
We now compute $Dg_{1}$.
To simplify notation, denote $(\vx^{+},\vlambda)=g_{1}(\vx,\vlambda)$. Then
\begin{align}
Dg_{1}(\vx,\vlambda)=
\begin{bmatrix}
\nabla_{\vx} \vx^{+}&  \nabla_{\vlambda}\vx^{+}\\
&\eye_{p}
\end{bmatrix}
\label{eqn:MM:Dg1}
\end{align}
Relative smoothness implies the positive-semidefinite inequality (Definition~\ref{defi:adaptive:lipschitz}) 
\[
L_f\nabla^2 h(\vx)\pm \nabla^2 f(\vx)\succeq 0;\] 
in particular,
\[
\nabla^2 f(\vx^{+})\succeq -L_f\nabla^2 h(\vx^{+}).
\]
Since $h$ is strongly convex, $\nabla^2 h(\vx^{+})\succ 0$. Hence for $\eta<1/L_f$,
\[
\nabla^{2}f(\vx^{+})+\rho\mA^{\top}\mA+\frac{1}{\eta}\nabla^{2} h(\vx^{+})
\succeq \Bigl(\frac{1}{\eta}-L_f\Bigr)\nabla^{2} h(\vx^{+}) \succ 0.
\]
By \Cref{ass:domain}, the $\vx$ update is well defined. The positive definiteness above makes the $\vx$ subproblem strongly convex, hence its minimizer is unique, and the implicit function theorem applies to the optimality condition.

We use the following standard form of the implicit function theorem \cite[Theorem~1.3.1]{krantz2003introduction}: if $F(u,v)=0$ and $\partial_v F(u_0,v_0)$ is nonsingular, then locally $v$ is a $\mathcal{C}^1$ function of $u$, and its derivative is given by \eqref{eq:ift-derivative}.
\begin{equation}
\label{eq:ift-derivative}
Dv(u)=-\bigl(\partial_v F(u,v(u))\bigr)^{-1}\partial_u F(u,v(u)).
\end{equation}
The optimality condition of the first block of \eqref{eqn:bmm} is
\begin{equation}
\begin{aligned}
&\nabla f(\vx^{+})+\mA^{\top}\vlambda+\rho\mA^{\top}(\mA\vx^{+}-\vb)+\frac{\nabla h(\vx^{+})-\nabla h(\vx)}{\eta}=\vzero.
\end{aligned} 
\label{eqn:opt:block1:bmm}
\end{equation}
We take $u=(\vx,\vlambda)$ and $v=\vx^{+}$, with $F(u,v)$ equal to the left-hand side of \eqref{eqn:opt:block1:bmm}. The Jacobian $\partial_v F$ equals $\nabla^{2}f(\vx^{+})+\rho\mA^{\top}\mA+\frac{1}{\eta}\nabla^{2} h(\vx^{+})$, which is nonsingular by the inequality above.
Applying the implicit function theorem to \eqref{eqn:opt:block1:bmm},
\begin{equation}
\begin{aligned}
& \nabla_{\vx}\vx^{+}=\left(\nabla^{2}f(\vx^{+})+\rho\mA^{\top}\mA+\frac{1}{\eta}\nabla^{2} h(\vx^{+})\right)^{-1} \left(\frac{1}{\eta}\nabla^{2} h(\vx)\right)
\\
&  \nabla_{\vlambda} \vx^{+}=-\left(\nabla^{2}f(\vx^{+})+\rho\mA^{\top}\mA+\frac{1}{\eta}\nabla^{2} h(\vx^{+})\right)^{-1}\mA^{\top}.
\end{aligned}
\label{eqn:MM:Dg1:1}
\end{equation}

Therefore, plugging \eqref{eqn:MM:Dg1:1} into \eqref{eqn:MM:Dg1}, we get
\begin{equation}\begin{split}
	Dg_{1}(\vx,\vlambda)
&=\begin{bmatrix}
\left(\nabla^{2}f(\vx^{+})+\rho\mA^{\top}\mA+\frac{1}{\eta}\nabla^{2} h(\vx^{+})\right)^{-1}
&
\\
& \eye_{p}
\end{bmatrix}
\begin{bmatrix}
\frac{1}{\eta}\nabla^{2} h(\vx)
&
-\mA^{\top}
\\
& \eye_{p}
\end{bmatrix}
\end{split}\label{Dg1:bmm}
\end{equation}
Finally, by the chain rule, we get that $Dg(\vx,\vlambda)=Dg_{2}(g_{1}(\vx,\vlambda))Dg_{1}(\vx,\vlambda)$, which further implies
\begin{equation}\begin{split}
&Dg(\vx,\vlambda)
=
\begin{bmatrix}
\eye_{n} &\\
\rho \mA &  \eye_{p}
\end{bmatrix}
\begin{bmatrix}
\left(\nabla^{2}f(\vx^{+})+\rho\mA^{\top}\mA+\frac{1}{\eta}\nabla^{2} h(\vx^{+})\right)^{-1}
&
\\
& \eye_{p}
\end{bmatrix}
\begin{bmatrix}
\frac{1}{\eta}\nabla^{2} h(\vx)
&
-\mA^{\top}
\\
& \eye_{p}
\end{bmatrix}.
\end{split}\label{def:Dg:bmm}
\end{equation}

\subsubsection{Showing that \texorpdfstring{$\det(Dg)$}{det(Dg)} is nonzero globally}
Because
$Dg(\vx,\vlambda)=Dg_{2}(g_{1}(\vx,\vlambda))Dg_{1}(\vx,\vlambda)
$
and $Dg_{1}$, $Dg_{2}$ are square matrices, it suffices to show the global nonsingularity of both $Dg_{1}$, $Dg_{2}$.
$Dg_1$ is nonsingular for all $(\vx,\vlambda)\in\Omega$ because it is a product of two nonsingular matrices in \eqref{Dg1:bmm}: the first factor $(\nabla^{2}f(\vx^{+})+\rho\mA^{\top}\mA+\frac{1}{\eta}\nabla^{2}h(\vx^{+}))^{-1}$ is nonsingular by the positive definiteness shown above, and the second factor has $\det\!\bigl(\begin{bmatrix}\begin{smallmatrix}\frac{1}{\eta}\nabla^2 h(\vx) & -\mA^\top \\ 0 & \eye_p\end{smallmatrix}\end{bmatrix}\bigr) = \det(\frac{1}{\eta}\nabla^2 h(\vx))\neq 0$ since $h$ is strongly convex. Moreover, $Dg_2(\vx,\vlambda)=\bigl[\begin{smallmatrix}\eye_n&0\\ \rho\mA&\eye_p\end{smallmatrix}\bigr]$ by \eqref{Dg2:bmm}, so $\det(Dg_2)=1$.

\subsubsection{Showing any strict saddle lies in the unstable set}
\begin{lemma}\label{lem:mm:critical}
For any stationary point $\vx^{\star}$ of Problem \eqref{eqn:problem:1}, there exists $\vlambda^{\star}\in\R^{p}$ such that
$(\vx^{\star},\vlambda^{\star})$ is a fixed point of the mapping $g\doteq g_{2}\circ g_{1}$. 	
\end{lemma}
\begin{proof}[Proof of \Cref{lem:mm:critical}]
By \Cref{def:critical:point},  for any stationary point $\vx^{\star}$ of  \eqref{eqn:problem:1},  there exists $\vlambda^{\star}$ such that $\mA\vx^\star=\vb$ and $\nabla f(\vx^{\star})+ \mA^{\top}\vlambda^{\star}=\vzero$.
This implies that $(\vx^{+},\vlambda^{+})=(\vx^{\star},\vlambda^{\star})$ and $(\vx,\vlambda)=(\vx^{\star},\vlambda^{\star})$ satisfy  \eqref{eqn:opt:block1:bmm} (defining $g_1$) and \eqref{eqn:opt:block2:bmm} (defining $g_2$) and therefore $(\vx^\star,\vlambda^\star)$ is a fixed point of  $g_2\circ g_1$.
\end{proof}

\begin{lemma}\label{lem:mm:unstable}
Let $\vx^\star$ be a strict saddle of Problem~\eqref{eqn:problem:1}, and let $\vlambda^\star$ be a corresponding multiplier such that $(\vx^\star,\vlambda^\star)\in\Omega$. Then the Jacobian matrix $Dg(\vx^{\star},\vlambda^{\star})$ has an eigenvalue with magnitude greater than 1.
\end{lemma}
\begin{proof}[Proof of \Cref{lem:mm:unstable}]
To simplify notation, we denote $\mH=\nabla^{2}h(\vx^{\star}),\mF=\nabla^{2}f(\vx^{\star})$.
We compute the Jacobian matrix $Dg(\vx^{\star},\vlambda^{\star})$ by plugging 
\[(\vx,\vlambda)=(\vx^{+},\vlambda^{+})=(\vx^{\star},\vlambda^{\star})\]
into \eqref{def:Dg:bmm}:
\begin{align*}
Dg(\vx^{\star},\vlambda^{\star})
&=
\begin{bmatrix}
\eye_{n} &\\
\rho \mA & \eye_{p}
\end{bmatrix}
\begin{bmatrix}
(\mF+\rho\mA^{\top}\mA+\frac{1}{\eta} \mH)^{-1}
&
\\
& \eye_{p}
\end{bmatrix}
\begin{bmatrix}
\frac{1}{\eta} \mH
&
-\mA^{\top}
\\
& \eye_{p}
\end{bmatrix}
\\
&=
\begin{bmatrix}
\mF+\rho\mA^{\top}\mA+\frac{1}{\eta} \mH
&
\\
-\rho\mA& \eye_{p}
\end{bmatrix}^{-1}
\begin{bmatrix}
\frac{1}{\eta} \mH
&
-\mA^{\top}
\\
& \eye_{p}
\end{bmatrix}\\&=\eye - 
\begin{bmatrix}
\mF+\rho\mA^{\top}\mA+\frac{1}{\eta} \mH
&
\\
-\rho\mA
& \eye_{p}
\end{bmatrix}^{-1}
\begin{bmatrix}
\mF+\rho\mA^{\top}\mA
&
\mA^{\top}
\\
-\rho\mA
&
\end{bmatrix}
\\&\doteq\eye - \mPhi.
\end{align*}
It suffices to show that $\mPhi$ has a real negative eigenvalue (equivalently, $\det(\mPhi+\mu\eye)=0$ for some $\mu>0$).
For $\mu>0$, taking the Schur complement with respect to the bottom-right block $\mu\eye_p$ gives
\begin{align*}
0=\det(\mPhi+\mu\eye)
&\Longleftrightarrow
\det\!\left(
\begin{bmatrix}
(1+\mu)\mF+(1+\mu)\rho\mA^{\top}\mA+\frac{\mu}{\eta}\mH &
\mA^{\top}\\
-(1+\mu)\rho\mA &\mu\eye_{p}
\end{bmatrix}
\right)=0
\\
&\Longleftrightarrow
\mu^{p}\det\!\left((1+\mu)\mF+\frac{\mu}{\eta}\mH+\frac{(1+\mu)^{2}}{\mu}\rho\mA^{\top}\mA\right)=0
\\
&\Longleftrightarrow
\det\!\left((1+\mu)\mF+\frac{\mu}{\eta}\mH+\frac{(1+\mu)^{2}}{\mu}\rho\mA^{\top}\mA\right)=0.
\end{align*}
Note that
\begin{align}
\mJ(\mu)\doteq(1+\mu)\mF+\frac{\mu}{\eta}\mH + \frac{(1+\mu)^{2}}{\mu}\rho\mA^{\top}\mA
\label{def:Jmu:bmm}
\end{align}
is real symmetric for every $\mu>0$, and its entries depend continuously on $\mu$.

\begin{remark}[Intuition for the test direction]\label{rem:intuition:ALM}
The strict-saddle condition (\Cref{def:critical:point}, item~3) provides $\vz\in\Null(\mA)$ with $\vz^{\top}\nabla^2 f(\vx^\star)\vz<0$.
Since $\mA\vz=\vzero$, we have $\vz^\top\mA^\top\mA\vz=\|\mA\vz\|^2=0$. Thus the penalty term in $\vz^\top\mJ(\mu)\vz$ vanishes for every $\mu>0$, so along this feasible direction only the Hessian and Bregman kernel terms remain.
At $\mu\to0^{+}$ the Bregman kernel term $\frac{\mu}{\eta}\mH$ vanishes as well, so the sign of $\vz^{\top}\mJ(\mu)\vz$ is governed by the negative Hessian alone, giving $\lambda_{\min}(\mJ(\mu))<0$.
At $\mu\to\infty$ both $\frac{\mu}{\eta}\mH\succ0$ and $\frac{(1+\mu)^{2}}{\mu}\rho\mA^{\top}\mA\succeq0$ dominate, forcing $\mJ(\mu)\succ0$.
Continuity of eigenvalues then yields a zero crossing of $\lambda_{\min}(\mJ(\mu))$ at some intermediate $\mu^\star>0$.
\end{remark}

First, we show that for $\mu\to0^{+}$, $\lambda_{\min}(\mJ(\mu))<0$. Using the third item of \Cref{def:critical:point} and that $\mF=\nabla^2 f(\vx^\star)$,  there exists $\vz\in\Null(\mA)$ such that $\vz^{\top}\mF \vz<-\sigma\|\vz\|_{2}^{2}$ for some $\sigma>0$.   Now
\begin{align*}
\vz^{\top}\mJ(\mu)\vz
&= (1+\mu)\vz^{\top}\mF\vz+\frac{\mu}{\eta}\vz^{\top}\mH\vz
< -\sigma(1+\mu)\|\vz\|_{2}^{2}+\frac{\mu}{\eta}\vz^{\top}\mH\vz
=-\sigma\|\vz\|_{2}^{2}-\mu\left(\sigma\|\vz\|_{2}^{2}-\frac{1}{\eta}\vz^{\top}\mH\vz\right).
\end{align*}
The first term $-\sigma\|\vz\|_{2}^{2}$ is strictly negative and independent of $\mu$, while the second term is $O(\mu)$; hence the right-hand side is negative for all sufficiently small $\mu>0$.
Since $\mJ(\mu)$ is a symmetric matrix, this implies that $\lambda_{\min}(\mJ(\mu))<0$ when $\mu$ is a sufficiently small positive number.
Second, note that
\begin{align*}
\lim_{\mu\to\infty} \frac{\mJ(\mu)}{\mu}&= \mF+ \frac{\mH}{\eta}+\rho\mA^{\top}\mA\succ 0,
\end{align*}
where the positive definiteness follows from $\mF\succeq -L_f\mH$ (relative smoothness) and $\frac{1}{\eta}\mH\succ L_f\mH$ (since $\eta<1/L_f$ and $\mH\succ0$), giving $\mF+\frac{1}{\eta}\mH\succ0$, plus $\rho\mA^\top\mA\succeq0$.
Since $\mJ(\mu)$ is symmetric, its eigenvalues are real and continuous in $\mu$ \cite[Theorem~5.1]{kato2013perturbation}; the intermediate value theorem then gives $\lambda_{\min}(\mJ(\mu))=0$ for some $\mu>0$, i.e., $\det(\mJ(\mu))=0$.
\end{proof}

By \Cref{thm:jason}, the set of initial points from which \Cref{alg:bmm} converges to a strict saddle has Lebesgue measure zero.
Since random initialization is absolutely continuous with respect to Lebesgue measure, the probability that the iterates converge to a strict saddle is zero.

\subsection{Proof sketch for the two-block case (\Cref{thm:badmm})}
\label{sec:sketch:twoblock}
The full proof is in \Cref{pf:thm:badmm}; here we outline the key steps and highlight where the argument differs from the single-block case above.

\paragraph{Step 1: Fixed-point map and Jacobian.}
One iteration of \Cref{alg:badmm} is decomposed as $g=g_3\circ g_2\circ g_1$ (primal $\vx$-update, primal $\vy$-update, dual ascent). By the implicit function theorem applied to the two primal optimality conditions, $g_1$ and $g_2$ are $\mathcal{C}^1$ on $\Omega$; their Jacobians are nonsingular because the corresponding subproblem Hessians with respect to the updated primal variables are positive definite. The dual update $g_3$ has a nonsingular block triangular Jacobian, and hence $Dg=Dg_3\,Dg_2\,Dg_1$ is well defined and nonsingular on $\Omega$.

\paragraph{Step 2: Jacobian at a fixed point.}
At a strict-saddle fixed point $(\vx^\star,\vy^\star,\vlambda^\star)$, the Jacobian evaluates to $Dg=\eye-\mPhi$ with
\[
\mPhi=\mM^{-1}\mN,
\]
where $\mM$ collects the Bregman-regularized subproblem Hessian blocks, while $\mN$ collects the remaining curvature and primal dual coupling blocks arising from the augmented Lagrangian and the Gauss Seidel ordering (the explicit expression is given in \eqref{phi:badmm} of \Cref{pf:thm:badmm}).
It suffices to find $\mu>0$ with $\det(\mPhi+\mu\eye)=0$.

\paragraph{Step 3: Schur reduction and symmetrization.}
After Schur-complementing the dual block (as in the single-block case), the determinant condition reduces to the singularity of a $2\times 2$ block matrix $\mC(\mu)\in\R^{(n+m)\times(n+m)}$. Unlike the single-block case, $\mC(\mu)$ is \emph{not} symmetric: its $(1,2)$ and $(2,1)$ blocks carry different scalar coefficients depending on~$\mu$. The remedy is a diagonal similarity $\mD=\mathrm{diag}(\eye_n,\alpha(\mu)\eye_m)$ with $\alpha(\mu)=\sqrt{(2+\mu+1/\mu)/(2+1/\mu)}$, which symmetrizes $\mC(\mu)$ to a real-symmetric matrix $\mJ(\mu)=\mD^{-1}\mC(\mu)\mD$ without changing its determinant.

\paragraph{Step 4: Sign change via a $\mu$-dependent test direction.}
The symmetrization changes the natural test direction $(d_x,d_y)$ to $(d_x,s(\mu)d_y)$ with $s(\mu)=\alpha(\mu)\to 1$ as $\mu\to 0^+$. Using the KKT constraint $\mA d_x+\mB d_y=\vzero$ and the strict-saddle negativity $d_x^\top\nabla^2 f_1 d_x+d_y^\top\nabla^2 f_2 d_y<0$, one shows $\lambda_{\min}(\mJ(\mu))<0$ for small $\mu>0$, while $\mJ(\mu)/\mu\to\mathrm{diag}(\mF_1+\frac{1}{\eta}\mH_1+\rho\mA^\top\mA,\,\mF_2+\frac{1}{\eta}\mH_2+\rho\mB^\top\mB)\succ0$ gives $\lambda_{\min}(\mJ(\mu))>0$ for large $\mu$. The intermediate value theorem then yields the desired zero crossing. Equivalently, $\det(\mPhi+\mu\eye)=0$ for some $\mu>0$, so $Dg=\eye-\mPhi$ has an eigenvalue with magnitude greater than one.
 
\subsection{\texorpdfstring{Proof sketch for the consensus case (\Cref{thm:cadmm})}{Proof sketch for the consensus case (Theorem~\ref{thm:cadmm})}}
\label{sec:sketch:consensus}
The full proof is in \Cref{sec:thm:cadmm}. Here we indicate how the single block mechanism and the two block spectral argument extend to the star consensus formulation of \Cref{sec:mult}.

\paragraph{Step 1: Fixed point map and nonsingularity.}
One iteration of \Cref{alg:cadmm} defines a smooth fixed point map on the state $(\vx_1,\vy_1,\ldots,\vx_J,\vy_J,\vlambda_2,\ldots,\vlambda_J)$. The update decomposes into $2J+1$ elementary maps, one $\vx_j$ update and one $\vy_j$ update for each agent, followed by the dual update. Relative bi-smoothness, strong bi-convexity of the kernels, and the stepsize condition make each local Bregman proximal Hessian positive definite, so the Jacobian of each $\vx_j$ and $\vy_j$ update is nonsingular by the implicit function theorem. The dual update has a nonsingular block triangular Jacobian. Hence every elementary Jacobian is nonsingular, and the chain rule gives $\det(Dg)\neq0$ on $\Omega$.

\paragraph{Step 2: Jacobian at a strict saddle fixed point.}
At a strict saddle KKT point of the consensus problem, the Jacobian again has the form $Dg=\eye-\mPhi$. As in the two block case, it is enough to find $\mu>0$ such that $\det(\mPhi+\mu\eye)=0$. After a positive row column scaling that symmetrizes the determinant condition and after eliminating the dual blocks by Schur complementation, the condition reduces to the singularity of a real symmetric matrix $\tilde{\mJ}(\mu)$ built from the local Hessian blocks and the star consensus coupling.

\paragraph{Step 3: Star consensus cancellation.}
The key new feature is the star graph structure. A feasible strict saddle direction has equal $\vx$ components across all agents, say $(\vd_\vx,\vd_\vy^1,\ldots,\vd_\vx,\vd_\vy^J)$, because the constraints impose $\vx_j=\vx_1$. This equal $\vx$ direction lies in the kernel of the star graph Laplacian. Consequently, the consensus penalty and the Schur complement correction vanish on the test direction, leaving only the block Hessian contribution. As $\mu\to0^+$, the corresponding quadratic form has a well-defined limit, and this limit is exactly the strict saddle negative curvature.

\paragraph{Step 4: Large $\mu$ positivity and instability.}
For large $\mu$, the scaled matrix $\tilde{\mJ}(\mu)/\mu$ converges to a block diagonal matrix whose diagonal blocks are the positive definite Bregman proximal Hessian blocks. Thus $\lambda_{\min}(\tilde{\mJ}(\mu))>0$ for large $\mu$, while the previous step gives $\lambda_{\min}(\tilde{\mJ}(\mu))<0$ for small positive $\mu$. Continuity of the eigenvalues gives a zero crossing, hence an eigenvalue of $Dg$ with magnitude greater than one. Applying \Cref{thm:jason} then excludes convergence to strict saddle KKT points almost surely.

\section{Numerical Experiments}
\label{sec:application}

We consider two polynomial models: distributed matrix factorization and symmetric tensor factorization. Both fall outside the globally Lipschitz-smooth regime. The former fits the separable consensus framework of \Cref{thm:cadmm} and admits an explicit relative-bi-smoothness certificate; the latter is included because the same splitting idea leads to explicit Bregman-proximal updates even though the resulting objective is not separable across agents. For distributed matrix factorization, the corresponding centralized problem has a benign strict saddle landscape under standard assumptions, so second order stationary points are globally optimal; the tensor example is used as an algorithmic illustration of the same Bregman proximal splitting idea beyond the separable setting.

\subsection{Distributed matrix factorization}
Given a data matrix $\mZ\in\R^{n\times m}$ with column partitioning
$
\mZ=
\begin{bmatrix}
\mZ_1 &\mZ_2 &\cdots & \mZ_J	
\end{bmatrix},
$
where $\mZ_j\in\R^{n\times m_j}$,
suppose $\operatorname{rank}(\mZ) = r$ and let $\mX^\star\in\R^{n\times r}$ span the column space of $\mZ$. The centralized and distributed matrix factorization problems are
\begin{align}
\underline{\tmop{Centralized}}:&\minimize_{\mX,\{\mY_j\}}  \sum_{j=1}^J \|\mX\mY_j^\top-\mZ_j\|_F^2;
\label{centralized}
\\
\underline{\tmop{Distributed}}:&\minimize_{\{\mX_j\},\{\mY_j\}}  \sum_{j=1}^J \|\mX_j\mY_j^\top-\mZ_j\|_F^2 \label{distributed} ~\st~ \mX_j = \mX_1 ~ \forall j\ge 2.
\end{align}
The distributed formulation enforces exact consensus on the shared factor while allowing local right factors, a standard model in distributed low-rank matrix recovery~\cite{zhu2019distributed}.

\paragraph{Applicability of \Cref{thm:cadmm}.}
Each local loss $f_j(\mX_j,\mY_j)\doteq \|\mX_j\mY_j^\top-\mZ_j\|_F^2$ is a polynomial of bidegree~$(2,2)$ in the entries of~$(\mX_j,\mY_j)$, hence relatively bi-smooth under a suitable kernel by \Cref{lem:h:exist2}. In contrast, many standard consensus algorithms (e.g., linearized ADMM~\cite{hong2018gradient}, decentralized gradient descent~\cite{zeng2018nonconvex}, and EXTRA~\cite{shi2015extra}) require global Lipschitz gradient assumptions, which do not hold for such polynomial objectives. The centralized formulation~\eqref{centralized} has a strict-saddle landscape with no spurious local minima under the rank condition above and standard identifiability conditions~\cite{zhu2019distributed,nouiehed2018learning}, so second-order stationary points are global minimizers.

\paragraph{Kernel choice and block subproblems.}
With $\mX_j\in\R^{n\times r}$ and $\mY_j\in\R^{m_j\times r}$, we use the product-type bi-kernel
\begin{equation}
h_j(\mX_j,\mY_j)\doteq\Big(\tfrac{1}{2}\|\mX_j\|_F^2+1\Big)\Big(\tfrac{1}{2}\|\mY_j\|_F^2+1\Big),
\end{equation}
which is strongly bi-convex.  Writing $h_j=h_{x,j}\cdot h_{y,j}$ with $h_{x,j}(\mX_j)\doteq\tfrac{1}{2}\|\mX_j\|_F^2+1$ and $h_{y,j}(\mY_j)\doteq\tfrac{1}{2}\|\mY_j\|_F^2+1$, the product-kernel identity in \Cref{sec:prelim} gives
\begin{align}
D^{x}_{h_j}(\mX_j,\mX_j';\mY_j')&=h_{y,j}(\mY_j')\,D_{h_{x,j}}(\mX_j,\mX_j')=\tfrac{1}{2}\Big(\tfrac{1}{2}\|\mY_j'\|_F^2+1\Big)\|\mX_j-\mX_j'\|_F^2,\nn\\
D^{y}_{h_j}(\mY_j,\mY_j';\mX_j')&=h_{x,j}(\mX_j')\,D_{h_{y,j}}(\mY_j,\mY_j')=\tfrac{1}{2}\Big(\tfrac{1}{2}\|\mX_j'\|_F^2+1\Big)\|\mY_j-\mY_j'\|_F^2.
\end{align}
In \Cref{alg:cadmm}, the Bregman term $\frac{1}{\eta}D^x_{h_j}(\mX_j,\mX_j^{k-1};\mY_j^{k-1})$ contributes a quadratic proximal regularizer $\frac{\alpha_j^k}{2}\|\mX_j-\mX_j^{k-1}\|_F^2$ with iteration-dependent coefficient
$\alpha_j^k\doteq\frac{1}{\eta}(\tfrac{1}{2}\|\mY_j^{k-1}\|_F^2+1)$;
the subsequent $\mY_j$-update uses $\frac{\beta_j^k}{2}\|\mY_j-\mY_j^{k-1}\|_F^2$ with $\beta_j^k\doteq\frac{1}{\eta}(\tfrac{1}{2}\|\mX_j^{k}\|_F^2+1)$.
For a non-hub agent ($j\ge 2$), setting the gradient of the $\mX_j$-subproblem to zero yields the $r\times r$ normal equation
\[
\mX_j^k\big(2\mY_j^{k-1\top}\mY_j^{k-1}+\gamma_j^k\eye_r\big)=2\mZ_j\mY_j^{k-1}-\vlambda_j^{k-1}+\rho\mX_1^k+\alpha_j^k\mX_j^{k-1},
\]
where $\gamma_j^k\doteq\rho+\alpha_j^k$ (the hub-node update has the same structure with $(J{-}1)\rho$ replacing $\rho$). Since $\gamma_j^k>0$, the coefficient matrix $2\mY_j^{k-1\top}\mY_j^{k-1}+\gamma_j^k\eye_r\in\R^{r\times r}$ is positive definite, so each block update reduces to an $r\times r$ linear solve. The $\mY_j$-update carries no consensus constraint; its normal equation is $\mY_j^k(2\mX_j^{k\top}\mX_j^k+\beta_j^k\eye_r)=2\mZ_j^{\top}\mX_j^k+\beta_j^k\mY_j^{k-1}$.

\paragraph{Results.}
We generate a noiseless rank-$r$ instance $\mZ=\mA\mB^\top$ with $\mA\in\R^{n\times r}$ and $\mB\in\R^{m\times r}$ having independent standard normal entries, and partition $\mZ$ into $J$ blocks. The primal variables $\{\mX_j^0\}$, $\{\mY_j^0\}$ and the multipliers are initialized randomly. In \Cref{fig:BADMM}, we use a larger instance with $n=50$, $m_j=50$ for all $j$, $r=4$, $J=100$, and $(\eta,\rho)=(1,1000)$.\footnote{The guarantee of \Cref{thm:cadmm} applies when $\eta$ satisfies the corresponding relative-bi-smoothness stepsize bound. For the kernel above, one may take $L_j^x=L_j^y=4$, so $\eta<1/4$ is sufficient. We use the larger practical stepsize $\eta=1$ for speed; the normal equations above show that the block subproblems remain strongly convex because $\gamma_j^k>0$ and $\beta_j^k>0$.} The objective decreases rapidly to approximately $6.43\times 10^{-25}$ and the consensus residual to approximately $7.31\times 10^{-28}$ after $600$ iterations, showing that the algorithm both fits the data and enforces agreement among the distributed factors at machine precision.

% \paragraph{Strict-saddle avoidance.}
% To directly illustrate the theoretical prediction, we run $50$ independent trials on a smaller instance ($n=10$, $m_j=10$, $r=2$, $J=5$) with the same algorithm and independent random initializations drawn from a standard Gaussian distribution. In every trial, the algorithm converges to a point $(\mX^\star,\{\mY_j^\star\})$ satisfying $\sum_j\|\mX^\star\mY_j^{\star\top}-\mZ_j\|_F^2<10^{-20}$---i.e., a global minimizer---consistent with the prediction that strict saddles are avoided almost surely from random initialization. No trial converges to a saddle point or a spurious local minimum.

\begin{figure}[htp!]
\centering
\includegraphics[width=.6\textwidth]{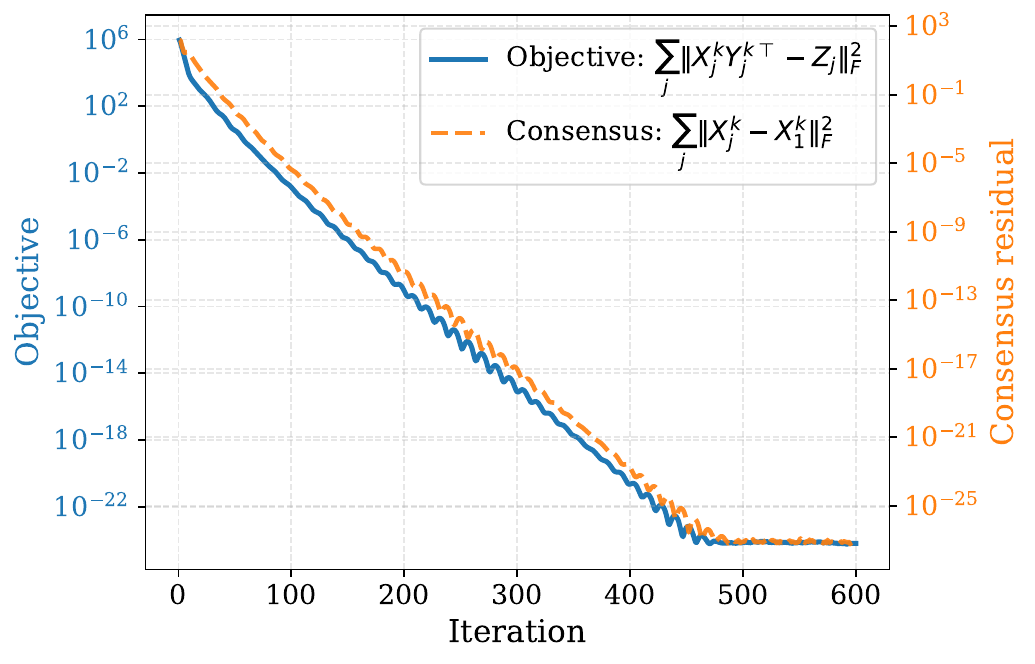}
\caption{Consensus Bregman ADMM on distributed matrix factorization.}
\label{fig:BADMM}
\end{figure}

\subsection{Symmetric tensor factorization}
As a second application, we consider the symmetric tensor factorization problem
\begin{align}
\minimize_{\mU\in\R^{n\times r}}\|\calT - \sum_{i=1}^r \vu_i\otimes \vu_i\otimes \vu_i\|_F^2,
\label{tensordec}
\end{align}
where $\mU = [\vu_1,\cdots,\vu_r]$, $\vu\otimes \vu\otimes \vu$ denotes the symmetric rank-one tensor with $(i,j,k)$-th entry~$u_iu_ju_k$, and $\calT\in\R^{n\times n\times n}$ is a given symmetric tensor. Symmetric tensor decomposition arises widely in signal processing and statistical estimation, including method of moments estimators~\cite{kolda2009tensor,comon2008symmetric,anandkumar2014tensor,li2022local}.

Problem~\eqref{tensordec} is structurally more coupled than a standard CP model because the same factor matrix $\mU$ enters all three modes. Consequently, the usual mode-wise alternating least-squares strategy for nonsymmetric CP, whose block updates reduce to linear least-squares problems, is not directly available in the symmetric formulation~\cite{kolda2009tensor}. A common remedy is to introduce separate mode factors and enforce symmetry through exact coupling constraints, or to encourage symmetry through penalties in related penalized formulations. The exact constraint formulation restores blockwise multilinearity while preserving equivalence with the original symmetric model. Related formulations appear in numerical optimization approaches to symmetric CP decomposition and, more broadly, in structured/coupled factorization frameworks~\cite{kolda2015numerical,sorber2015structured}. Motivated by this observation, we introduce auxiliary factors $\mV,\mW$ and impose symmetry via consensus constraints, thereby rewriting~\eqref{tensordec} as
\begin{align}
\minimize_{\mU,\mV,\mW\in\R^{n\times r}}\|\calT - \sum_{i=1}^r \vu_i\otimes \vv_i\otimes \vw_i\|_F^2~\st~\mU=\mV=\mW,
\label{tensordec:dist}
\end{align}
where $\mV = [\vv_1,\cdots,\vv_r]$, $\mW = [\vw_1,\cdots,\vw_r]$, and $\vu\otimes \vv\otimes \vw$ is a rank-one tensor with $(i,j,k)$-th entry~$u_iv_jw_k$.
Motivated by the consensus update in \Cref{alg:cadmm}, we apply the same Bregman proximal update pattern to~\eqref{tensordec:dist}, treating $(\mU,\mV,\mW)$ as three consensus coupled copies of the factor.

\begin{remark}[Scope of the theoretical guarantee]
\label{rem:tensor:scope}
The consensus model \eqref{eqn:problemC} assumes a separable objective of the form $\sum_j f_j(\vx_j,\vy_j)$, with each term depending only on agent~$j$'s own variables. The symmetric tensor splitting~\eqref{tensordec:dist} does not have this structure, because the single loss $\|\calT-\sum_i \vu_i\otimes\vv_i\otimes\vw_i\|_F^2$ still couples the three factor copies. For this reason, \Cref{thm:cadmm} does not apply to~\eqref{tensordec:dist} as stated. Even so, the Bregman-proximal subproblems remain strongly convex, as the normal equations below show, so the algorithm is perfectly well defined. We include this example primarily as an \emph{algorithmic illustration}: it demonstrates the practical value of the Bregman regularization in stabilizing subproblems for polynomial objectives beyond the Lipschitz-smooth regime. We note that the key obstacle to extending \Cref{thm:cadmm} to the non-separable case is that the fixed-point Jacobian acquires nonzero off-diagonal blocks between agents (through the coupled loss), so the block-sparsity exploited in \Cref{sec:thm:cadmm} no longer holds. A natural conjecture is that the same instability conclusion persists whenever the coupling is ``sufficiently sparse'' relative to the consensus graph---formalizing this is an open direction discussed in the conclusion.
\end{remark}

\paragraph{Kernel choice and block subproblems.}
With $\mU,\mV,\mW\in\R^{n\times r}$, we use the three-factor product kernel
\begin{equation}
h(\mU,\mV,\mW)\doteq \Big(\tfrac{1}{2}\|\mU\|_F^2+1\Big)\Big(\tfrac{1}{2}\|\mV\|_F^2+1\Big)\Big(\tfrac{1}{2}\|\mW\|_F^2+1\Big).
\end{equation}
Fixing two factors and applying the product-kernel identity from \Cref{sec:prelim} to the third gives a scaled Frobenius proximal term. (The identity in \Cref{sec:prelim} is stated for two-factor product kernels; for the three-factor case one simply treats the product of the two fixed factors as a single scalar, reducing to the same formula.)  For the $\mU$-variable:
\begin{equation}
D^{\mU}_{h}(\mU,\mU';\mV',\mW')=\tfrac{1}{2}\Big(\tfrac{1}{2}\|\mV'\|_F^2+1\Big)\Big(\tfrac{1}{2}\|\mW'\|_F^2+1\Big)\|\mU-\mU'\|_F^2,
\end{equation}
and analogously for $\mV$ and $\mW$ by cyclic permutation of the three factors.
In \Cref{alg:cadmm}, $\frac{1}{\eta}D^{\mU}_{h}(\mU,\mU^{k-1};\mV^{k-1},\mW^{k-1})$ contributes a proximal regularizer $\frac{p_{\mU}^k}{2}\|\mU-\mU^{k-1}\|_F^2$ with coefficient
\[
p_{\mU}^{k}\doteq\frac{1}{\eta}\Big(\tfrac{1}{2}\|\mV^{k-1}\|_F^2+1\Big)\Big(\tfrac{1}{2}\|\mW^{k-1}\|_F^2+1\Big),
\]
and similarly $p_{\mV}^k,p_{\mW}^k$ by cyclic permutation.
The loss can be written via the mode-1 matricization $\calT_{(1)}\in\R^{n\times n^2}$ (the matrix whose rows index the first mode and whose columns enumerate all index pairs of the remaining two modes) as $\|\calT_{(1)}-\mU\mQ^{\top}\|_F^2$, where $\mQ\doteq\mW\kr\mV\in\R^{n^2\times r}$ is the Khatri--Rao (column-wise Kronecker) product~\cite{liu2008hadamard}.
Since $\mU$ is the hub node with $J{-}1=2$ consensus constraints, setting the gradient of the $\mU$-subproblem to zero gives the $r\times r$ normal equation
\[
\mU^k\big(2\mQ^{k-1\top}\mQ^{k-1}+\gamma_{\mU}^k\eye_r\big)=2\calT_{(1)}\mQ^{k-1}+\vlambda_{\mV}^{k-1}+\vlambda_{\mW}^{k-1}+\rho(\mV^{k-1}+\mW^{k-1})+p_{\mU}^k\mU^{k-1},
\]
where $\mQ^{k-1}\doteq\mW^{k-1}\kr\mV^{k-1}$ and $\gamma_{\mU}^k\doteq 2\rho+p_{\mU}^k>0$.  The $\mV$- and $\mW$-updates (non-hub agents) have the same structure with mode-2 and mode-3 matricizations $\calT_{(2)},\calT_{(3)}$ and the corresponding Khatri--Rao products $\mW^{k-1}\kr\mU^k$ and $\mV^k\kr\mU^k$.

\paragraph{Results.}
We form $\calT=\sum_{i=1}^r \vu_i^\star\otimes\vu_i^\star\otimes\vu_i^\star$ from a random $\mU^\star\in\R^{n\times r}$ and initialize $(\mU^0,\mV^0,\mW^0)$ as Gaussian perturbations of~$\mU^\star$ with noise level $0.01$ and zero multipliers. In \Cref{fig:BtensorSymADMM}, we use a larger instance with $n=64$, $r=8$, and $(\eta,\rho)=(8\times 10^3,100)$. The reconstruction error decreases to approximately $2.55\times 10^{-17}$ and the symmetry residual $\|\mU-\mV\|_F^2+\|\mU-\mW\|_F^2$ to approximately $3.18\times 10^{-22}$ after $1000000$ iterations. In other words, the split formulation does what it is supposed to do: it drives the factors toward agreement while continuing to reduce the tensor-fitting objective. The run is slower than in the matrix-factorization example, which is unsurprising here because the admissible Bregman regularization is much stronger and the underlying polynomial has higher degree.

% \paragraph{Convergence behavior.}
% As noted in \Cref{rem:tensor:scope}, \Cref{thm:cadmm} does not formally apply to~\eqref{tensordec:dist}. Nevertheless, the first-order convergence conditions (sufficient decrease, boundedness, and the K\L{} inequality) can still be verified for this problem: the objective is a real polynomial and hence satisfies the K\L{} inequality; the Bregman-proximal updates ensure sufficient decrease by construction; and boundedness of iterates is observed empirically (and can be argued via the coercive augmented Lagrangian structure when $\rho$ is large enough). Extending the strict-saddle avoidance guarantee to this non-separable setting remains an open problem, but the numerical results suggest that the qualitative behavior---avoidance of saddle points and convergence to a global minimizer---persists in practice.

\begin{figure}[htp!]
\centering
\includegraphics[width=.6\textwidth]{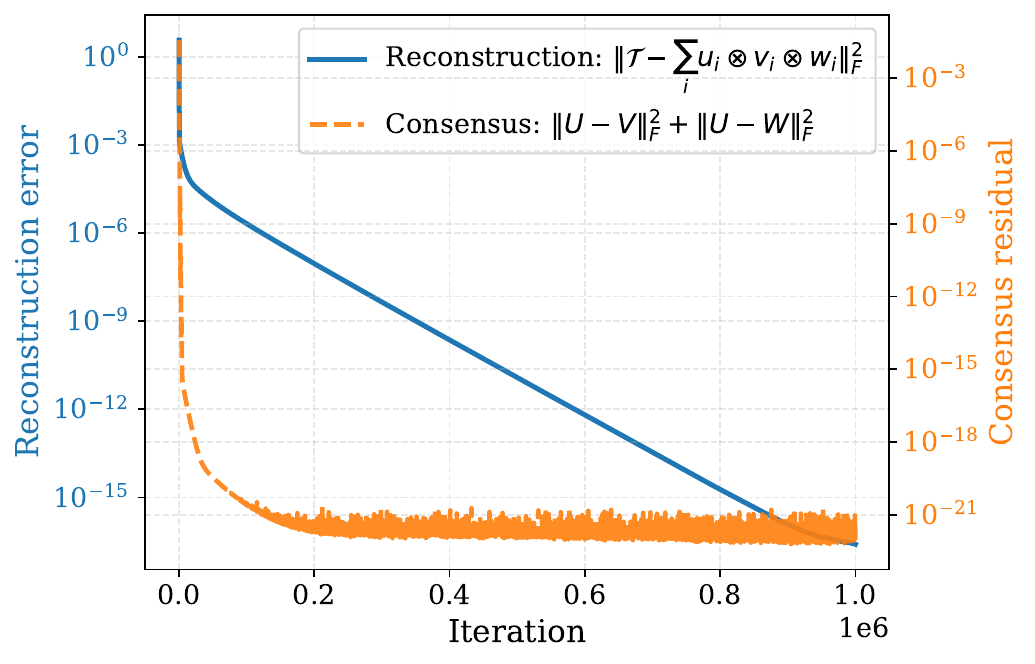}
\caption{Consensus Bregman ADMM on symmetric tensor factorization.}
\label{fig:BtensorSymADMM}
\end{figure}

\section{Conclusion}
The main point of the paper is that the strict-saddle-avoidance mechanism familiar from smooth first-order methods survives in the Bregman ADMM setting, even when global Lipschitz smoothness is unavailable. What is nontrivial is not the dynamical-systems principle itself, but the verification that it continues to apply after replacing Euclidean proximal terms by Bregman regularization. In the two block case this requires a determinant reduction and symmetrization and scaling step specific to the Bregman setting in the spectral argument, while in the consensus case it relies on the special hub and peripheral structure of the star graph and the associated null space cancellation of the consensus penalty. Once these ingredients are in place, the iteration can again be viewed as a smooth primal--dual map on an invariant open domain, so strict saddle KKT points correspond to unstable fixed points and random initialization rules out convergence to them almost surely. Combined with standard first-order convergence results, this gives second-order stationarity of limiting KKT points.

Several natural questions remain. In the two-block setting, our analysis assumes a separable objective $f(\vx,\vy)=f_1(\vx)+f_2(\vy)$, allowing nonzero cross-Hessian terms $\nabla^2_{\vx\vy}f$ would widen the scope of the theory. In the consensus setting, the current result still relies on sum-separability across agents, whereas the symmetric tensor splitting in \Cref{rem:tensor:scope} suggests that interesting non-separable consensus formulations should also be accessible. The key technical barrier is that non-separability introduces off-diagonal Jacobian blocks between agents, breaking the star-shaped block sparsity used in \Cref{sec:thm:cadmm}. A promising approach is to treat these inter-agent couplings as perturbations of the separable Jacobian and show that the negative-curvature direction still dominates for small enough coupling, we leave a rigorous treatment to future work. It would also be worthwhile to understand whether the same dynamical picture persists for stochastic or mini-batch variants of Bregman ADMM.

\section*{Acknowledgment}
This work was supported in part by NSF grants ECCS-2409701 and ECCS-2409702.

\bibliographystyle{unsrt}
\bibliography{nonconvex}

\newpage
\appendix
\setlength{\emergencystretch}{3em}% slightly more generous for dense appendix proofs

\begin{center}
\huge\bfseries \textsc{Appendix}	
\end{center}
\section{Proof of \texorpdfstring{\Cref{thm:badmm}}{Theorem~\ref{thm:badmm}}}
\label{pf:thm:badmm}

\begin{proof}[Proof of \Cref{thm:badmm}]

We verify the two conditions of \Cref{thm:jason}: $\det(Dg)\neq 0$ everywhere, and every strict-saddle KKT point in $\Omega$ is an unstable fixed point of $g$.

One iteration of \Cref{alg:badmm} decomposes into three elementary maps $g_1,g_2,g_3$ (defined below). By \Cref{ass:domain}, the full update, including the intermediate elementary updates $g_1,g_2,g_3$, is well defined on the relevant open state space domains. The relative smoothness and stepsize assumptions make the Bregman proximal Hessians positive definite, so any minimizer is unique and the implicit function theorem applies to the optimality conditions. The nonsingularity of $Dg$ is verified below. For instability, the Bregman proximal term contributes a positive definite shift that cannot cancel the negative curvature at a strict saddle, so the linearized map has an eigenvalue of magnitude greater than one.

\paragraph{Constructing the fixed-point map $g$.}
Represent \eqref{eqn:badmm} as
\begin{equation}
\begin{aligned}
(\vx^+,\vy,\vlambda)&=g_{1}(\vx,\vy,\vlambda),\\
(\vx,\vy^{+},\vlambda)&=g_{2}(\vx,\vy,\vlambda),\\
(\vx,\vy,\vlambda^{+})&=g_{3}(\vx,\vy,\vlambda).
\label{def:g123:badmm}
\end{aligned}
\end{equation}
The elementary updates are well defined on the corresponding open domains by \Cref{ass:domain}; the positive definiteness of the Bregman proximal Hessians gives local uniqueness and smooth dependence of the primal solution maps.
Then \Cref{alg:badmm} iterates the composite map
\begin{align}
(\vx^{+},\vy^{+},\vlambda^{+})=g(\vx,\vy,\vlambda),
\label{eqn:g:badmm}
\end{align}
with
\begin{align}
g\doteq g_{3}\circ g_{2}\circ g_{1}.
\label{def:g:badmm}
\end{align}

\paragraph{Computing the Jacobian matrix $Dg$ of the fixed-point map for \Cref{alg:badmm}.}
By the chain rule, we have
\[
Dg=Dg_{3}Dg_{2}Dg_{1}.
\]

\begin{itemize}[leftmargin=*]
\item{\bf Computing $Dg_{3}$. }
Since $(\vx,\vy,\vlambda^{+})=g_{3}(\vx,\vy,\vlambda)$ with $g_{3}$ defined by the optimality condition of the third-block of \eqref{eqn:badmm}
\begin{align}
\vlambda^{+}= \vlambda+\rho(\mA\vx+\mB\vy-\vb),
\label{eqn:opt:block3:badmm}
\end{align}
then we get
\begin{align}
Dg_{3}(\vx,\vy,\vlambda)=
\begin{bmatrix}
\eye_{n} &&\\
& \eye_{m} &\\
\rho \mA & \rho \mB & \eye_{p}
\end{bmatrix}.
\label{Dg3:badmm}
\end{align}
\item{\bf Computing $Dg_{2}$. }
Since $(\vx,\vy^{+},\vlambda)=g_{2}(\vx,\vy,\vlambda)$, then we get
\begin{align*}
Dg_{2}(\vx,\vy,\vlambda)=
\begin{bmatrix}
\eye_{n} &&\\
\nabla_{\vx} \vy^{+}& \nabla_{\vy}\vy^{+} & \nabla_{\vlambda}\vy^{+}\\
& & \eye_{p}
\end{bmatrix},
\end{align*}
Before invoking the implicit function theorem, we note that relative smoothness means
$L_i\nabla^2 h_i(\vz)\pm \nabla^2 f_i(\vz)\succeq 0$ for $i=1,2$, hence
$\nabla^2 f_i(\vz)\succeq -L_i\nabla^2 h_i(\vz)$.
Since each $h_i$ is strongly convex, $\nabla^2 h_i(\vz)\succ 0$ on the domain. Hence for $\eta<1/L_i$,
\[
\nabla^2 f_i(\vz)+\frac{1}{\eta}\nabla^2 h_i(\vz)\succeq \Bigl(\frac{1}{\eta}-L_i\Bigr)\nabla^2 h_i(\vz)\succ 0, ~ i=1,2.
\]
Adding the positive semidefinite quadratic terms from the augmented Lagrangian (e.g., $\rho\mA^\top\mA$ or $\rho\mB^\top\mB$) preserves positive definiteness. Therefore each primal subproblem in \eqref{eqn:badmm} is strongly convex with a unique minimizer, and the implicit function theorem applies to the optimality conditions.
Specifically, $\nabla_{\vx} \vy^{+}, \nabla_{\vy}\vy^{+}, \nabla_{\vlambda}\vy^{+}$ are obtained by applying the implicit function theorem to the optimality condition of the second block of \eqref{eqn:badmm}
\begin{align}
\nabla f_2(\vy^{+})+\mB^{\top}\vlambda+\rho\mB^{\top}(\mA\vx+\mB\vy^{+}-\vb)+\frac{\nabla h_{2}(\vy^{+})-\nabla h_{2}(\vy)}{\eta}=\vzero.
   \label{eqn:opt:block2:badmm}
\end{align}
Applying the implicit function theorem to \eqref{eqn:opt:block2:badmm} gives
\begin{align*}
&\nabla^{2}f_2(\vy^{+})\nabla_{\vx}\vy^{+}+\rho\mB^{\top}\mA+\rho\mB^{\top}\mB\nabla_{\vx}\vy^{+}+\frac{1}{\eta}\nabla^{2}h_{2}(\vy^{+})\nabla_{\vx}\vy^{+}=\vzero
\\
\Longrightarrow & \nabla_{\vx}\vy^{+}=-\left(\nabla^{2}f_2(\vy^{+})+\rho\mB^{\top}\mB+\frac{1}{\eta}\nabla^{2}h_{2}(\vy^{+})\right)^{-1}\left(\rho\mB^{\top}\mA\right);
\\
&\nabla^{2}f_2(\vy^{+})\nabla_{\vy}\vy^{+}+\rho\mB^{\top}\mB\nabla_{\vy}\vy^{+}+\frac{1}{\eta} \nabla^{2} h_{2}(\vy^{+})\nabla_{\vy}\vy^{+}-\frac{1}{\eta}\nabla^{2} h_{2}(\vy)=\vzero
\\
\Longrightarrow & \nabla_{\vy}\vy^{+}=\left(\nabla^{2}f_2(\vy^{+})+\rho\mB^{\top}\mB+\frac{1}{\eta}\nabla^{2}h_{2}(\vy^{+})\right)^{-1}\left(\frac{1}{\eta} \nabla^{2} h_{2}(\vy)\right);
\\
& \nabla^{2}f_2(\vy^{+}) \nabla_{\vlambda}\vy^{+}+\mB^{\top} +\rho\mB^{\top}\mB\nabla_{\vlambda}\vy^{+}+ \frac{1}{\eta}\nabla^{2} h_{2}(\vy^{+}) \nabla_{\vlambda} \vy^{+}=\vzero
\\
\Longrightarrow &  \nabla_{\vlambda} \vy^{+}=-\left(\nabla^{2}f_2(\vy^{+})+\rho\mB^{\top}\mB+\frac{1}{\eta}\nabla^{2}h_{2}(\vy^{+})\right)^{-1}\mB^{\top}.
\end{align*}

Therefore,
\begin{align}
&Dg_{2}(\vx,\vy,\vlambda)\nonumber
\\
&=
\begin{bmatrix}
\eye_{n} &&\\
\nabla_{\vx} \vy^{+}& \nabla_{\vy}\vy^{+} & \nabla_{\vlambda}\vy^{+}\\
&& \eye_{p}
\end{bmatrix}
\nonumber
\\
&=
\begin{bmatrix}
\eye_{n} &&\\
&
\begin{bmatrix}\nabla^{2}f_2(\vy^{+})+\rho\mB^{\top}\mB+\frac{1}{\eta}\nabla^{2}h_{2}(\vy^{+})\end{bmatrix}^{-1}&
\\
&& \eye_{p}
\end{bmatrix}
\begin{bmatrix}
\eye_{n} &&\\
-\rho\mB^{\top}\mA
&
\frac{1}{\eta} \nabla^{2} h_{2}(\vy)
&
-\mB^{\top}
\\
&& \eye_{p}
\end{bmatrix}\nonumber
\\
&=
\begin{bmatrix}
\eye_{n} &&\\
&
\begin{bmatrix}\nabla^{2}f_2(\vy^{+})+\rho\mB^{\top}\mB+\frac{1}{\eta}\nabla^{2}h_{2}(\vy^{+})\end{bmatrix}^{-1}
&
\\
&& \eye_{p}
\end{bmatrix}
\begin{bmatrix}
\eye_{n} &&\\
-\rho\mB^{\top}\mA
&
\eye_{m}
&
-\mB^{\top}
\\
&& \eye_{p}
\end{bmatrix}
\begin{bmatrix}
\eye_{n}&&\\
&\frac{1}{\eta} \nabla^{2} h_{2}(\vy)&\\
&&\eye_{p}
\end{bmatrix}
\label{Dg2:badmm}
\end{align}
where $\nabla^{2}f_2(\vy^{+})+\rho\mB^{\top}\mB+\frac{1}{\eta}\nabla^{2}h_{2}(\vy^{+})$ is positive definite by relative smoothness, strong convexity of $h_2$, and $\eta<1/L_2$ (and adding the positive semidefinite term $\rho\mB^\top\mB$ preserves positive definiteness).

\item{\bf Computing $Dg_{1}$. }
Since $(\vx^{+},\vy,\vlambda)=g_{1}(\vx,\vy,\vlambda)$, then we get
\begin{align*}
Dg_{1}(\vx,\vy,\vlambda)=
\begin{bmatrix}
\nabla_{\vx} \vx^{+}& \nabla_{\vy}\vx^{+} & \nabla_{\vlambda}\vx^{+}\\
&\eye_{m} & \\
& & \eye_{p}
\end{bmatrix},
\end{align*}
where $\nabla_{\vx} \vx^{+}, \nabla_{\vy}\vx^{+}, \nabla_{\vlambda}\vx^{+}$ can be obtained by applying the implicit function theorem to the optimality condition of the first block of \eqref{eqn:badmm}
\begin{align}
\nabla f_1(\vx^{+})+\mA^{\top}\vlambda+\rho\mA^{\top}(\mA\vx^{+}+\mB\vy-\vb)+\frac{\nabla h_{1}(\vx^{+})-\nabla h_{1}(\vx)}{\eta}=\vzero.
   \label{eqn:opt:block1:badmm}
\end{align}
Applying the implicit function theorem to \eqref{eqn:opt:block1:badmm} gives
\begin{align*}
&\nabla^{2}f_1(\vx^{+})\nabla_{\vx}\vx^{+}+\rho\mA^{\top}\mA\nabla_{\vx}\vx^{+}+\frac{1}{\eta} \nabla^{2} h_{1}(\vx^{+})\nabla_{\vx}\vx^{+}-\frac{1}{\eta}\nabla^{2} h_{1}(\vx)=\vzero
\\
\Longrightarrow & \nabla_{\vx}\vx^{+}=\left(\nabla^{2}f_1(\vx^{+})+\rho\mA^{\top}\mA+\frac{1}{\eta}\nabla^{2}h_{1}(\vx^{+})\right)^{-1}\left(\frac{1}{\eta} \nabla^{2} h_{1}(\vx)\right);
\\
&\nabla^{2}f_1(\vx^{+})\nabla_{\vy}\vx^{+}+\rho\mA^{\top}\mB+\rho\mA^{\top}\mA\nabla_{\vy}\vx^{+}+\frac{1}{\eta}\nabla^{2}h_{1}(\vx^{+})\nabla_{\vy}\vx^{+}=\vzero
\\
\Longrightarrow & \nabla_{\vy}\vx^{+}=-\left(\nabla^{2}f_1(\vx^{+})+\rho\mA^{\top}\mA+\frac{1}{\eta}\nabla^{2}h_{1}(\vx^{+})\right)^{-1}\left(\rho\mA^{\top}\mB\right);
\\
& \nabla^{2}f_1(\vx^{+}) \nabla_{\vlambda}\vx^{+}+\mA^{\top} +\rho\mA^{\top}\mA\nabla_{\vlambda}\vx^{+}+  \frac{1}{\eta} \nabla^{2} h_{1}(\vx^{+}) \nabla_{\vlambda} \vx^{+}=\vzero
\\
\Longrightarrow &  \nabla_{\vlambda} \vx^{+}=-\left(\nabla^{2}f_1(\vx^{+})+\rho\mA^{\top}\mA+\frac{1}{\eta}\nabla^{2}h_{1}(\vx^{+})\right)^{-1}\mA^{\top}.
\end{align*}
Therefore,
\begin{align}
Dg_{1}(\vx,\vy,\vlambda)
=&
\begin{bmatrix}
\nabla_{\vx} \vx^{+}& \nabla_{\vy}\vx^{+} & \nabla_{\vlambda}\vx^{+}\\
& \eye_{m} & \\
&& \eye_{p}
\end{bmatrix}
\nonumber
\\
=&
\begin{bmatrix}
\left(\nabla^{2}f_1(\vx^{+})+\rho\mA^{\top}\mA+\frac{1}{\eta}\nabla^{2}h_{1}(\vx^{+})\right)^{-1}
&&
\\
& \eye_{m}  &
\\
&& \eye_{p}
\end{bmatrix}
\begin{bmatrix}
\frac{1}{\eta} \nabla^{2} h_{1}(\vx)
&
-\rho\mA^{\top}\mB
&
-\mA^{\top}
\\
& \eye_{m} &
\\
&& \eye_{p}
\end{bmatrix}
\label{Dg1:badmm}
\end{align}
where $\nabla^{2}f_1(\vx^{+})+\rho\mA^\top\mA+\frac{1}{\eta}\nabla^{2}h_{1}(\vx^{+})$ is positive definite by relative smoothness, strong convexity of $h_1$, and $\eta<1/L_1$.

\end{itemize}

Finally, by the chain rule, we get
\begin{align}
Dg(\vx,\vy,\vlambda)=Dg_{3}(g_{2}(\vx,\vy,\vlambda))Dg_{2}(g_{1}(\vx,\vy,\vlambda))Dg_{1}(\vx,\vy,\vlambda).
\label{def:Dg:badmm}
\end{align}
\paragraph{Showing that $\det(Dg)$ is nonzero globally.}
Because
\begin{align*}
Dg(\vx,\vy,\vlambda)=Dg_{3}(g_{2}(\vx,\vy,\vlambda))Dg_{2}(g_{1}(\vx,\vy,\vlambda))Dg_{1}(\vx,\vy,\vlambda),
\end{align*}
with $Dg_{1}$, $Dg_{2}$ and $Dg_{3}$ being square matrices, it suffices to show the global nonsingularity of all $Dg_{1}$, $Dg_{2}$ and $Dg_{3}$.

We begin with $Dg_{1}$. Fix $(\vx,\vy,\vlambda)\in\Omega$ and suppose $(\valpha,\vbeta,\vgamma)\in\R^{n}\times\R^{m}\times\R^{p}$ satisfies
$Dg_{1}(\vx,\vy,\vlambda)[\valpha^{\top}~\vbeta^{\top}~\vgamma^{\top}]^{\top}=\vzero$. We show $\valpha=\vbeta=\vgamma=\vzero$. Observe that
\begin{align*}
Dg_{1}(\vx,\vy,\vlambda)
\begin{bmatrix}
\valpha \\ \vbeta \\ \vgamma
\end{bmatrix}=\vzero
\iff &
\begin{bmatrix}
\nabla_{\vx} \vx^{+}& \nabla_{\vy}\vx^{+} & \nabla_{\vlambda}\vx^{+}\\
& \eye_{m} & \\
&& \eye_{p}
\end{bmatrix}
\begin{bmatrix}
\valpha \\ \vbeta \\ \vgamma
\end{bmatrix}
=
\begin{bmatrix}
\vzero \\ \vzero \\ \vzero
\end{bmatrix}
\implies 
\vbeta=\vzero, \vgamma=\vzero,\nabla_{\vx} \vx^{+}\,\valpha=\vzero
\implies  \valpha=\vzero
\end{align*}
where the last step uses the nonsingularity, for every state in $\Omega$, of
\[\nabla_{\vx}\vx^{+}=\left(\nabla^{2}f_1(\vx^{+})+\rho\mA^{\top}\mA+\frac{1}{\eta}\nabla^{2}h_{1}(\vx^{+})\right)^{-1}\left(\frac{1}{\eta} \nabla^{2} h_{1}(\vx)\right),\]
which follows from relative smoothness: the first factor is nonsingular by the positive definiteness established above, and $\frac{1}{\eta}\nabla^2 h_1(\vx)\succ 0$ by strong convexity of $h_1$, so the product is nonsingular. (Rows 2--3 give $\vbeta=\vzero$ and $\vgamma=\vzero$; substituting into row 1 yields $\nabla_\vx\vx^+\,\valpha=\vzero$, hence $\valpha=\vzero$.)

Next consider $Dg_{2}$. Fix $(\vx,\vy,\vlambda)\in\Omega$ and suppose $Dg_{2}(\vx,\vy,\vlambda)[\valpha^{\top}~\vbeta^{\top}~\vgamma^{\top}]^{\top}=\vzero$. We show again that $\valpha=\vbeta=\vgamma=\vzero$. Observe that
\begin{align*}
Dg_{2}(\vx,\vy,\vlambda)
\begin{bmatrix}
\valpha \\ \vbeta \\ \vgamma
\end{bmatrix}=\vzero
\iff &
\begin{bmatrix}
\eye_{n} && \\
\nabla_{\vx} \vy^{+}& \nabla_{\vy}\vy^{+} & \nabla_{\vlambda}\vy^{+}\\
&& \eye_{p}
\end{bmatrix}
\begin{bmatrix}
\valpha \\ \vbeta \\ \vgamma
\end{bmatrix}
=
\begin{bmatrix}
\vzero \\ \vzero \\ \vzero
\end{bmatrix}
\implies 
\valpha=\vzero, \vgamma=\vzero,\nabla_{\vy} \vy^{+}\,\vbeta=\vzero
\implies  \vbeta=\vzero
\end{align*}
where the last step uses the nonsingularity, for every state in $\Omega$, of
\[\nabla_{\vy}\vy^{+}=\begin{bmatrix}\nabla^{2}f_2(\vy^{+})+\rho\mB^{\top}\mB+\frac{1}{\eta}\nabla^{2}h_{2}(\vy^{+})\end{bmatrix}^{-1}\left(\frac{1}{\eta} \nabla^{2} h_{2}(\vy)\right),
\]
which follows from relative smoothness: the first factor is nonsingular by positive definiteness, and $\frac{1}{\eta}\nabla^2 h_2(\vy)\succ 0$ by strong convexity of $h_2$, so the product is nonsingular. (Rows 1 and 3 give $\valpha=\vzero$ and $\vgamma=\vzero$; row 2 then gives $\nabla_\vy\vy^+\,\vbeta=\vzero$, hence $\vbeta=\vzero$.)
Finally, for $Dg_{3}$, fix $(\vx,\vy,\vlambda)\in\Omega$ and suppose $Dg_{3}(\vx,\vy,\vlambda)[\valpha^{\top}~\vbeta^{\top}~\vgamma^{\top}]^{\top}=\vzero$. Then
\begin{align*}
Dg_{3}(\vx,\vy,\vlambda)
\begin{bmatrix}
\valpha \\ \vbeta \\ \vgamma
\end{bmatrix}
=\vzero
\iff &
\begin{bmatrix}
\eye_{n} && \\
&\eye_{m} & \\
\nabla_{\vx} \vlambda^{+}& \nabla_{\vy}\vlambda^{+} & \nabla_{\vlambda}\vlambda^{+}
\end{bmatrix}
\begin{bmatrix}
\valpha \\ \vbeta \\ \vgamma
\end{bmatrix}
=
\begin{bmatrix}
\vzero \\ \vzero \\ \vzero
\end{bmatrix}
\implies 
\begin{bmatrix}
 \valpha
\\
\vbeta
\\
\nabla_{\vlambda}\vlambda^{+} \vgamma
\end{bmatrix}
=
\begin{bmatrix}
\vzero \\ \vzero  \\ \vzero
\end{bmatrix}
\end{align*}
Finally, rows~1--2 give $\valpha=\vzero$ and $\vbeta=\vzero$, and the third row gives $\nabla_{\vlambda}\vlambda^{+}\vgamma=\vzero$; since $\nabla_{\vlambda}\vlambda^{+}=\eye_{p}$ by \eqref{eqn:opt:block3:badmm}, we obtain $\vgamma=\vzero$.

\paragraph{Showing any strict-saddle KKT point in $\Omega$ lies in the unstable set.}
Let $(\vx^{\star},\vy^{\star},\vlambda^{\star})\in\Omega$ be a strict-saddle KKT point of Problem \eqref{eqn:problem:2}. We first show that it is a fixed point of $g=g_{3}\circ g_{2}\circ g_{1}$, i.e.,
\[g(\vx^{\star},\vy^{\star},\vlambda^{\star})=(\vx^{\star},\vy^{\star},\vlambda^{\star}).\]

By the KKT conditions and strict-saddle condition for this point,
\begin{align}
&\begin{cases}
\mA\vx^{\star}+\mB\vy^{\star}&=\vb,\\
\nabla f_1(\vx^{\star})+ \mA^{\top}\vlambda^{\star}&=\vzero,\\
\nabla f_2(\vy^{\star})+ \mB^{\top}\vlambda^{\star}&=\vzero,
\end{cases}
\label{KKT:1:badmm}
\\
 &\vd_{x}^{\top}\nabla^2f_1(\vx^\star)\vd_x+\vd_y^\top\nabla^2 f_2(\vy^\star)\vd_y <0\quad\text{for some $(\vd_{x},\vd_{y}): \mA\vd_{x}+\mB\vd_{y}=\vzero$}.
\label{KKT:2:badmm}
\end{align}
Note that the Hessian $\nabla^2 F(\vz^\star)$ is block-diagonal because $f(\vx,\vy)=f_1(\vx)+f_2(\vy)$ is separable, so the strict-saddle condition involves no cross terms $\nabla^2_{\vx\vy}f$.

Second, $(\vx^{+},\vy^{+},\vlambda^{+})=(\vx^{\star},\vy^{\star},\vlambda^{\star})$ and $(\vx,\vy,\vlambda)=(\vx^{\star},\vy^{\star},\vlambda^{\star})$ satisfy the optimality conditions \eqref{eqn:opt:block1:badmm}, \eqref{eqn:opt:block2:badmm} and \eqref{eqn:opt:block3:badmm} (which define the mappings $g_{1},g_{2},g_{3}$, respectively). Therefore $g_j(\vx^{\star},\vy^{\star},\vlambda^{\star})=(\vx^{\star},\vy^{\star},\vlambda^{\star})$ for $j=1,2,3$, which follows from the well-posedness (unique minimizers) of the subproblems on the primal domain.

It remains to show that the Jacobian matrix $Dg(\vx^{\star},\vy^{\star},\vlambda^{\star})$ has an eigenvalue with magnitude greater than 1.  To simplify notation, we denote
\[\mH_{1}=\nabla^{2}h_{1}(\vx^{\star}),\quad\mH_{2}=\nabla^{2}h_{2}(\vy^{\star}),\quad \mF_1=\nabla^{2}f_1(\vx^{\star}),\quad \mF_2=\nabla^{2}f_2(\vy^{\star}).
\]

We compute the Jacobian matrix $Dg(\vx^{\star},\vy^{\star},\vlambda^{\star})$ by plugging 
\[(\vx,\vy,\vlambda)=(\vx^{+},\vy^{+},\vlambda^{+})=(\vx^{\star},\vy^{\star},\vlambda^{\star})\]
into \eqref{def:Dg:badmm}:
\begin{align*}
&Dg(\vx^{\star},\vy^{\star},\vlambda^{\star})
\\
=&Dg_{3}(\vx^{\star},\vy^{\star},\vlambda^{\star})Dg_{2}(\vx^{\star},\vy^{\star},\vlambda^{\star})Dg_{1}(\vx^{\star},\vy^{\star},\vlambda^{\star})
\\
=&
\begin{bmatrix}
\eye_{n} &&\\
& \eye_{m} &\\
\rho \mA & \rho \mB & \eye_{p}
\end{bmatrix}
\begin{bmatrix}
\eye_{n} &&\\
&(\mF_{2}+\rho\mB^\top\mB+\frac{1}{\eta}\mH_{2})^{-1}&\\
&& \eye_{p}
\end{bmatrix}
\begin{bmatrix}
\eye_{n} &&\\
-\rho\mB^{\top}\mA&\eye_{m}&
-\mB^{\top}\\
&& \eye_{p}
\end{bmatrix}
\begin{bmatrix}
\eye_{n}&&\\
&\frac{1}{\eta}\mH_{2} &\\
&&\eye_{p}
\end{bmatrix}
\\
&
\begin{bmatrix}
(\mF_{1}+\rho\mA^\top\mA+\frac{1}{\eta}\mH_{1})^{-1}&&\\
& \eye_{m}  &\\
&& \eye_{p}
\end{bmatrix}
\begin{bmatrix}
\frac{1}{\eta} \mH_{1}&-\rho\mA^{\top}\mB&-\mA^{\top}\\
& \eye_{m} &\\
&& \eye_{p}
\end{bmatrix}
\\
=&
\begin{bmatrix}
\eye_{n} &&\\
& \eye_{m} &\\
\rho \mA & \rho \mB & \eye_{p}
\end{bmatrix}
\begin{bmatrix}
\eye_{n} &&\\
&(\mF_{2}+\rho\mB^\top\mB+\frac{1}{\eta}\mH_{2})^{-1}&\\
&& \eye_{p}
\end{bmatrix}
\begin{bmatrix}
\eye_{n} &&\\
-\rho\mB^{\top}\mA&\eye_{m}&
-\mB^{\top}\\
&& \eye_{p}
\end{bmatrix}
\\
&
\begin{bmatrix}
(\mF_{1}+\rho\mA^\top\mA+\frac{1}{\eta}\mH_{1})^{-1}&&\\
& \eye_{m}  &\\
&& \eye_{p}
\end{bmatrix}
\begin{bmatrix}
\eye_{n}&&\\
&\frac{1}{\eta}\mH_{2} &\\
&&\eye_{p}
\end{bmatrix}
\begin{bmatrix}
\frac{1}{\eta} \mH_{1}&-\rho\mA^{\top}\mB&-\mA^{\top}\\
& \eye_{m} &\\
&& \eye_{p}
\end{bmatrix}
\\
=&
\begin{bmatrix}
\eye_{n} &&\\
& \eye_{m} &\\
\rho \mA & \rho \mB & \eye_{p}
\end{bmatrix}
\begin{bmatrix}
\eye_{n} &&\\
&(\mF_{2}+\rho\mB^\top\mB+\frac{1}{\eta}\mH_{2})^{-1}&\\
&& \eye_{p}
\end{bmatrix}
\begin{bmatrix}
\eye_{n} &&\\
-\rho\mB^{\top}\mA&\eye_{m}&
-\mB^{\top}\\
&& \eye_{p}
\end{bmatrix}
\\
&
\begin{bmatrix}
(\mF_{1}+\rho\mA^\top\mA+\frac{1}{\eta}\mH_{1})^{-1}&&\\
& \eye_{m}  &\\
&& \eye_{p}
\end{bmatrix}
\begin{bmatrix}
\frac{1}{\eta} \mH_{1}&-\rho\mA^{\top}\mB&-\mA^{\top}\\
&\frac{1}{\eta}\mH_{2}  &\\
&& \eye_{p}
\end{bmatrix}
\\
=&\begin{bmatrix}
\mF_{1}+\frac{1}{\eta}\mH_{1} +\rho\mA^{\top}\mA&&\\
&\mF_{2}+\frac{1}{\eta}\mH_{2}&\mB^{\top}\\
-\rho\mA & -\rho \mB & \eye_{p}
\end{bmatrix}^{-1}
\begin{bmatrix}
\frac{1}{\eta}\mH_{1} & -\rho\mA^{\top}\mB& -\mA^{\top}
\\
& \frac{1}{\eta}\mH_{2} &
\\
&&\eye_{p}
\end{bmatrix}
\\
=&
\eye-
\begin{bmatrix}
\mF_{1}+\frac{1}{\eta}\mH_{1} +\rho\mA^{\top}\mA&&\\
&\mF_{2}+\frac{1}{\eta}\mH_{2}&\mB^{\top}\\
-\rho\mA & -\rho \mB & \eye_{p}
\end{bmatrix}^{-1}
\begin{bmatrix}
\mF_{1}+\rho\mA^{\top}\mA &  \rho\mA^{\top}\mB& \mA^{\top}
\\
&\mF_{2} & \mB^{\top}
\\
-\rho\mA&-\rho\mB&\mzero
\end{bmatrix}
\\
\doteq&\eye-\mPhi,
\end{align*}
where the third equality from the end follows from \Cref{lem:matrix:inverse}, and in the last line we define
\begin{align}
\mPhi\doteq
\begin{bmatrix}
\mF_{1}+\frac{1}{\eta}\mH_{1} +\rho\mA^{\top}\mA&&\\
&\mF_{2}+\frac{1}{\eta}\mH_{2}&\mB^{\top}\\
-\rho\mA & -\rho \mB & \eye_{p}
\end{bmatrix}^{-1}
\begin{bmatrix}
\mF_{1}+\rho\mA^{\top}\mA &  \rho\mA^{\top}\mB& \mA^{\top}
\\
&\mF_{2} & \mB^{\top}
\\
-\rho\mA&-\rho\mB&\mzero
\end{bmatrix}
\label{phi:badmm}
\end{align}

\begin{lemma}\label{lem:matrix:inverse}
If $\mF_{2}+\rho\mB^\top\mB+\frac{1}{\eta}\mH_{2}$ is positive definite (which holds on the primal domain by relative smoothness, strong convexity of $h_2$, and $\eta<1/L_2$), then
{\footnotesize$\begin{bmatrix}
\eye_{n} &&\\
&\mF_{2}+\frac{1}{\eta}\mH_{2}&\mB^{\top}\\
-\rho\mA & -\rho \mB & \eye_{p}
\end{bmatrix}$}
is nonsingular on the primal domain, and its inverse is given by
\begin{align*}
&\begin{bmatrix}
\eye_{n} &&\\
&\mF_{2}+\frac{1}{\eta}\mH_{2}&\mB^{\top}\\
-\rho\mA & -\rho \mB & \eye_{p}
\end{bmatrix}^{-1}
=
\begin{bmatrix}
\eye_{n} &&\\
& \eye_{m} &\\
\rho \mA & \rho \mB & \eye_{p}
\end{bmatrix}
\begin{bmatrix}
\eye_{n} &&\\
&(\mF_{2}+\rho\mB^\top\mB+\frac{1}{\eta}\mH_{2})^{-1}&\\
&& \eye_{p}
\end{bmatrix}
\begin{bmatrix}
\eye_{n} &&\\
-\rho\mB^{\top}\mA&\eye_{m}&
-\mB^{\top}\\
&& \eye_{p}
\end{bmatrix}
\end{align*}
\end{lemma}

\begin{proof}[Proof of \Cref{lem:matrix:inverse}]
Multiplying the four factors below shows that their product equals the identity matrix:
\begin{align*}
&\begin{bmatrix}
\eye_{n} &&\\
&\mF_{2}+\frac{1}{\eta}\mH_{2}&\mB^{\top}\\
-\rho\mA & -\rho \mB & \eye_{p}
\end{bmatrix}
\begin{bmatrix}
\eye_{n} &&\\
& \eye_{m} &\\
\rho \mA & \rho \mB & \eye_{p}
\end{bmatrix}
\begin{bmatrix}
\eye_{n} &&\\
&(\mF_{2}+\rho\mB^\top\mB+\frac{1}{\eta}\mH_{2})^{-1}&\\
&& \eye_{p}
\end{bmatrix}
\begin{bmatrix}
\eye_{n} &&\\
-\rho\mB^{\top}\mA&\eye_{m}&
-\mB^{\top}\\
&& \eye_{p}
\end{bmatrix}
=
\eye.
\end{align*}
Therefore {\footnotesize$\begin{bmatrix}
\eye_{n} &&\\
&\mF_{2}+\frac{1}{\eta}\mH_{2}&\mB^{\top}\\
-\rho\mA & -\rho \mB & \eye_{p}
\end{bmatrix}$} is nonsingular  as long as
$\mF_{2}+\rho\mB^\top\mB+\frac{1}{\eta}\mH_{2}$ is nonsingular.
\end{proof}
Second, we reduce the problem of showing that $Dg(\vx^{\star},\vy^{\star},\vlambda^{\star})$ has an eigenvalue of magnitude greater than 1 to showing that $\mPhi$ (see \eqref{phi:badmm}) has a real negative eigenvalue, or equivalently, that
$\det(\mPhi+\mu\eye)=0$ for some $\mu>0$. The following chain of equivalences reduces this singularity condition to that of a real-symmetric matrix $\mJ(\mu)$; the key steps are Schur complementation (to eliminate the dual block) and a diagonal similarity transform (to symmetrize the off-diagonal blocks). Using $\det(\mU\mV)=\det(\mU)\det(\mV)$ and $\det(\mU^{-1})=\det(\mU)^{-1}$, and multiplying by the nonzero factor {\footnotesize$\det\begin{pmatrix}\mF_{1}+\frac{1}{\eta}\mH_{1} +\rho\mA^{\top}\mA&&\\ &\mF_{2}+\frac{1}{\eta}\mH_{2}&\mB^{\top}\\ -\rho\mA & -\rho \mB & \eye_{p}\end{pmatrix}^{-1}$}, we obtain the following equivalent determinant conditions:
\begin{align*}
&\det(\mPhi+\mu\eye)=0
\\
\Leftrightarrow &
\det\left(
\begin{bmatrix}
\mF_{1}+\frac{1}{\eta}\mH_{1} +\rho\mA^{\top}\mA&&\\
&\mF_{2}+\frac{1}{\eta}\mH_{2}&\mB^{\top}\\
-\rho\mA & -\rho \mB & \eye_{p}
\end{bmatrix}^{-1}
\begin{bmatrix}
\mF_{1}+\rho\mA^{\top}\mA &  \rho\mA^{\top}\mB& \mA^{\top}
\\
&\mF_{2} & \mB^{\top}
\\
-\rho\mA&-\rho\mB&\mzero
\end{bmatrix}
+\mu\eye
\right)=0
\\
\Leftrightarrow &
\det\left(
\begin{bmatrix}
\mF_{1}+\rho\mA^{\top}\mA &  \rho\mA^{\top}\mB& \mA^{\top}
\\
  &\mF_{2} & \mB^{\top}
\\
-\rho\mA&-\rho\mB&\mzero
\end{bmatrix}
+\mu
\begin{bmatrix}
\mF_{1}+\frac{1}{\eta}\mH_{1} +\rho\mA^{\top}\mA&&\\
  &\mF_{2}+\frac{1}{\eta}\mH_{2}&\mB^{\top}\\
-\rho\mA & -\rho \mB & \eye_{p}
\end{bmatrix}
\right)=0
\\
\Leftrightarrow &
\det\left(
\begin{bmatrix}
(1+\mu)\mF_{1}+(1+\mu)\rho\mA^{\top}\mA+\frac{\mu}{\eta}\mH_{1} &  \rho\mA^{\top}\mB& \mA^{\top}
\\
&(1+\mu)\mF_{2}+\frac{\mu}{\eta}\mH_{2} & (1+\mu)\mB^{\top}
\\
-(1+\mu)\rho\mA&-(1+\mu)\rho\mB&\mu\eye_{p}
\end{bmatrix}
\right)=0
\\
\Leftrightarrow &
\det\left(
\begin{bmatrix}
(1+\mu)\mF_{1}+(1+\mu)\rho\mA^{\top}\mA+\frac{\mu}{\eta}\mH_{1} &  \rho\mA^{\top}\mB& \mA^{\top}
\\
(1+\mu)\rho\mB^\top\mA &(1+\mu)\mF_{2}+(1+\mu)\rho\mB^\top\mB+\frac{\mu}{\eta}\mH_{2} & \mB^{\top}
\\
-(1+\mu)\rho\mA&-(1+\mu)\rho\mB&\mu\eye_{p}
\end{bmatrix}
\right)=0
\\
\Leftrightarrow &
\det\left(
\begin{bmatrix}
(1+\mu)\mF_{1}+(1+\mu)\rho\mA^{\top}\mA+\frac{\mu}{\eta}\mH_{1} &  \rho\mA^{\top}\mB
\\
(1+\mu)\rho\mB^\top\mA& (1+\mu)\mF_{2}+(1+\mu)\rho\mB^\top\mB+\frac{\mu}{\eta}\mH_{2}
\end{bmatrix}
+(1+\frac{1}{\mu})\rho
\begin{bmatrix}
	\mA^\top\mA& \mA^\top\mB\\
\mB^\top\mA & \mB^\top\mB
\end{bmatrix}
\right)=0
\\
\Leftrightarrow &
\det\left(
\begin{bmatrix}
(1+\mu)\mF_{1}+(2+\mu+\frac{1}{\mu})\rho\mA^{\top}\mA+\frac{\mu}{\eta}\mH_{1} &  (2+\frac{1}{\mu})\rho\mA^{\top}\mB
\\
(2+\mu+\frac{1}{\mu})\rho\mB^\top\mA &
(1+\mu)\mF_{2}+(2+\mu+\frac{1}{\mu})\rho\mB^\top\mB+\frac{\mu}{\eta}\mH_{2}
\end{bmatrix}
\right)=0
\\
\Leftrightarrow &
\det\left(
\begin{bmatrix}
(1+\mu)\mF_{1}+(2+\mu+\frac{1}{\mu})\rho\mA^{\top}\mA+\frac{\mu}{\eta}\mH_{1} &  \sqrt{(2+\frac{1}{\mu})(2+\mu+\frac{1}{\mu})}\rho\mA^{\top}\mB
\\
\sqrt{(2+\frac{1}{\mu})(2+\mu+\frac{1}{\mu})}\rho\mB^\top\mA
&
(1+\mu)\mF_{2}+(2+\mu+\frac{1}{\mu})\rho\mB^\top\mB+\frac{\mu}{\eta}\mH_{2}
\end{bmatrix}
\right)=0,
\end{align*}
The first two equivalences use $\mPhi+\mu\eye=\mM^{-1}(\mN+\mu\mM)$ (from $\mPhi=\mM^{-1}\mN$), so $\det(\mPhi+\mu\eye)=\det(\mM^{-1})\det(\mN+\mu\mM)$. Here $\mM$ is nonsingular: taking the Schur complement with respect to its bottom-right block $\eye_p$ yields a block-lower-triangular matrix with positive definite diagonal blocks $\mF_{1}+\frac{1}{\eta}\mH_{1}+\rho\mA^{\top}\mA$ and $\mF_{2}+\frac{1}{\eta}\mH_{2}+\rho\mB^{\top}\mB$, so $\det(\mM^{-1})\neq 0$. The third adds $\mN$ and $\mu\mM$ blockwise.
The fourth applies the block row operation $R_2\leftarrow R_2-\mB^{\top}R_3$ (where $R_i$ denotes the $i$-th block row), which preserves the determinant. Since $R_3=[-(1+\mu)\rho\mA,\,-(1+\mu)\rho\mB,\,\mu\eye_p]$, subtracting $\mB^{\top}R_3$ from $R_2=[0,\,(1+\mu)\mF_2+\tfrac{\mu}{\eta}\mH_2,\,(1+\mu)\mB^{\top}]$ gives
\begin{align*}  
&[(1+\mu)\rho\mB^{\top}\mA,\;(1+\mu)\mF_2+(1+\mu)\rho\mB^{\top}\mB+\tfrac{\mu}{\eta}\mH_2,\;(1+\mu)\mB^{\top}-\mu\mB^{\top}]
\\=&
[(1+\mu)\rho\mB^{\top}\mA,\;(1+\mu)\mF_2+(1+\mu)\rho\mB^{\top}\mB+\tfrac{\mu}{\eta}\mH_2,\;\mB^{\top}].
\end{align*}
The fifth applies the Schur complement with respect to $\mu\eye_p$.
Writing the $3\times 3$ matrix after the fourth step as $\bigl[\begin{smallmatrix}\mP&\mQ\\\mR^{\top}&\mu\eye_p\end{smallmatrix}\bigr]$ with $\mQ=\left[\begin{smallmatrix}\mA^{\top}\\\mB^{\top}\end{smallmatrix}\right]$ and $\mR^{\top}=-(1+\mu)\rho[\mA\ \mB]$, the Schur formula gives
\[
\det\!\begin{bmatrix}\mP&\mQ\\\mR^{\top}&\mu\eye_p\end{bmatrix}=\mu^p\det\!\Bigl(\mP-\tfrac{1}{\mu}\mQ\mR^{\top}\Bigr).
\]
Since $\mu^p\neq 0$, the condition reduces to $\det(\mP-\tfrac{1}{\mu}\mQ\mR^{\top})=0$, where
\[
-\tfrac{1}{\mu}\mQ\mR^{\top}=\Bigl(1+\tfrac{1}{\mu}\Bigr)\rho\begin{bmatrix}\mA^{\top}\mA&\mA^{\top}\mB\\\mB^{\top}\mA&\mB^{\top}\mB\end{bmatrix}.
\]
The sixth collects like terms: $(1+\mu)+(1+\tfrac{1}{\mu})=2+\mu+\tfrac{1}{\mu}$ on the diagonal $\rho\mA^{\top}\mA$, $\rho\mB^{\top}\mB$ blocks and $1+(1+\tfrac{1}{\mu})=2+\tfrac{1}{\mu}$ on the off-diagonal $\rho\mA^{\top}\mB$ block; call the resulting $2\times 2$ block matrix $\mC(\mu)$. Its $(1,2)$ and $(2,1)$ blocks carry different coefficients $2+\tfrac{1}{\mu}$ and $2+\mu+\tfrac{1}{\mu}$, so $\mC(\mu)$ is not yet symmetric.
The last step symmetrizes $\mC(\mu)$:
Setting $\mD=\mathrm{diag}(\eye_n,\alpha\eye_m)$ with $\alpha=\sqrt{(2+\mu+\tfrac{1}{\mu})/(2+\tfrac{1}{\mu})}$, the similarity $\mD^{-1}\mC(\mu)\mD$ scales the $(1,2)$ block by $\alpha$ and the $(2,1)$ block by $\alpha^{-1}$, equalizing both to $\sqrt{(2+\tfrac{1}{\mu})(2+\mu+\tfrac{1}{\mu})}\,\rho\mA^{\top}\mB$, yielding the symmetric matrix $\mJ(\mu)$; and $\det(\mC(\mu))=0\Leftrightarrow\det(\mJ(\mu))=0$ since $\det(\mD)\neq 0$.

Therefore, the problem is now reduced to showing $\mJ(\mu)$ is singular for some $\mu>0$,
where
\[
\mJ(\mu)
\doteq
\begin{bmatrix}
(1+\mu)\mF_{1}+(2+\mu+\frac{1}{\mu})\rho\mA^{\top}\mA+\frac{\mu}{\eta}\mH_{1} &  \sqrt{(2+\frac{1}{\mu})(2+\mu+\frac{1}{\mu})}\rho\mA^{\top}\mB
\\
\sqrt{(2+\frac{1}{\mu})(2+\mu+\frac{1}{\mu})}\rho\mB^\top\mA &(1+\mu)\mF_{2}+(2+\mu+\frac{1}{\mu})\rho\mB^\top\mB+\frac{\mu}{\eta}\mH_{2}
\end{bmatrix}
\]
The matrix $\mJ(\mu)$ is real-symmetric and varies continuously with $\mu$ \cite[Theorem~5.1]{kato2013perturbation}, so its eigenvalues are real and continuous. It suffices to prove that $\lambda_{\min}(\mJ(\mu))$ crosses zero for some $\mu>0$, since $\lambda_{\min}(\mJ(\mu^\star))=0$ implies $\det(\mJ(\mu^\star))=0$.

\begin{remark}[Intuition for the $\mu$-dependent test direction]\label{rem:intuition:ADMM}
In the single-block ALM (\Cref{lem:mm:unstable}) one can take a fixed $\vz\in\Null(\mA)$: the penalty $\rho\mA^{\top}\mA$ vanishes on $\Null(\mA)$ independently of~$\mu$.
With two blocks the cross-coupling $\rho\mA^{\top}\mB$ complicates matters.
Before symmetrization, the reduced matrix $\mC(\mu)$ has asymmetric off-diagonal coefficients $2+\frac{1}{\mu}$ and $2+\mu+\frac{1}{\mu}$. The diagonal similarity transform equalizes these coefficients in the symmetric matrix $\mJ(\mu)$, but it also changes the natural strict-saddle test direction to the scaled direction $(\vd_x,s(\mu)\vd_y)$, where $s(\mu)\doteq \sqrt{(2+\mu+\frac{1}{\mu})/(2+\frac{1}{\mu})}$.
This scaled direction lets the KKT constraint $\mA\vd_x+\mB\vd_y=\vzero$ eliminate the leading cross term; after expansion, the remaining penalty contribution is a single non-negative term $\rho\cdot(\text{positive factor})\cdot\mu\|\mB\vd_y\|_2^2$ that vanishes as $\mu\to0^{+}$.
Since $s(\mu)\to1$ as $\mu\to0^{+}$, the test direction reduces to the strict-saddle pair $(\vd_x,\vd_y)$ and the quadratic form tends to $\vd_x^{\top}\mF_1\vd_x+\vd_y^{\top}\mF_2\vd_y<0$.
\end{remark}

\begin{lemma}\label{lem:phi:zero:badmm}
Let $\vy(\mu)=(\vd_x, s(\mu) \vd_y)$ with $\vd_x,\vd_y$ defined in KKT condition \eqref{KKT:2:badmm} and $s(\mu)\doteq \frac{\sqrt{2+\mu+\frac{1}{\mu}}}{\sqrt{2+\frac{1}{\mu}}}$. Then $\lambda_{\min}(\mJ(\mu^\star))=0$ for some $\mu^\star>0$, hence $\det(\mJ(\mu^\star))=0$.
\end{lemma}
\begin{proof}[Proof of \Cref{lem:phi:zero:badmm}]
Define $\phi(\mu)\doteq\vy(\mu)^{\top}\mJ(\mu)\vy(\mu)$. Expanding the quadratic form using the definition of $\mJ(\mu)$ and $\vy(\mu)=(\vd_x,s(\mu)\vd_y)$, and using the KKT constraint $\mA\vd_x+\mB\vd_y=\vzero$ to eliminate the cross term, yields
\begin{align*}
\phi(\mu)
&=\frac{\mu}{\eta}(\vd_x^\top\mH_1\vd_x+s(\mu)^2\vd_y^\top\mH_2\vd_y)+(1+\mu)(\vd_x^\top\mF_1\vd_x+s(\mu)^2\vd_y^\top\mF_2\vd_y)
+\rho(2+\mu+\frac{1}{\mu})(s(\mu)^2-1)\|\mB\vd_y\|_2^2
\\
 &=\frac{\mu}{\eta}(\vd_x^\top\mH_1\vd_x+s(\mu)^2\vd_y^\top\mH_2\vd_y)+(1+\mu)(\vd_x^\top\mF_1\vd_x+s(\mu)^2\vd_y^\top\mF_2\vd_y)+\rho\frac{2+\mu+\frac{1}{\mu}}{2+\frac{1}{\mu}}\mu\|\mB\vd_y\|_2^2
\\
 &=\frac{\mu}{\eta}(\vd_x^\top\mH_1\vd_x+s(\mu)^2\vd_y^\top\mH_2\vd_y)+(1+\mu)(\vd_x^\top\mF_1\vd_x+s(\mu)^2\vd_y^\top\mF_2\vd_y)+\rho\frac{2\mu+\mu^2+1}{2\mu+1}\mu\|\mB\vd_y\|_2^2
\end{align*}
where in the second line we used $s(\mu)^2-1=\frac{\mu}{2+\frac{1}{\mu}}$ (from the definition of $s(\mu)$), and in the last line we rewrote $\frac{2+\mu+\frac{1}{\mu}}{2+\frac{1}{\mu}}=\frac{2\mu+\mu^2+1}{2\mu+1}$.
We first use $\phi(\mu)$ to show that $\lambda_{\min}(\mJ(\mu))<0$ for sufficiently small $\mu>0$. We then prove separately that $\mJ(\mu)\succ0$ for sufficiently large $\mu$. The intermediate value theorem applied to $\lambda_{\min}(\mJ(\mu))$ gives the desired zero crossing.
\begin{itemize}[leftmargin=*]
	\item First, since $\lim_{\mu\to0^+}s(\mu)=1$, we have
\[\lim_{\mu\to0^+}\phi(\mu)=\vd_{x}^{\top}\mF_1\vd_x+\vd_y^\top\mF_2\vd_y <0,\]
by KKT condition \eqref{KKT:2:badmm}.
\item Second, we show that $\phi(N)>0$ for some sufficiently large $N$.
If $\vd_y\neq\vzero$, then since 
$\lim_{\mu\to\infty}\frac{s(\mu)^2}{\mu}=\frac{1}{2},$
we have
\begin{align*}
\lim_{\mu\to\infty}\frac{\phi(\mu)}{\mu^{2}}
&=\frac{1}{2\eta}\vd_y^\top\mH_2\vd_y+\frac{1}{2}\vd_y^\top\mF_2\vd_y+\frac{\rho}{2}\|\mB\vd_y\|_2^2
= \frac{1}{2}\vd_y^\top\left( \frac{1}{\eta}\mH_2+\mF_2+\rho\mB^\top\mB\right)\vd_y
>0,
\end{align*}
where the last inequality uses $\frac{1}{\eta}\mH_2+\mF_2\succ0$ (from relative smoothness, strong convexity of $h_2$, and $\eta<1/L_2$) and $\mB^\top\mB\succeq0$.
Therefore $\phi(N)>0$ for some sufficiently large $N$.
If $\vd_y=\vzero$, then $\vd_x\neq\vzero$ (since $(\vd_x,\vd_y)\neq(\vzero,\vzero)$), and the KKT condition \eqref{KKT:2:badmm} gives $\mA\vd_x=\vzero$.
In this case $s(\mu)^2\vd_y^\top\mH_2\vd_y=0$, $s(\mu)^2\vd_y^\top\mF_2\vd_y=0$, and $\|\mB\vd_y\|_2^2=0$, so
\[
\lim_{\mu\to\infty}\frac{\phi(\mu)}{\mu}=\vd_x^\top\!\left(\frac{1}{\eta}\mH_1+\mF_1\right)\vd_x >0,
\]
by the positive definiteness of $\frac{1}{\eta}\mH_1+\mF_1+\rho\mA^\top\mA$ (from relative smoothness, strong convexity of $h_1$, and $\eta<1/L_1$) and $\mA\vd_x=\vzero$. Therefore $\phi(N)>0$ for some sufficiently large $N$.

\item Finally, we combine the two cases to conclude that $\mJ(\mu)$ is singular for some $\mu^\star>0$.
From the first item we have, by continuity of $\phi$ on $(0,\infty)$, that $\phi(\mu)<0$ for all sufficiently small $\mu>0$, which gives $\vy(\mu)^\top\mJ(\mu)\vy(\mu)<0$
and therefore $\lambda_{\min}(\mJ(\mu))<0$.
From the second item we have $\phi(N)>0$ for some sufficiently large $N>0$. In addition, we show $\lambda_{\min}(\mJ(N))>0$ for large $N$ via a large-$\mu$ positive-definiteness argument: dividing $\mJ(\mu)$ by $\mu$, the off-diagonal coefficient is $O(\sqrt{\mu})$ and therefore vanishes after scaling by $1/\mu$, while the diagonal blocks converge. Define
\[
\bar{\mJ}\doteq\lim_{\mu\to\infty}\frac{\mJ(\mu)}{\mu}
=
\begin{bmatrix}
\mF_1+\rho\mA^\top\mA+\tfrac{1}{\eta}\mH_1 & \mzero\\
\mzero & \mF_2+\rho\mB^\top\mB+\tfrac{1}{\eta}\mH_2
\end{bmatrix}\succ0,
\]
where the positive definiteness follows from relative smoothness \eqref{eqn:general:lipschitz}, strong convexity of $h_i$ ($i=1,2$), and $\eta<1/L_i$. By continuity of eigenvalues for symmetric matrices, $\lambda_{\min}(\mJ(\mu)/\mu)\to \lambda_{\min}(\bar{\mJ})>0$, so $\lambda_{\min}(\mJ(N))>0$ for some sufficiently large $N$.
Since $\mJ(\mu)$ is a real-symmetric matrix that varies continuously in $\mu$, its eigenvalues are continuous functions of $\mu$ (see, e.g., \cite[Theorem~5.1]{kato2013perturbation}).
Therefore, by the intermediate value theorem applied to $\lambda_{\min}(\mJ(\mu))$, there exists $\mu^\star\in(0,N)$ such that
$\lambda_{\min}(\mJ(\mu^\star))=0$,
i.e., $\mJ(\mu^\star)$ is singular.\end{itemize}
\end{proof}

This completes the proof that strict-saddle KKT points in $\Omega$ are unstable fixed points of the fixed-point map of \Cref{alg:badmm}.
By \Cref{thm:jason}, the set of initial points in $\Omega$ from which the iteration converges to such a strict-saddle KKT point has Lebesgue measure zero; since random initialization is absolutely continuous with respect to Lebesgue measure, the probability that the iterates converge to such a strict-saddle KKT point is zero.
\end{proof}

\section{Proof of \texorpdfstring{\Cref{thm:cadmm}}{Theorem~\ref{thm:cadmm}}}\label{sec:thm:cadmm}

\begin{proof}[Proof of \Cref{thm:cadmm}]

We apply \Cref{thm:jason} to the fixed-point map $g$ of \Cref{alg:cadmm}, verifying $\det(Dg)\neq 0$ everywhere and that every strict-saddle KKT point in $\Omega$ is an unstable fixed point.

The consensus formulation couples local primal blocks through agreement constraints, yielding a Jacobian with large but structured block form. The argument parallels the two-block case: positive definiteness of the local Bregman-proximal Hessians gives nonsingularity, and the consensus coupling cannot cancel the negative curvature at a strict saddle, so the linearized map retains an expanding direction.

\paragraph{Constructing the fixed-point map $g$.}
Define the fixed-point map of \Cref{alg:cadmm} as
\begin{align}
(\vx_{1}^{+}, \vy_{1}^{+},\cdots,\vx_{J}^{+},\vy_{J}^{+}, \vlambda_{2}^{+},\cdots,\vlambda_{J}^{+})=g(\vx_{1}, \vy_{1},\cdots,\vx_{J},\vy_{J}, \vlambda_{2},\cdots,\vlambda_{J}).
\label{def:mapping2}
\end{align}
We first write the Jacobian and instability calculation explicitly for the representative case $J=3$, which is the smallest case showing the star coupling between the hub and multiple non-hub agents. The same block argument extends to arbitrary $J\ge2$; the case $J=2$ is the simpler one-non-hub specialization. For a general number of agents, write
\[
\vz=(\vx_{1},\vy_{1},\vx_{2},\vy_{2},\ldots,\vx_{J},\vy_{J},\vlambda_{2},\ldots,\vlambda_{J}).
\]
The star-consensus constraints are $\vx_j-\vx_1=\vzero$ for $j\ge 2$, with multipliers $\vlambda_j$. Each constraint couples only the pair $(\vx_1,\vx_j)$ through $\vlambda_j$, so different non-hub agents do not interact directly: the block $(\vx_j,\vy_j,\vlambda_j)$ interacts with $(\vx_1,\vy_1)$, but there is no coupling between $(\vx_j,\vy_j,\vlambda_j)$ and $(\vx_\ell,\vy_\ell,\vlambda_\ell)$ for $j\neq \ell$, $j,\ell\ge 2$. Accordingly, $Dg(\vz)$ has the same star-shaped block sparsity as in the $J=3$ case.
All nonsingularity calculations in the $J=3$ case detailed next are agent-wise and use only strong convexity and the implicit function theorem for each local $(\vx_j,\vy_j)$-subproblem; the same calculation is repeated for $j=2,\ldots,J$. The instability construction is also unchanged: the Schur-reduced matrix has the same replicated $j$-block structure, and the general-$J$ test direction is obtained by repeating the same $\vx$-component across agents (see \Cref{rem:cadmm:generalJ} below for the explicit coupling-matrix form used in that step).
\begin{align}
(\vx_{1}^{+}, \vy_{1}^{+},\vx_{2}^{+}, \vy_{2}^{+},\vx_{3}^{+},\vy_{3}^{+}, \vlambda_{2}^{+},\vlambda_{3}^{+})=g(\vx_{1}, \vy_{1},\vx_{2}, \vy_{2},\vx_{3},\vy_{3}, \vlambda_{2},\vlambda_{3}).
\label{def:mapping}
\end{align}

\paragraph{Computing the Jacobian $Dg$ at a fixed point of \Cref{alg:cadmm}.}
Relative bi-smoothness is the positive-semidefinite inequality (Definition~\ref{defi:adaptive:lipschitz:coordinate})
\[
L_j^x\nabla^{2}_{\vx\vx}h_j(\vx,\vy)\pm \nabla^{2}_{\vx\vx}f_j(\vx,\vy)\succeq 0,\qquad
L_j^y\nabla^{2}_{\vy\vy}h_j(\vx,\vy)\pm \nabla^{2}_{\vy\vy}f_j(\vx,\vy)\succeq 0,
\]
in particular,
\[
\nabla^2_{\vx\vx} f_j(\vx,\vy)\succeq -L_j^x\nabla^2_{\vx\vx} h_j(\vx,\vy),\qquad
\nabla^2_{\vy\vy} f_j(\vx,\vy)\succeq -L_j^y\nabla^2_{\vy\vy} h_j(\vx,\vy)
\]
By assumption, each bi-Bregman kernel $h_j$ is strongly bi-convex, so $\nabla^2_{\vx\vx}h_j(\vx,\vy)\succ 0$ and $\nabla^2_{\vy\vy}h_j(\vx,\vy)\succ 0$ on the primal domain.
Hence for $\eta<\min(1/L_j^x,1/L_j^y)$,
\[
\nabla^2_{\vx\vx} f_j(\vx,\vy)+\frac{1}{\eta}\nabla^2_{\vx\vx} h_j(\vx,\vy)\succ 0,\qquad
\nabla^2_{\vy\vy} f_j(\vx,\vy)+\frac{1}{\eta}\nabla^2_{\vy\vy} h_j(\vx,\vy)\succ 0.
\]
Adding the consensus penalty terms preserves positive definiteness, so each block subproblem in \Cref{alg:cadmm} is strongly convex with a unique minimizer. For the global nonsingularity argument below, it is therefore enough to differentiate each elementary update map only with respect to its updated variable. For the instability argument, we only need the Jacobian at a fixed point. At such a point, differentiating the bi-Bregman terms produces no leftover mixed-anchor terms, because the anchor and updated variables coincide ($\vx_j^+=\vx_j$ and $\vy_j^+=\vy_j$). Thus the fixed-point Jacobian of $g$ satisfies
\begin{align}
&Dg(\vx_{1}, \vy_{1},\vx_{2}, \vy_{2},\vx_{3},\vy_{3}, \vlambda_{2},\vlambda_{3})
=
\begin{bmatrix}
\frac{\partial\vx_{1}^{+}}{\partial\vx_{1}}&
\frac{\partial\vx_{1}^{+}}{\partial\vy_{1}}&
\frac{\partial\vx_{1}^{+}}{\partial\vx_{2}}&
\frac{\partial\vx_{1}^{+}}{\partial\vy_{2}}&
\frac{\partial\vx_{1}^{+}}{\partial\vx_{3}}&
\frac{\partial\vx_{1}^{+}}{\partial\vy_{3}}&
\frac{\partial\vx_{1}^{+}}{\partial\vlambda_{2}}&
\frac{\partial\vx_{1}^{+}}{\partial\vlambda_{3}}
\\
\frac{\partial\vy_{1}^{+}}{\partial\vx_{1}}&
\frac{\partial\vy_{1}^{+}}{\partial\vy_{1}}&
\frac{\partial\vy_{1}^{+}}{\partial\vx_{2}}&
\frac{\partial\vy_{1}^{+}}{\partial\vy_{2}}&
\frac{\partial\vy_{1}^{+}}{\partial\vx_{3}}&
\frac{\partial\vy_{1}^{+}}{\partial\vy_{3}}&
\frac{\partial\vy_{1}^{+}}{\partial\vlambda_{2}}&
\frac{\partial\vy_{1}^{+}}{\partial\vlambda_{3}}
\\
\frac{\partial\vx_{2}^{+}}{\partial\vx_{1}}&
\frac{\partial\vx_{2}^{+}}{\partial\vy_{1}}&
\frac{\partial\vx_{2}^{+}}{\partial\vx_{2}}&
\frac{\partial\vx_{2}^{+}}{\partial\vy_{2}}&
\frac{\partial\vx_{2}^{+}}{\partial\vx_{3}}&
\frac{\partial\vx_{2}^{+}}{\partial\vy_{3}}&
\frac{\partial\vx_{2}^{+}}{\partial\vlambda_{2}}&
\frac{\partial\vx_{2}^{+}}{\partial\vlambda_{3}}
\\
\frac{\partial\vy_{2}^{+}}{\partial\vx_{1}}&
\frac{\partial\vy_{2}^{+}}{\partial\vy_{1}}&
\frac{\partial\vy_{2}^{+}}{\partial\vx_{2}}&
\frac{\partial\vy_{2}^{+}}{\partial\vy_{2}}&
\frac{\partial\vy_{2}^{+}}{\partial\vx_{3}}&
\frac{\partial\vy_{2}^{+}}{\partial\vy_{3}}&
\frac{\partial\vy_{2}^{+}}{\partial\vlambda_{2}}&
\frac{\partial\vy_{2}^{+}}{\partial\vlambda_{3}}
\\
\frac{\partial\vx_{3}^{+}}{\partial\vx_{1}}&
\frac{\partial\vx_{3}^{+}}{\partial\vy_{1}}&
\frac{\partial\vx_{3}^{+}}{\partial\vx_{2}}&
\frac{\partial\vx_{3}^{+}}{\partial\vy_{2}}&
\frac{\partial\vx_{3}^{+}}{\partial\vx_{3}}&
\frac{\partial\vx_{3}^{+}}{\partial\vy_{3}}&
\frac{\partial\vx_{3}^{+}}{\partial\vlambda_{2}}&
\frac{\partial\vx_{3}^{+}}{\partial\vlambda_{3}}
\\
\frac{\partial\vy_{3}^{+}}{\partial\vx_{1}}&
\frac{\partial\vy_{3}^{+}}{\partial\vy_{1}}&
\frac{\partial\vy_{3}^{+}}{\partial\vx_{2}}&
\frac{\partial\vy_{3}^{+}}{\partial\vy_{2}}&
\frac{\partial\vy_{3}^{+}}{\partial\vx_{3}}&
\frac{\partial\vy_{3}^{+}}{\partial\vy_{3}}&
\frac{\partial\vy_{3}^{+}}{\partial\vlambda_{2}}&
\frac{\partial\vy_{3}^{+}}{\partial\vlambda_{3}}
\\
\frac{\partial\lambda_{2}^{+}}{\partial\vx_{1}}&
\frac{\partial\lambda_{2}^{+}}{\partial\vy_{1}}&
\frac{\partial\lambda_{2}^{+}}{\partial\vx_{2}}&
\frac{\partial\lambda_{2}^{+}}{\partial\vy_{2}}&
\frac{\partial\lambda_{2}^{+}}{\partial\vx_{3}}&
\frac{\partial\lambda_{2}^{+}}{\partial\vy_{3}}&
\frac{\partial\lambda_{2}^{+}}{\partial\vlambda_{2}}&
\frac{\partial\lambda_{2}^{+}}{\partial\vlambda_{3}}
\\
\frac{\partial\lambda_{3}^{+}}{\partial\vx_{1}}&
\frac{\partial\lambda_{3}^{+}}{\partial\vy_{1}}&
\frac{\partial\lambda_{3}^{+}}{\partial\vx_{2}}&
\frac{\partial\lambda_{3}^{+}}{\partial\vy_{2}}&
\frac{\partial\lambda_{3}^{+}}{\partial\vx_{3}}&
\frac{\partial\lambda_{3}^{+}}{\partial\vy_{3}}&
\frac{\partial\lambda_{3}^{+}}{\partial\vlambda_{2}}&
\frac{\partial\lambda_{3}^{+}}{\partial\vlambda_{3}}
\end{bmatrix}.
\label{def:Dg}
\end{align}

At a fixed point, the optimality conditions of \eqref{eqn:algorithm} and $\calL(\cdot)$ in \eqref{eqn:Lagrange} read:
\begin{align}
\nabla_{\vx}f_{1}(\vx_{1}^{+},\vy_{1})-\vlambda_{2}-\vlambda_{3}+\rho(2\vx_{1}^{+}-\vx_{2}-\vx_{3})+\frac{1}{\eta}(\nabla_{\vx} h_{1}(\vx_{1}^{+},\vy_1)-\nabla_{\vx} h_{1}(\vx_{1},\vy_1))&=\vzero\label{eqn:optu1}
\\
\nabla_{\vy}f_{1}(\vx_{1}^{+},\vy_{1}^{+})+\frac{1}{\eta}(\nabla_{\vy} h_{1}(\vx_1^+,\vy_{1}^{+})-\nabla_{\vy} h_{1}(\vx_1^+,\vy_{1}))&=\vzero\label{eqn:optv1}
\\
\nabla_{\vx}f_{2}(\vx_{2}^{+},\vy_{2})+\vlambda_{2}+\rho(\vx_{2}^{+}-\vx_{1}^{+})+\frac{1}{\eta}(\nabla_{\vx} h_2(\vx_{2}^{+},\vy_2)-\nabla_{\vx} h_{2}(\vx_{2},\vy_2))&=\vzero\label{eqn:optu2}
\\
\nabla_{\vy}f_{2}(\vx_{2}^{+},\vy_{2}^{+})+\frac{1}{\eta}(\nabla_{\vy} h_{2}(\vx_2^+,\vy_{2}^{+})-\nabla_{\vy} h_{2}(\vx_2^+,\vy_{2}))&=\vzero\label{eqn:optv2}
\\
\nabla_{\vx}f_{3}(\vx_{3}^{+},\vy_{3})+\vlambda_{3}+\rho(\vx_{3}^{+}-\vx_{1}^{+})+\frac{1}{\eta}(\nabla_{\vx} h_{3}(\vx_{3}^{+},\vy_3)-\nabla_{\vx} h_{3}(\vx_{3},\vy_3))&=\vzero\label{eqn:optu3}
\\
\nabla_{\vy}f_{3}(\vx_{3}^{+},\vy_{3}^{+})+\frac{1}{\eta}(\nabla_{\vy} h_{3}(\vx_3^+,\vy_{3}^{+})-\nabla_{\vy} h_{3}(\vx_3^+,\vy_{3}))&=\vzero\label{eqn:optv3}
\\
- \vlambda_{2}^{+} +\vlambda_{2}  +\rho(\vx_{2}^{+}-\vx_{1}^{+} ) &=\vzero\label{eqn:optlam2}
 \\
- \vlambda_{3}^{+} +\vlambda_{3}  +\rho(\vx_{3}^{+}-\vx_{1}^{+} ) &=\vzero\label{eqn:optlam3}
\end{align}

Before applying the implicit function theorem, to simplify notations, we define
\begin{align}
\mGamma_{\vx\vx}^{1}&\doteq\nabla_{\vx\vx}^{2}f_1(\vx_{1}^{+},\vy_{1})+2\rho\eye+\frac{1}{\eta}\nabla_{\vx\vx}^{2}h_{1}(\vx_{1}^{+},\vy_1),& 	\mF_{\vx\vy}^{1}&\doteq \nabla_{\vx\vy}^{2}f_{1}(\vx_{1}^{+},\vy_{1}), &\mH_{\vx\vx}^{1}&\doteq\nabla^{2}_{\vx\vx}h_{1}(\vx_{1},\vy_1),
\nn\\
\mGamma_{\vy\vy}^{1}&\doteq\nabla_{\vy\vy}^{2}f_{1}(\vx_{1}^{+},\vy^{+}_{1})+\frac{1}{\eta}\nabla_{\vy\vy}^{2}h_{1}(\vx_{1}^{+},\vy_{1}^{+}),& \mF_{\vy\vx}^{1}&\doteq\nabla_{\vy\vx}^{2}f_{1}(\vx_{1}^{+},\vy^{+}_{1}), &\mH_{\vy\vy}^{1}&\doteq\nabla^{2}_{\vy\vy} h_{1}(\vx_{1}^{+},\vy_{1}),
\nn\\
\mGamma_{\vx\vx}^{2}&\doteq\nabla_{\vx\vx}^{2}f_{2}(\vx_{2}^{+},\vy_{2})+\rho\eye+\frac{1}{\eta}\nabla^{2}_{\vx\vx}h_{2}(\vx_{2}^{+},\vy_{2}),& \mF_{\vx\vy}^{2}&\doteq\nabla_{\vx\vy}^{2}f_{2}(\vx_{2}^{+},\vy_{2}), &\mH_{\vx\vx}^{2}&\doteq \nabla^{2}_{\vx\vx}h_{2}(\vx_{2},\vy_{2}),
\nn\\
\mGamma_{\vy\vy}^{2}&\doteq \nabla_{\vy\vy}^{2}f_{2}(\vx_{2}^{+},\vy^{+}_{2})+\frac{1}{\eta}\nabla^{2}_{\vy\vy}h_{2}(\vx_{2}^{+},\vy_{2}^{+}), &\mF_{\vy\vx}^{2}&\doteq \nabla_{\vy\vx}^{2}f_{2}(\vx_{2}^{+},\vy^{+}_{2}), &\mH_{\vy\vy}^{2}&\doteq \nabla^{2}_{\vy\vy}h_{2}(\vx_{2}^{+},\vy_{2}),
\nn\\
\mGamma_{\vx\vx}^{3}&\doteq \nabla_{\vx\vx}^{2}f_{3}(\vx_{3}^{+},\vy_{3})+\rho\eye+\frac{1}{\eta}\nabla^{2}_{\vx\vx}h_{3}(\vx_{3}^{+},\vy_{3}), &\mF_{\vx\vy}^{3}&\doteq \nabla_{\vx\vy}^{2}f_{3}(\vx_{3}^{+},\vy_{3}), & \mH_{\vx\vx}^{3}&\doteq \nabla^{2}_{\vx\vx}h_{3}(\vx_{3},\vy_{3}),
\nn\\
\mGamma_{\vy\vy}^{3} &\doteq \nabla_{\vy\vy}^{2}f_{3}(\vx_{3}^{+},\vy^{+}_{3})+\frac{1}{\eta}\nabla^{2}_{\vy\vy}h_{3}(\vx_{3}^{+},\vy_{3}^{+}), & \mF_{\vy\vx}^{3} &\doteq \nabla_{\vy\vx}^{2}f_{3}(\vx_{3}^{+},\vy^{+}_{3}), & \mH_{\vy\vy}^{3} &\doteq \nabla^{2}_{\vy\vy} h_{3}(\vx_{3}^{+},\vy_{3}).
\nn\\
\mF_{\vx\vx}^{j}&\doteq\nabla_{\vx\vx}^{2}f_{j}(\vx_{j}^{+},\vy_{j}),& \mF_{\vy\vy}^{j}&\doteq\nabla_{\vy\vy}^{2}f_{j}(\vx_{j}^{+},\vy_{j}^{+}),  & \forall j &=1,2,3. &&
\label{definition:notations:1}
\end{align}
Note that $\mH_{\vx\vx}^{j}=\nabla^{2}_{\vx\vx}h_{j}(\vx_{j},\vy_{j})$ is evaluated at the \emph{old} iterate $(\vx_j,\vy_j)$, whereas $\mH_{\vy\vy}^{j}=\nabla^{2}_{\vy\vy}h_{j}(\vx_{j}^{+},\vy_{j})$ is evaluated at $(\vx_j^{+},\vy_j)$, i.e., with the \emph{already-updated} $\vx_j^{+}$. This asymmetry reflects the Gauss--Seidel ordering in \Cref{alg:cadmm}: the $\vx_j$-subproblem uses the bi-Bregman divergence $D^{x}_{h_j}(\vx_j,\vx_j^{k-1};\vy_j^{k-1})$ whose anchor point is $(\vx_j^{k-1},\vy_j^{k-1})=(\vx_j,\vy_j)$, while the subsequent $\vy_j$-subproblem uses $D^{y}_{h_j}(\vy_j,\vy_j^{k-1};\vx_j^{k})$ whose anchor point is $(\vx_j^{k},\vy_j^{k-1})=(\vx_j^{+},\vy_j)$.

\bigskip

\noindent
\begin{minipage}[t]{0.31\textwidth}
\footnotesize
\textbf{Apply implicit function theorem to \eqref{eqn:optu1}:}
\begin{align}
\mGamma_{\vx\vx}^{1}\frac{\partial\vx_{1}^{+}}{\partial\vx_{1}}&=\frac{1}{\eta}\mH_{\vx\vx}^{1}
\nn
\\
\mGamma_{\vx\vx}^{1}\frac{\partial\vx_{1}^{+}}{\partial\vy_{1}}&=-\mF_{\vx\vy}^{1}
\nn
\\
\mGamma_{\vx\vx}^{1}\frac{\partial\vx_{1}^{+}}{\partial\vx_{2}}&=\rho\eye
\nn
\\
\mGamma_{\vx\vx}^{1}\frac{\partial\vx_{1}^{+}}{\partial\vy_{2}}&=\mzero
\nn
\\
\mGamma_{\vx\vx}^{1}\frac{\partial\vx_{1}^{+}}{\partial\vx_{3}}&=\rho\eye
\nn
\\
\mGamma_{\vx\vx}^{1}\frac{\partial\vx_{1}^{+}}{\partial\vy_{3}}&=\mzero
\nn
\\
\mGamma_{\vx\vx}^{1}\frac{\partial\vx_{1}^{+}}{\partial\vlambda_{2}}&=\eye
\nn
\\
\mGamma_{\vx\vx}^{1}\frac{\partial\vx_{1}^{+}}{\partial\vlambda_{3}}&=\eye
\label{eqn:u1lam3}
\end{align}
\end{minipage}\hfill
\begin{minipage}[t]{0.31\textwidth}
\footnotesize
\textbf{Apply implicit function theorem to \eqref{eqn:optv1}:}
\begin{align}
\mF_{\vy\vx}^{1}\frac{\partial\vx_{1}^{+}}{\partial\vx_{1}}+
\mGamma_{\vy\vy}^{1}\frac{\partial\vy_{1}^{+}}{\partial\vx_{1}}&=\mzero
\nn
\\
\mF_{\vy\vx}^{1}\frac{\partial\vx_{1}^{+}}{\partial\vy_{1}}+
\mGamma_{\vy\vy}^{1}\frac{\partial\vy_{1}^{+}}{\partial\vy_{1}}&=\frac{\mH_{\vy\vy}^{1}}{\eta}
\nn
\\
\mF_{\vy\vx}^{1}\frac{\partial\vx_{1}^{+}}{\partial\vx_{2}}+
\mGamma_{\vy\vy}^{1}\frac{\partial\vy_{1}^{+}}{\partial\vx_{2}}&=\mzero
\nn
\\
\mF_{\vy\vx}^{1}\frac{\partial\vx_{1}^{+}}{\partial\vy_{2}}+
\mGamma_{\vy\vy}^{1}\frac{\partial\vy_{1}^{+}}{\partial\vy_{2}}&=\mzero
\nn
\\
\mF_{\vy\vx}^{1}\frac{\partial\vx_{1}^{+}}{\partial\vx_{3}}+
\mGamma_{\vy\vy}^{1}\frac{\partial\vy_{1}^{+}}{\partial\vx_{3}}&=\mzero
\nn
\\
\mF_{\vy\vx}^{1}\frac{\partial\vx_{1}^{+}}{\partial\vy_{3}}+
\mGamma_{\vy\vy}^{1}\frac{\partial\vy_{1}^{+}}{\partial\vy_{3}}&=\mzero
\nn
\\
\mF_{\vy\vx}^{1}\frac{\partial\vx_{1}^{+}}{\partial\vlambda_{2}}+
\mGamma_{\vy\vy}^{1}\frac{\partial\vy_{1}^{+}}{\partial\vlambda_{2}}&=\mzero
\nn
\\
\mF_{\vy\vx}^{1}\frac{\partial\vx_{1}^{+}}{\partial\vlambda_{3}}+
\mGamma_{\vy\vy}^{1}\frac{\partial\vy_{1}^{+}}{\partial\vlambda_{3}}&=\mzero
\label{eqn:v1lam3}
\end{align}
\end{minipage}\hfill
\begin{minipage}[t]{0.31\textwidth}
\footnotesize
\textbf{Apply implicit function theorem to \eqref{eqn:optu2}:}
\begin{align}
-\rho \frac{\partial \vx_{1}^{+}}{\partial\vx_{1}}+\mGamma_{\vx\vx}^{2}\frac{\partial \vx_{2}^{+}}{\partial\vx_{1}}&=\mzero
\nn
\\
-\rho \frac{\partial \vx_{1}^{+}}{\partial\vy_{1}}+\mGamma_{\vx\vx}^{2}\frac{\partial \vx_{2}^{+}}{\partial\vy_{1}}&=\mzero
\nn
\\
-\rho \frac{\partial \vx_{1}^{+}}{\partial\vx_{2}}+\mGamma_{\vx\vx}^{2}\frac{\partial \vx_{2}^{+}}{\partial\vx_{2}}&=\frac{\mH_{\vx\vx}^{2}}{\eta}
\nn
\\
-\rho \frac{\partial \vx_{1}^{+}}{\partial\vy_{2}}+\mGamma_{\vx\vx}^{2}\frac{\partial \vx_{2}^{+}}{\partial\vy_{2}}&=-\mF_{\vx\vy}^{2}
\nn
\\
-\rho \frac{\partial \vx_{1}^{+}}{\partial\vx_{3}}+\mGamma_{\vx\vx}^{2}\frac{\partial \vx_{2}^{+}}{\partial\vx_{3}}&=\mzero
\nn
\\
-\rho \frac{\partial \vx_{1}^{+}}{\partial\vy_{3}}+\mGamma_{\vx\vx}^{2}\frac{\partial \vx_{2}^{+}}{\partial\vy_{3}}&=\mzero
\nn
\\
-\rho \frac{\partial \vx_{1}^{+}}{\partial\vlambda_{2}}+\mGamma_{\vx\vx}^{2}\frac{\partial \vx_{2}^{+}}{\partial\vlambda_{2}}&=-\eye
\nn
\\
-\rho \frac{\partial \vx_{1}^{+}}{\partial\vlambda_{3}}+\mGamma_{\vx\vx}^{2}\frac{\partial \vx_{2}^{+}}{\partial\vlambda_{3}}&=\mzero.
\label{eqn:u2lam3}
\end{align}
\end{minipage}

\bigskip

\noindent
\begin{minipage}[t]{0.31\textwidth}
\footnotesize
\textbf{Apply implicit function theorem to \eqref{eqn:optv2}:}
\begin{align}
\mF_{\vy\vx}^{2}\frac{\partial\vx_{2}^{+}}{\partial\vx_{1}}+
\mGamma_{\vy\vy}^{2}\frac{\partial\vy_{2}^{+}}{\partial\vx_{1}}&=\mzero
\nn
\\
\mF_{\vy\vx}^{2}\frac{\partial\vx_{2}^{+}}{\partial\vy_{1}}+
\mGamma_{\vy\vy}^{2}\frac{\partial\vy_{2}^{+}}{\partial\vy_{1}}&=\mzero
\nn
\\
\mF_{\vy\vx}^{2}\frac{\partial\vx_{2}^{+}}{\partial\vx_{2}}+
\mGamma_{\vy\vy}^{2}\frac{\partial\vy_{2}^{+}}{\partial\vx_{2}}&=\mzero
\nn
\\
\mF_{\vy\vx}^{2}\frac{\partial\vx_{2}^{+}}{\partial\vy_{2}}+
\mGamma_{\vy\vy}^{2}\frac{\partial\vy_{2}^{+}}{\partial\vy_{2}}&=\frac{\mH_{\vy\vy}^{2}}{\eta}
\nn
\\
\mF_{\vy\vx}^{2}\frac{\partial\vx_{2}^{+}}{\partial\vx_{3}}+
\mGamma_{\vy\vy}^{2}\frac{\partial\vy_{2}^{+}}{\partial\vx_{3}}&=\mzero
\nn
\\
\mF_{\vy\vx}^{2}\frac{\partial\vx_{2}^{+}}{\partial\vy_{3}}+
\mGamma_{\vy\vy}^{2}\frac{\partial\vy_{2}^{+}}{\partial\vy_{3}}&=\mzero
\nn
\\
\mF_{\vy\vx}^{2}\frac{\partial\vx_{2}^{+}}{\partial\vlambda_{2}}+
\mGamma_{\vy\vy}^{2}\frac{\partial\vy_{2}^{+}}{\partial\vlambda_{2}}&=\mzero
\nn
\\
\mF_{\vy\vx}^{2}\frac{\partial\vx_{2}^{+}}{\partial\vlambda_{3}}+
\mGamma_{\vy\vy}^{2}\frac{\partial\vy_{2}^{+}}{\partial\vlambda_{3}}&=\mzero.
\label{eqn:v2lam3}
\end{align}
\end{minipage}\hfill
\begin{minipage}[t]{0.31\textwidth}
\footnotesize
\textbf{Apply implicit function theorem to \eqref{eqn:optu3}:}
\begin{align}
-\rho \frac{\partial \vx_{1}^{+}}{\partial\vx_{1}}+\mGamma_{\vx\vx}^{3}\frac{\partial \vx_{3}^{+}}{\partial\vx_{1}}&=\mzero
\nn
\\
-\rho \frac{\partial \vx_{1}^{+}}{\partial\vy_{1}}+\mGamma_{\vx\vx}^{3}\frac{\partial \vx_{3}^{+}}{\partial\vy_{1}}&=\mzero
\nn
\\
-\rho \frac{\partial \vx_{1}^{+}}{\partial\vx_{2}}+\mGamma_{\vx\vx}^{3}\frac{\partial \vx_{3}^{+}}{\partial\vx_{2}}&=\mzero
\nn
\\
-\rho \frac{\partial \vx_{1}^{+}}{\partial\vy_{2}}+\mGamma_{\vx\vx}^{3}\frac{\partial \vx_{3}^{+}}{\partial\vy_{2}}&= \mzero
\nn
\\
-\rho \frac{\partial \vx_{1}^{+}}{\partial\vx_{3}}+\mGamma_{\vx\vx}^{3}\frac{\partial \vx_{3}^{+}}{\partial\vx_{3}}&=\frac{\mH_{\vx\vx}^{3}}{\eta}
\nn
\\
-\rho \frac{\partial \vx_{1}^{+}}{\partial\vy_{3}}+\mGamma_{\vx\vx}^{3}\frac{\partial \vx_{3}^{+}}{\partial\vy_{3}}&=-\mF_{\vx\vy}^{3}
\nn
\\
-\rho \frac{\partial \vx_{1}^{+}}{\partial\vlambda_{2}}+\mGamma_{\vx\vx}^{3}\frac{\partial \vx_{3}^{+}}{\partial\vlambda_{2}}&=\mzero
\nn
\\
-\rho \frac{\partial \vx_{1}^{+}}{\partial\vlambda_{3}}+\mGamma_{\vx\vx}^{3}\frac{\partial \vx_{3}^{+}}{\partial\vlambda_{3}}&=-\eye.\label{eqn:u3lam3}
\end{align}
\end{minipage}\hfill
\begin{minipage}[t]{0.31\textwidth}
\footnotesize
\textbf{Apply implicit function theorem to \eqref{eqn:optv3}:}
\begin{align}
\mF_{\vy\vx}^{3}\frac{\partial\vx_{3}^{+}}{\partial\vx_{1}}+
\mGamma_{\vy\vy}^{3}\frac{\partial\vy_{3}^{+}}{\partial\vx_{1}}&=\mzero
\nn
\\
\mF_{\vy\vx}^{3}\frac{\partial\vx_{3}^{+}}{\partial\vy_{1}}+
\mGamma_{\vy\vy}^{3}\frac{\partial\vy_{3}^{+}}{\partial\vy_{1}}&=\mzero
\nn
\\
\mF_{\vy\vx}^{3}\frac{\partial\vx_{3}^{+}}{\partial\vx_{2}}+
\mGamma_{\vy\vy}^{3}\frac{\partial\vy_{3}^{+}}{\partial\vx_{2}}&=\mzero
\nn
\\
\mF_{\vy\vx}^{3}\frac{\partial\vx_{3}^{+}}{\partial\vy_{2}}+
\mGamma_{\vy\vy}^{3}\frac{\partial\vy_{3}^{+}}{\partial\vy_{2}}&=\mzero
\nn
\\
\mF_{\vy\vx}^{3}\frac{\partial\vx_{3}^{+}}{\partial\vx_{3}}+
\mGamma_{\vy\vy}^{3}\frac{\partial\vy_{3}^{+}}{\partial\vx_{3}}&=\mzero
\nn
\\
\mF_{\vy\vx}^{3}\frac{\partial\vx_{3}^{+}}{\partial\vy_{3}}+
\mGamma_{\vy\vy}^{3}\frac{\partial\vy_{3}^{+}}{\partial\vy_{3}}&=\frac{\mH_{\vy\vy}^{3}}{\eta}
\nn
\\
\mF_{\vy\vx}^{3}\frac{\partial\vx_{3}^{+}}{\partial\vlambda_{2}}+
\mGamma_{\vy\vy}^{3}\frac{\partial\vy_{3}^{+}}{\partial\vlambda_{2}}&=\mzero
\nn
\\
\mF_{\vy\vx}^{3}\frac{\partial\vx_{3}^{+}}{\partial\vlambda_{3}}+
\mGamma_{\vy\vy}^{3}\frac{\partial\vy_{3}^{+}}{\partial\vlambda_{3}}&=\mzero.
\label{eqn:v3lam3}
\end{align}
\end{minipage}

\bigskip

\noindent
\begin{minipage}[t]{0.31\textwidth}
\footnotesize
\textbf{Apply implicit function theorem to \eqref{eqn:optlam2}:}
\begin{align}
-\frac{\partial\vlambda_{2}^{+}}{\partial\vx_{1}}+\rho\frac{\partial\vx_{2}^{+}}{\partial\vx_{1}}-\rho\frac{\partial\vx_{1}^{+}}{\partial\vx_{1}}&=\mzero
\nn
\\
-\frac{\partial\vlambda_{2}^{+}}{\partial\vy_{1}}+\rho\frac{\partial\vx_{2}^{+}}{\partial\vy_{1}}-\rho\frac{\partial\vx_{1}^{+}}{\partial\vy_{1}}&=\mzero
\nn
\\
-\frac{\partial\vlambda_{2}^{+}}{\partial\vx_{2}}+\rho\frac{\partial\vx_{2}^{+}}{\partial\vx_{2}}-\rho\frac{\partial\vx_{1}^{+}}{\partial\vx_{2}}&=\mzero
\nn
\\
-\frac{\partial\vlambda_{2}^{+}}{\partial\vy_{2}}+\rho\frac{\partial\vx_{2}^{+}}{\partial\vy_{2}}-\rho\frac{\partial\vx_{1}^{+}}{\partial\vy_{2}}&=\mzero
\nn
\\
-\frac{\partial\vlambda_{2}^{+}}{\partial\vx_{3}}+\rho\frac{\partial\vx_{2}^{+}}{\partial\vx_{3}}-\rho\frac{\partial\vx_{1}^{+}}{\partial\vx_{3}}&=\mzero
\nn
\\
-\frac{\partial\vlambda_{2}^{+}}{\partial\vy_{3}}+\rho\frac{\partial\vx_{2}^{+}}{\partial\vy_{3}}-\rho\frac{\partial\vx_{1}^{+}}{\partial\vy_{3}}&=\mzero
\nn
\\
-\frac{\partial\vlambda_{2}^{+}}{\partial\vlambda_{2}}+\rho\frac{\partial\vx_{2}^{+}}{\partial\vlambda_{2}}-\rho\frac{\partial\vx_{1}^{+}}{\partial\vlambda_{2}}&=-\eye
\nn
\\
-\frac{\partial\vlambda_{2}^{+}}{\partial\vlambda_{3}}+\rho\frac{\partial\vx_{2}^{+}}{\partial\vlambda_{3}}-\rho\frac{\partial\vx_{1}^{+}}{\partial\vlambda_{3}}&=\mzero.
\label{eqn:lam2lam3}
\end{align}
\end{minipage}\hfill
\begin{minipage}[t]{0.31\textwidth}
\footnotesize
\textbf{Apply implicit function theorem to \eqref{eqn:optlam3}:}
\begin{align}
-\frac{\partial\vlambda_{3}^{+}}{\partial\vx_{1}}+\rho\frac{\partial\vx_{3}^{+}}{\partial\vx_{1}}-\rho\frac{\partial\vx_{1}^{+}}{\partial\vx_{1}}&=\mzero
\nn
\\
-\frac{\partial\vlambda_{3}^{+}}{\partial\vy_{1}}+\rho\frac{\partial\vx_{3}^{+}}{\partial\vy_{1}}-\rho\frac{\partial\vx_{1}^{+}}{\partial\vy_{1}}&=\mzero
\nn
\\
-\frac{\partial\vlambda_{3}^{+}}{\partial\vx_{2}}+\rho\frac{\partial\vx_{3}^{+}}{\partial\vx_{2}}-\rho\frac{\partial\vx_{1}^{+}}{\partial\vx_{2}}&=\mzero
\nn
\\
-\frac{\partial\vlambda_{3}^{+}}{\partial\vy_{2}}+\rho\frac{\partial\vx_{3}^{+}}{\partial\vy_{2}}-\rho\frac{\partial\vx_{1}^{+}}{\partial\vy_{2}}&=\mzero
\nn
\\
-\frac{\partial\vlambda_{3}^{+}}{\partial\vx_{3}}+\rho\frac{\partial\vx_{3}^{+}}{\partial\vx_{3}}-\rho\frac{\partial\vx_{1}^{+}}{\partial\vx_{3}}&=\mzero
\nn
\\
-\frac{\partial\vlambda_{3}^{+}}{\partial\vy_{3}}+\rho\frac{\partial\vx_{3}^{+}}{\partial\vy_{3}}-\rho\frac{\partial\vx_{1}^{+}}{\partial\vy_{3}}&=\mzero
\nn
\\
-\frac{\partial\vlambda_{3}^{+}}{\partial\vlambda_{2}}+\rho\frac{\partial\vx_{3}^{+}}{\partial\vlambda_{2}}-\rho\frac{\partial\vx_{1}^{+}}{\partial\vlambda_{2}}&=\mzero
\nn
\\
-\frac{\partial\vlambda_{3}^{+}}{\partial\vlambda_{3}}+\rho\frac{\partial\vx_{3}^{+}}{\partial\vlambda_{3}}-\rho\frac{\partial\vx_{1}^{+}}{\partial\vlambda_{3}}&=-\eye.
\label{eqn:lam3lam3}
\end{align}
\end{minipage}

\bigskip

Representing \eqref{eqn:u1lam3}-\eqref{eqn:lam3lam3} in matrix form, we further get
\begin{align}
&
\begin{bmatrix}
\mGamma_{\vx\vx}^{1}&
\mzero&
\mzero&
\mzero&
\mzero&
\mzero&
\mzero&
\mzero
\\
\mF_{\vy\vx}^{1}&
\mGamma_{\vy\vy}^{1}&
\mzero&
\mzero&
\mzero&
\mzero&
\mzero&
\mzero
\\
-\rho\eye&
\mzero&
\mGamma_{\vx\vx}^{2}&
\mzero&
\mzero&
\mzero&
\mzero&
\mzero
\\
\mzero&
\mzero&
\mF_{\vy\vx}^{2}&
\mGamma_{\vy\vy}^{2}&
\mzero&
\mzero&
\mzero&
\mzero
\\
-\rho\eye&
\mzero&
\mzero&
\mzero&
\mGamma_{\vx\vx}^{3}&
\mzero&
\mzero&
\mzero
\\
\mzero&
\mzero&
\mzero&
\mzero&
\mF_{\vy\vx}^{3}&
\mGamma_{\vy\vy}^{3}&
\mzero&
\mzero
\\
-\rho\eye&
\mzero&
\rho\eye&
\mzero&
\mzero&
\mzero&
-\eye&
\mzero
\\
-\rho\eye&
\mzero&
\mzero&
\mzero&
\rho\eye&
\mzero&
\mzero&
-\eye
\end{bmatrix}
\begin{bmatrix}
\frac{\partial\vx_{1}^{+}}{\partial\vx_{1}}&
\frac{\partial\vx_{1}^{+}}{\partial\vy_{1}}&
\frac{\partial\vx_{1}^{+}}{\partial\vx_{2}}&
\frac{\partial\vx_{1}^{+}}{\partial\vy_{2}}&
\frac{\partial\vx_{1}^{+}}{\partial\vx_{3}}&
\frac{\partial\vx_{1}^{+}}{\partial\vy_{3}}&
\frac{\partial\vx_{1}^{+}}{\partial\vlambda_{2}}&
\frac{\partial\vx_{1}^{+}}{\partial\vlambda_{3}}
\\
\frac{\partial\vy_{1}^{+}}{\partial\vx_{1}}&
\frac{\partial\vy_{1}^{+}}{\partial\vy_{1}}&
\frac{\partial\vy_{1}^{+}}{\partial\vx_{2}}&
\frac{\partial\vy_{1}^{+}}{\partial\vy_{2}}&
\frac{\partial\vy_{1}^{+}}{\partial\vx_{3}}&
\frac{\partial\vy_{1}^{+}}{\partial\vy_{3}}&
\frac{\partial\vy_{1}^{+}}{\partial\vlambda_{2}}&
\frac{\partial\vy_{1}^{+}}{\partial\vlambda_{3}}
\\
\frac{\partial\vx_{2}^{+}}{\partial\vx_{1}}&
\frac{\partial\vx_{2}^{+}}{\partial\vy_{1}}&
\frac{\partial\vx_{2}^{+}}{\partial\vx_{2}}&
\frac{\partial\vx_{2}^{+}}{\partial\vy_{2}}&
\frac{\partial\vx_{2}^{+}}{\partial\vx_{3}}&
\frac{\partial\vx_{2}^{+}}{\partial\vy_{3}}&
\frac{\partial\vx_{2}^{+}}{\partial\vlambda_{2}}&
\frac{\partial\vx_{2}^{+}}{\partial\vlambda_{3}}
\\
\frac{\partial\vy_{2}^{+}}{\partial\vx_{1}}&
\frac{\partial\vy_{2}^{+}}{\partial\vy_{1}}&
\frac{\partial\vy_{2}^{+}}{\partial\vx_{2}}&
\frac{\partial\vy_{2}^{+}}{\partial\vy_{2}}&
\frac{\partial\vy_{2}^{+}}{\partial\vx_{3}}&
\frac{\partial\vy_{2}^{+}}{\partial\vy_{3}}&
\frac{\partial\vy_{2}^{+}}{\partial\vlambda_{2}}&
\frac{\partial\vy_{2}^{+}}{\partial\vlambda_{3}}
\\
\frac{\partial\vx_{3}^{+}}{\partial\vx_{1}}&
\frac{\partial\vx_{3}^{+}}{\partial\vy_{1}}&
\frac{\partial\vx_{3}^{+}}{\partial\vx_{2}}&
\frac{\partial\vx_{3}^{+}}{\partial\vy_{2}}&
\frac{\partial\vx_{3}^{+}}{\partial\vx_{3}}&
\frac{\partial\vx_{3}^{+}}{\partial\vy_{3}}&
\frac{\partial\vx_{3}^{+}}{\partial\vlambda_{2}}&
\frac{\partial\vx_{3}^{+}}{\partial\vlambda_{3}}
\\
\frac{\partial\vy_{3}^{+}}{\partial\vx_{1}}&
\frac{\partial\vy_{3}^{+}}{\partial\vy_{1}}&
\frac{\partial\vy_{3}^{+}}{\partial\vx_{2}}&
\frac{\partial\vy_{3}^{+}}{\partial\vy_{2}}&
\frac{\partial\vy_{3}^{+}}{\partial\vx_{3}}&
\frac{\partial\vy_{3}^{+}}{\partial\vy_{3}}&
\frac{\partial\vy_{3}^{+}}{\partial\vlambda_{2}}&
\frac{\partial\vy_{3}^{+}}{\partial\vlambda_{3}}
\\
\frac{\partial\lambda_{2}^{+}}{\partial\vx_{1}}&
\frac{\partial\lambda_{2}^{+}}{\partial\vy_{1}}&
\frac{\partial\lambda_{2}^{+}}{\partial\vx_{2}}&
\frac{\partial\lambda_{2}^{+}}{\partial\vy_{2}}&
\frac{\partial\lambda_{2}^{+}}{\partial\vx_{3}}&
\frac{\partial\lambda_{2}^{+}}{\partial\vy_{3}}&
\frac{\partial\lambda_{2}^{+}}{\partial\vlambda_{2}}&
\frac{\partial\lambda_{2}^{+}}{\partial\vlambda_{3}}
\\
\frac{\partial\lambda_{3}^{+}}{\partial\vx_{1}}&
\frac{\partial\lambda_{3}^{+}}{\partial\vy_{1}}&
\frac{\partial\lambda_{3}^{+}}{\partial\vx_{2}}&
\frac{\partial\lambda_{3}^{+}}{\partial\vy_{2}}&
\frac{\partial\lambda_{3}^{+}}{\partial\vx_{3}}&
\frac{\partial\lambda_{3}^{+}}{\partial\vy_{3}}&
\frac{\partial\lambda_{3}^{+}}{\partial\vlambda_{2}}&
\frac{\partial\lambda_{3}^{+}}{\partial\vlambda_{3}}
\end{bmatrix}
\nonumber\\
=&
\begin{bmatrix}
\frac{\mH_{\vx\vx}^{1}}{\eta}&
-\mF_{\vx\vy}^{1}&
\rho\eye&
\mzero&
\rho\eye&
\mzero&
\eye&
\eye
\\
\mzero&
\frac{\mH_{\vy\vy}^{1}}{\eta}&
\mzero&
\mzero&
\mzero&
\mzero&
\mzero&
\mzero
\\
\mzero&
\mzero&
\frac{\mH_{\vx\vx}^{2}}{\eta}&
-\mF_{\vx\vy}^{2}&
\mzero&
\mzero&
-\eye&
\mzero
\\
\mzero&
\mzero&
\mzero&
\frac{\mH_{\vy\vy}^{2}}{\eta}&
\mzero&
\mzero&
\mzero&
\mzero
\\
\mzero&
\mzero&
\mzero&
\mzero&
\frac{\mH_{\vx\vx}^{3}}{\eta}&
-\mF_{\vx\vy}^{3}&
\mzero&
-\eye
\\
\mzero&
\mzero&
\mzero&
\mzero&
\mzero&
\frac{\mH_{\vy\vy}^{3}}{\eta}&
\mzero&
\mzero
\\
\mzero&
\mzero&
\mzero&
\mzero&
\mzero&
\mzero&
-\eye&
\mzero
\\
\mzero&
\mzero&
\mzero&
\mzero&
\mzero&
\mzero&
\mzero&
-\eye
\end{bmatrix}
\label{Dg:system0}
\end{align}
where  $\mGamma^{j}_{\vx\vx},\mGamma^{j}_{\vy\vy},\mF_{\vx\vy}^{j},\mF_{\vy\vx}^{j},\mF_{\vx\vx}^{j},\mF_{\vy\vy}^{j},\mH_{\vx\vx}^{j},\mH_{\vy\vy}^{j}$ for $j=1,2,3$ are defined in \eqref{definition:notations:1}.

By the relative smoothness assumption and stepsize choice of $\eta$, we have from \eqref{definition:notations:1} that 
\begin{align*}
\mGamma_{\vx\vx}^1&\doteq 2\rho\eye+\frac{1}{\eta}\mH_{\vx\vx}^1+\mF_{\vx\vx}^1\succ0,
\quad
\mGamma_{\vx\vx}^j\doteq \rho\eye+\frac{1}{\eta}\mH_{\vx\vx}^j+\mF_{\vx\vx}^j\succ0,~~\forall j\ge 2,
\quad
\mGamma_{\vy\vy}^j\doteq \frac{1}{\eta}\mH_{\vy\vy}^j+\mF_{\vy\vy}^j\succ0,~~\forall j\ge 1.
\end{align*}
Therefore, at any fixed point (where $\vx_j^+=\vx_j$ and $\vy_j^+=\vy_j$), we can compute the fixed-point Jacobian $Dg$ by solving the invertible linear system \eqref{Dg:system0}.
\begin{lemma}
\label{lem:Dg}
For the case $J=3$, suppose $g$ is the mapping function of \Cref{alg:cadmm}. At any fixed point $(\vx_1^\star,\vy_1^\star,\vx_2^\star,\vy_2^\star,\vx_3^\star,\vy_3^\star,\vlambda_2^\star,\vlambda_3^\star)$ where $\vx_j^+=\vx_j$ and $\vy_j^+=\vy_j$ for $j=1,2,3$, the Jacobian $Dg$ is given by
 \begin{align}
Dg=
&\begin{bmatrix}
\mGamma_{\vx\vx}^{1}&
\mzero&
\mzero&
\mzero&
\mzero&
\mzero&
\mzero&
\mzero
\\
\mF_{\vy\vx}^{1}&
\mGamma_{\vy\vy}^{1}&
\mzero&
\mzero&
\mzero&
\mzero&
\mzero&
\mzero
\\
-\rho\eye&
\mzero&
\mGamma_{\vx\vx}^{2}&
\mzero&
\mzero&
\mzero&
\mzero&
\mzero
\\
\mzero&
\mzero&
\mF_{\vy\vx}^{2}&
\mGamma_{\vy\vy}^{2}&
\mzero&
\mzero&
\mzero&
\mzero
\\
-\rho\eye&
\mzero&
\mzero&
\mzero&
\mGamma_{\vx\vx}^{3}&
\mzero&
\mzero&
\mzero
\\
\mzero&
\mzero&
\mzero&
\mzero&
\mF_{\vy\vx}^{3}&
\mGamma_{\vy\vy}^{3}&
\mzero&
\mzero
\\
-\rho\eye&
\mzero&
\rho\eye&
\mzero&
\mzero&
\mzero&
-\eye&
\mzero
\\
-\rho\eye&
\mzero&
\mzero&
\mzero&
\rho\eye&
\mzero&
\mzero&
-\eye
\end{bmatrix}^{-1}
\begin{bmatrix}
\frac{\mH_{\vx\vx}^{1}}{\eta}&
-\mF_{\vx\vy}^{1}&
\rho\eye&
\mzero&
\rho\eye&
\mzero&
\eye&
\eye
\\
\mzero&
\frac{\mH_{\vy\vy}^{1}}{\eta}&
\mzero&
\mzero&
\mzero&
\mzero&
\mzero&
\mzero
\\
\mzero&
\mzero&
\frac{\mH_{\vx\vx}^{2}}{\eta}&
-\mF_{\vx\vy}^{2}&
\mzero&
\mzero&
-\eye&
\mzero
\\
\mzero&
\mzero&
\mzero&
\frac{\mH_{\vy\vy}^{2}}{\eta}&
\mzero&
\mzero&
\mzero&
\mzero
\\
\mzero&
\mzero&
\mzero&
\mzero&
\frac{\mH_{\vx\vx}^{3}}{\eta}&
-\mF_{\vx\vy}^{3}&
\mzero&
-\eye
\\
\mzero&
\mzero&
\mzero&
\mzero&
\mzero&
\frac{\mH_{\vy\vy}^{3}}{\eta}&
\mzero&
\mzero
\\
\mzero&
\mzero&
\mzero&
\mzero&
\mzero&
\mzero&
-\eye&
\mzero
\\
\mzero&
\mzero&
\mzero&
\mzero&
\mzero&
\mzero&
\mzero&
-\eye
\end{bmatrix}.
\label{eqn:Dg}
\end{align}
\end{lemma}

\paragraph{Showing that $\det(Dg)$ is nonzero globally.}
We decompose one full iteration of \Cref{alg:cadmm} into $2J+1$ elementary maps (for $J=3$: seven maps) and show each has a nonsingular Jacobian; then $\det(Dg) = \prod_k \det(Dg_k) \neq 0$ follows by the chain rule.

Define the elementary maps acting on the full state $(\vx_1,\vy_1,\vx_2,\vy_2,\vx_3,\vy_3,\vlambda_2,\vlambda_3)$:
\begin{itemize}[leftmargin=*]
  \item $h_{\vx_j}$: updates $\vx_j$ via the Bregman-proximal step \eqref{eqn:algorithm}, holding all other state variables fixed.
  \item $h_{\vy_j}$: updates $\vy_j$ via the bi-Bregman-proximal step in \eqref{eqn:algorithm}, using the already-updated $\vx_j^+$, holding all other variables fixed.
  \item $h_\vlambda$: dual ascent, updating $\vlambda_2,\vlambda_3$ using the already-updated primal variables.
\end{itemize}
Then 
\[
g = h_\vlambda \circ h_{\vy_3} \circ h_{\vx_3} \circ h_{\vy_2} \circ h_{\vx_2} \circ h_{\vy_1} \circ h_{\vx_1}
\]

\medskip
\noindent\textbf{Nonsingularity of $Dh_{\vx_j}$.}
The map $h_{\vx_j}$ only changes the $\vx_j$ component of the state, so $Dh_{\vx_j}$ is a block matrix with identity in every row except the $\vx_j$-row. Hence $\det(Dh_{\vx_j}) = \det\!\left(\tfrac{\partial \vx_j^+}{\partial \vx_j}\right)$.
From the implicit function theorem applied to the $\vx_j$ optimality condition (with all other variables treated as fixed parameters), differentiating with respect to $\vx_j$ gives
\[
  \frac{\partial \vx_j^+}{\partial \vx_j} = \bigl(\mGamma_{\vx\vx}^j\bigr)^{-1}\!\!\left(\tfrac{1}{\eta}\mH_{\vx\vx}^j\right),
\]
where $\mGamma_{\vx\vx}^j \succ 0$ by the relative bi-smoothness assumption and the stepsize condition, and $\mH_{\vx\vx}^j = \nabla^2_{\vx\vx}h_j(\vx_j,\vy_j) \succ 0$ by the strong convexity of $h_j$. Since $\tfrac{\partial \vx_j^+}{\partial \vx_j}$ is a product of two positive definite (hence nonsingular) matrices, $\det(Dh_{\vx_j}) \neq 0$.

\medskip
\noindent\textbf{Nonsingularity of $Dh_{\vy_j}$.}
By the same argument applied to the $\vy_j$ optimality condition, differentiating with respect to $\vy_j$ gives
\[
  \frac{\partial \vy_j^+}{\partial \vy_j} = \bigl(\mGamma_{\vy\vy}^j\bigr)^{-1}\!\!\left(\tfrac{1}{\eta}\mH_{\vy\vy}^j\right),
\]
where $\mGamma_{\vy\vy}^j \succ 0$ and $\mH_{\vy\vy}^j = \nabla^2_{\vy\vy}h_j(\vx_j^+,\vy_j) \succ 0$.
Hence $\det(Dh_{\vy_j}) = \det\!\left(\tfrac{\partial \vy_j^+}{\partial \vy_j}\right) \neq 0$.

\medskip
\noindent\textbf{Nonsingularity of $Dh_\vlambda$.}
The dual update is $\vlambda_j^+ = \vlambda_j + \rho(\vx_j^+ - \vx_1^+)$. Its full Jacobian is block triangular: the primal variables are left unchanged, and the derivative of each $\vlambda_j^+$ with respect to its own multiplier $\vlambda_j$ is $\eye$, while derivatives with respect to the updated primal variables appear only in lower-left blocks. Hence $\det(Dh_\vlambda)=1$, so $Dh_\vlambda$ is nonsingular.

\medskip
Since all seven elementary Jacobians are nonsingular, we conclude $\det(Dg) = \prod_k \det(Dh_k) \neq 0$ everywhere in the domain.

\paragraph{Strict-saddle KKT points.}
First, by \Cref{rem:stacked-definition}, let $\vz^\star\in\Omega$ be a strict-saddle KKT point of the consensus problem for $J=3$, with primal components $(\vx_1^\star,\vy_1^\star,\ldots,\vx_3^\star,\vy_3^\star)$ and multipliers $\vlambda_2^\star,\vlambda_3^\star$. Then there exist $\vd_{\vx},\vd_{\vy}^1,\vd_{\vy}^2,\vd_{\vy}^3$ such that
\begin{align}
&\begin{cases} \vx_1^\star=\vx_2^\star=\vx_3^\star,&
\\
	\nabla_{\vx}f_1(\vx_1^\star,\vy_1^\star)-\vlambda_2^\star-\vlambda_3^\star&=\vzero,
\\
	\nabla_{\vx}f_2(\vx_2^\star,\vy_2^\star)+\vlambda_2^\star&=\vzero,
\\
	\nabla_{\vx}f_3(\vx_3^\star,\vy_3^\star)+\vlambda_3^\star&=\vzero,
\\
\nabla_{\vy} f_j(\vx_j^\star,\vy_j^\star)&=\vzero,\quad\forall j\in\{1,2,3\},
\end{cases}\label{eqn:first}
\\
&\sum_{i=1}^3 [\vd_{\vx}^\top~\vd_{\vy}^{i\top}]\nabla^2f_i(\vx_i^\star,\vy_i^\star)\begin{bmatrix}
	\vd_{\vx} \\ \vd_{\vy}^i
\end{bmatrix}<0.
\label{eqn:second}
\end{align}
Now,  $(\vx_1^\star,\vy_1^\star,\cdots,\vx_3^\star,\vy_3^\star,\vlambda_2^\star,\vlambda_3^\star)$ is a fixed point of $g$ since \eqref{eqn:first} satisfies the fixed point equations \eqref{eqn:optu1}-\eqref{eqn:optlam3}.

It remains to show that the spectral radius of $Dg(\vx_{1}^{\star}, \vy_{1}^{\star},\vx_{2}^{\star}, \vy_{2}^{\star},\vx_{3}^{\star},\vy_{3}^{\star},\vlambda_{2}^{\star},\vlambda_{3}^{\star})$ is larger than 1.

Set
\begin{align}	
(\vx_{1}^{+}, \vy_{1}^{+},\vx_{2}^{+}, \vy_{2}^{+},\vx_{3}^{+},\vy_{3}^{+}, \vlambda_{2}^{+},\vlambda_{3}^{+})
=&(\vx_{1}, \vy_{1},\vx_{2}, \vy_{2},\vx_{3},\vy_{3}, \vlambda_{2},\vlambda_{3})
=(\vx_{1}^{\star}, \vy_{1}^{\star},\vx_{2}^{\star}, \vy_{2}^{\star},\vx_{3}^{\star},\vy_{3}^{\star},\vlambda_{2}^{\star},\vlambda_{3}^{\star}),
\label{eqn:evaluation}
\end{align}
and evaluate $Dg$ \eqref{eqn:Dg} at \eqref{eqn:evaluation}
{\small
\begin{align*}
& Dg =
\begin{bmatrix}
\mGamma_{\vx\vx}^{1}&
\mzero&
\mzero&
\mzero&
\mzero&
\mzero&
\mzero&
\mzero
\\
\mF_{\vy\vx}^{1}&
\mGamma_{\vy\vy}^{1}&
\mzero&
\mzero&
\mzero&
\mzero&
\mzero&
\mzero
\\
-\rho\eye&
\mzero&
\mGamma_{\vx\vx}^{2}&
\mzero&
\mzero&
\mzero&
\mzero&
\mzero
\\
\mzero&
\mzero&
\mF_{\vy\vx}^{2}&
\mGamma_{\vy\vy}^{2}&
\mzero&
\mzero&
\mzero&
\mzero
\\
-\rho\eye&
\mzero&
\mzero&
\mzero&
\mGamma_{\vx\vx}^{3}&
\mzero&
\mzero&
\mzero
\\
\mzero&
\mzero&
\mzero&
\mzero&
\mF_{\vy\vx}^{3}&
\mGamma_{\vy\vy}^{3}&
\mzero&
\mzero
\\
-\rho\eye&
\mzero&
\rho\eye&
\mzero&
\mzero&
\mzero&
-\eye&
\mzero
\\
-\rho\eye&
\mzero&
\mzero&
\mzero&
\rho\eye&
\mzero&
\mzero&
-\eye
\end{bmatrix}^{-1}
% \\
% &
\begin{bmatrix}
\frac{\mH_{\vx\vx}^{1}}{\eta}&
-\mF_{\vx\vy}^{1}&
\rho\eye&
\mzero&
\rho\eye&
\mzero&
\eye&
\eye
\\
\mzero&
\frac{\mH_{\vy\vy}^{1}}{\eta}&
\mzero&
\mzero&
\mzero&
\mzero&
\mzero&
\mzero
\\
\mzero&
\mzero&
\frac{\mH_{\vx\vx}^{2}}{\eta}&
-\mF_{\vx\vy}^{2}&
\mzero&
\mzero&
-\eye&
\mzero
\\
\mzero&
\mzero&
\mzero&
\frac{\mH_{\vy\vy}^{2}}{\eta}&
\mzero&
\mzero&
\mzero&
\mzero
\\
\mzero&
\mzero&
\mzero&
\mzero&
\frac{\mH_{\vx\vx}^{3}}{\eta}&
-\mF_{\vx\vy}^{3}&
\mzero&
-\eye
\\
\mzero&
\mzero&
\mzero&
\mzero&
\mzero&
\frac{\mH_{\vy\vy}^{3}}{\eta}&
\mzero&
\mzero
\\
\mzero&
\mzero&
\mzero&
\mzero&
\mzero&
\mzero&
-\eye&
\mzero
\\
\mzero&
\mzero&
\mzero&
\mzero&
\mzero&
\mzero&
\mzero&
-\eye
\end{bmatrix}
\\
=&\eye
-
\begin{bmatrix}
\mGamma_{\vx\vx}^{1}&
\mzero&
\mzero&
\mzero&
\mzero&
\mzero&
\mzero&
\mzero
\\
\mF_{\vy\vx}^{1}&
\mGamma_{\vy\vy}^{1}&
\mzero&
\mzero&
\mzero&
\mzero&
\mzero&
\mzero
\\
-\rho\eye&
\mzero&
\mGamma_{\vx\vx}^{2}&
\mzero&
\mzero&
\mzero&
\mzero&
\mzero
\\
\mzero&
\mzero&
\mF_{\vy\vx}^{2}&
\mGamma_{\vy\vy}^{2}&
\mzero&
\mzero&
\mzero&
\mzero
\\
-\rho\eye&
\mzero&
\mzero&
\mzero&
\mGamma_{\vx\vx}^{3}&
\mzero&
\mzero&
\mzero
\\
\mzero&
\mzero&
\mzero&
\mzero&
\mF_{\vy\vx}^{3}&
\mGamma_{\vy\vy}^{3}&
\mzero&
\mzero
\\
-\rho\eye&
\mzero&
\rho\eye&
\mzero&
\mzero&
\mzero&
-\eye&
\mzero
\\
-\rho\eye&
\mzero&
\mzero&
\mzero&
\rho\eye&
\mzero&
\mzero&
-\eye
\end{bmatrix}^{-1}
\begin{bmatrix}
\mF_{\vx\vx}^1+2\rho\eye&
\mF_{\vx\vy}^{1}&
-\rho\eye&
\mzero&
-\rho\eye&
\mzero&
-\eye&
-\eye
\\
\mF_{\vy\vx}^{1}&
\mF_{\vy\vy}^{1}&
\mzero&
\mzero&
\mzero&
\mzero&
\mzero&
\mzero
\\
-\rho\eye&
\mzero&
\mF_{\vx\vx}^2+\rho\eye&
\mF_{\vx\vy}^{2}&
\mzero&
\mzero&
\eye&
\mzero
\\
\mzero&
\mzero&
\mF_{\vy\vx}^2&
\mF_{\vy\vy}^2&
\mzero&
\mzero&
\mzero&
\mzero
\\
-\rho\eye&
\mzero&
\mzero&
\mzero&
\mF_{\vx\vx}^3+\rho\eye&
\mF_{\vx\vy}^{3}&
\mzero&
\eye
\\
\mzero&
\mzero&
\mzero&
\mzero&
\mF_{\vy\vx}^3&
\mF_{\vy\vy}^3&
\mzero&
\mzero
\\
-\rho\eye&
\mzero&
\rho\eye&
\mzero&
\mzero&
\mzero&
\mzero&
\mzero
\\
-\rho\eye&
\mzero&
\mzero&
\mzero&
\rho\eye&
\mzero&
\mzero&
\mzero
\end{bmatrix}
\\
\doteq& \eye- \mPhi.
\end{align*}}  
Here, in view of \eqref{definition:notations:1}, the second equality follows from that
\begin{align*}
\mGamma_{\vx\vx}^1\doteq 2\rho\eye+\frac{1}{\eta}\mH_{\vx\vx}^1+\mF_{\vx\vx}^1,
\qquad
\mGamma_{\vx\vx}^j\doteq \rho\eye+\frac{1}{\eta}\mH_{\vx\vx}^j+\mF_{\vx\vx}^j,~~\forall j\ge 2,
\qquad
\mGamma_{\vy\vy}^j\doteq \frac{1}{\eta}\mH_{\vy\vy}^j+\mF_{\vy\vy}^j,~~\forall j\ge 1.
\end{align*}
Therefore, to show that the spectral radius of $Dg$ is larger than 1, it suffices to show that $\mPhi$ has a real negative eigenvalue, which is further equivalent to
\[\det(\mPhi+\mu\eye)=0\quad\text{for some }\mu>0.\]
Write $Dg=\eye-\mPhi=\eye-\mM^{-1}\mN$, where $\mM$ is the block-lower-triangular matrix appearing with the inverse in the computation above and $\mN$ the remaining factor.
The first $\Leftrightarrow$ below uses $\mPhi+\mu\eye=\mM^{-1}\mN+\mu\eye=\mM^{-1}(\mN+\mu\mM)$, so $\det(\mPhi+\mu\eye)=\det(\mM^{-1})\det(\mN+\mu\mM)$. Since $\mM$ is block-lower-triangular with diagonal blocks built from $\mGamma_{\vx\vx}^j$ and $\mGamma_{\vy\vy}^j$ (positive definite by relative bi-smoothness and the stepsize condition) and identity blocks from the dual update, $\mM$ is nonsingular and $\det(\mM^{-1})\neq 0$, reducing the condition to $\det(\mN+\mu\mM)=0$.
The second $\Leftrightarrow$ symmetrizes $\mN+\mu \mM$ via a two-sided row-column scaling: blocks corresponding to the $y_j$-variables and the non-hub $x_j$-blocks ($j\ge 2$) are scaled by $s_2=\sqrt{1+\mu}$, and those corresponding to the Lagrangian dual variables $\vlambda_j$ by $s_1=\sqrt{(1+\mu)\rho}$; since all factors are positive, the zero-determinant condition is preserved.
That is,
{\footnotesize
\begin{align*}
&
\det(\mPhi+\mu\eye)=0
\\
&\iff
\\
&
\det
\begin{bmatrix}
\mF_{\vx\vx}^1+2\rho\eye+\mu\mGamma_{\vx\vx}^1&
\mF_{\vx\vy}^{1}&
-\rho\eye&
\mzero&
-\rho\eye&
\mzero&
-\eye&
-\eye
\\
(1+\mu)\mF_{\vy\vx}^{1}&
\mF_{\vy\vy}^{1}+\mu\mGamma_{\vy\vy}^1&
\mzero&
\mzero&
\mzero&
\mzero&
\mzero&
\mzero
\\
-(1+\mu)\rho\eye&
\mzero&
\mF_{\vx\vx}^2+\rho\eye+\mu\mGamma_{\vx\vx}^2&
\mF_{\vx\vy}^{2}&
\mzero&
\mzero&
\eye&
\mzero
\\
\mzero&
\mzero&
(1+\mu)\mF_{\vy\vx}^2&
\mF_{\vy\vy}^2+\mu\mGamma_{\vy\vy}^2&
\mzero&
\mzero&
\mzero&
\mzero
\\
-(1+\mu)\rho\eye&
\mzero&
\mzero&
\mzero&
\mF_{\vx\vx}^3+\rho\eye+\mu\mGamma_{\vx\vx}^3&
\mF_{\vx\vy}^{3}&
\mzero&
\eye
\\
\mzero&
\mzero&
\mzero&
\mzero&
(1+\mu)\mF_{\vy\vx}^3&
\mF_{\vy\vy}^3+\mu\mGamma_{\vy\vy}^3&
\mzero&
\mzero
\\
-(1+\mu)\rho\eye&
\mzero&
(1+\mu)\rho\eye&
\mzero&
\mzero&
\mzero&
-\mu\eye&
\mzero
\\
-(1+\mu)\rho\eye&
\mzero&
\mzero&
\mzero&
(1+\mu)\rho\eye&
\mzero&
\mzero&
-\mu\eye
\end{bmatrix}	
=0
\\
&\iff
\\
 &
\det
\underbrace{\begin{bmatrix}
\mF_{\vx\vx}^1+2\rho\eye+\mu\mGamma_{\vx\vx}^1&
s_2\mF_{\vx\vy}^{1}&
-s_2\rho\eye&
\mzero&
-s_2\rho\eye&
\mzero&
-s_1\eye&
-s_1\eye
\\
s_2\mF_{\vy\vx}^{1}&
\mF_{\vy\vy}^{1}+\mu\mGamma_{\vy\vy}^1&
\mzero&
\mzero&
\mzero&
\mzero&
\mzero&
\mzero
\\
-s_2\rho\eye&
\mzero&
\mF_{\vx\vx}^2+\rho\eye+\mu\mGamma_{\vx\vx}^2&
s_2\mF_{\vx\vy}^{2}&
\mzero&
\mzero&
s_1\eye&
\mzero
\\
\mzero&
\mzero&
s_2\mF_{\vy\vx}^2&
\mF_{\vy\vy}^2+\mu\mGamma_{\vy\vy}^2&
\mzero&
\mzero&
\mzero&
\mzero
\\
-s_2\rho\eye&
\mzero&
\mzero&
\mzero&
\mF_{\vx\vx}^3+\rho\eye+\mu\mGamma_{\vx\vx}^3&
s_2\mF_{\vx\vy}^{3}&
\mzero&
s_1\eye
\\
\mzero&
\mzero&
\mzero&
\mzero&
s_2\mF_{\vy\vx}^3&
\mF_{\vy\vy}^3+\mu\mGamma_{\vy\vy}^3&
\mzero&
\mzero
\\
-s_1\eye&
\mzero&
s_1\eye&
\mzero&
\mzero&
\mzero&
-\mu\eye&
\mzero
\\
-s_1\eye&
\mzero&
\mzero&
\mzero&
s_1\eye&
\mzero&
\mzero&
-\mu\eye
\end{bmatrix}}_{\mJ(\mu)}	
=0
\end{align*}
}
where $s_1\doteq\sqrt{(1+\mu)\rho}$ and $s_2\doteq\sqrt{1+\mu}.$

Therefore, the problem reduces to showing $\mJ(\mu)$ is singular for some $\mu>0$.
The last two diagonal blocks of $\mJ(\mu)$ are $-\mu\eye$, so $\lambda_{\min}(\mJ(\mu))\to-\infty$ as $\mu\to\infty$ and a direct sign-change argument on $\lambda_{\min}(\mJ(\mu))$ fails; instead we Schur-reduce the dual block.

Since the bottom-right diagonal blocks of $\mJ(\mu)$ are $-\mu\eye$ (invertible for $\mu>0$), partition
\[
\mJ(\mu)=\begin{bmatrix}\mS(\mu) & \mT(\mu) \\ \mT(\mu)^\top & -\mu\eye\end{bmatrix},
\]
where $\mS(\mu)$ is the top-left $6$-primal-block matrix and $\mT(\mu)$ contains the two dual coupling columns of $\mJ(\mu)$.  Then for any $\mu>0$:
\begin{align}
\det(\mJ(\mu)) = \det(-\mu\eye)\cdot\det(\tilde{\mJ}(\mu)),\quad \tilde{\mJ}(\mu)\doteq \mS(\mu)+\tfrac{1}{\mu}\mT(\mu)\mT(\mu)^\top,
\label{eqn:singularrule}
\end{align}
so $\mJ(\mu)$ is singular $\iff$ $\tilde{\mJ}(\mu)$ is singular.  The matrix $\tilde{\mJ}(\mu)$ is real-symmetric and varies continuously in $\mu$ for $\mu>0$.

Ordering the primal rows as $(\vx_1,\vy_1,\vx_2,\vy_2,\vx_3,\vy_3)$ and the dual columns as $(\vlambda_2,\vlambda_3)$, the coupling matrix $\mT(\mu)$ reads
\[
\mT(\mu)=
\begin{bmatrix}
-s_1\eye & -s_1\eye \\
\mzero   & \mzero   \\
 s_1\eye & \mzero   \\
\mzero   & \mzero   \\
\mzero   &  s_1\eye \\
\mzero   & \mzero
\end{bmatrix},
\quad s_1\doteq\sqrt{(1+\mu)\rho},
\]
where column $j-1$ encodes the constraint $\vx_j-\vx_1=\vzero$ via $-s_1\eye$ in the $\vx_1$-rows and $+s_1\eye$ in the $\vx_j$-rows ($j=2,3$).

\begin{remark}[Intuition for the consensus test direction]\label{rem:intuition:consensus}
For the $J=3$ calculation, the test direction $\vd=(\vd_{\vx},\vd_{\vy}^1,\vd_{\vx},\vd_{\vy}^2,\vd_{\vx},\vd_{\vy}^3,\vzero,\vzero)$ must satisfy two conditions at once: it should lie in the null space of the consensus coupling so that the $\rho$-penalty cancels, and it should expose negative curvature of the block-diagonal Hessian $\mathrm{diag}(\nabla^2 f_1,\nabla^2 f_2,\nabla^2 f_3)$.
Both are possible because a strict-saddle KKT point of \eqref{eqn:problemC} is feasible and therefore satisfies $\vx_1^\star=\vx_2^\star=\vx_3^\star$, so the negative-curvature direction from \eqref{eqn:second} may be chosen with a common $\vx$-component $\vd_{\vx}$ across the three agents.
This equal-component structure places the $\vx$-stack $(\vd_{\vx},\vd_{\vx},\vd_{\vx})$ in the kernel of the three-node star-graph Laplacian, so $\mT(\mu)^{\top}\vd_\mathrm{p}=\vzero$ for every $\mu$, where \(\vd_{\mathrm p}\) denotes the primal part of \(\vd\).
Equivalently, with $\vd_{\vx,\mathrm{stack}}=(\vd_{\vx},\vd_{\vx},\vd_{\vx})$, the consensus penalty contributes $\rho\,\vd_{\vx,\mathrm{stack}}^{\top}(\mL_{\mathrm{star}}\otimes\eye_n)\vd_{\vx,\mathrm{stack}}=0$ on this equal-$\vx$ direction.
At $\mu=0$ only the three Hessian blocks $\nabla^2 f_1,\nabla^2 f_2,\nabla^2 f_3$ survive, recovering the strict-saddle negativity from \eqref{eqn:second}.
In contrast to the two-block ADMM case (\Cref{rem:intuition:ADMM}), no $\mu$-dependent rescaling of the $\vy$-components is needed: the off-diagonal $\rho$-terms couple only $\vx$-blocks, and setting the dual components to zero decouples the primal form from the Schur complement correction $\tfrac{1}{\mu}\mT\mT^{\top}$.
\end{remark}

By feasibility of the consensus constraints in \eqref{eqn:problemC}, every stationary point, hence every strict saddle, satisfies $\vx_1^\star=\vx_2^\star=\vx_3^\star$. Therefore the negative-curvature direction from \eqref{eqn:second} can be chosen with all $\vx$-components equal to a common~$\vd_{\vx}$. Indeed, the constraints $\vx_j=\vx_1$ for $j=2,3$ force every feasible perturbation to satisfy $\delta\vx_j=\delta\vx_1$, so any direction witnessing the strict-saddle condition \eqref{eqn:second} must already have this equal-$\vx$ form.
Set
\[
\vd=(\vd_{\vx},\vd_{\vy}^1,\vd_{\vx},\vd_{\vy}^2,\vd_{\vx},\vd_{\vy}^3,\vzero,\vzero)\neq\vzero,
\]
where $\vd_{\vx},\vd_{\vy}^1,\vd_{\vy}^2,\vd_{\vy}^3$ satisfy \eqref{eqn:second} and the dual components are zero. Write $\vd_\mathrm{p}=(\vd_{\vx},\vd_{\vy}^1,\vd_{\vx},\vd_{\vy}^2,\vd_{\vx},\vd_{\vy}^3)$ for the primal part of $\vd$.
Since each column of $\mT(\mu)$ has a $-s_1\eye$ in the $\vx_1$-rows and a $+s_1\eye$ in exactly one other $\vx$-block, and all $\vx$-components of $\vd$ equal $\vd_{\vx}$, we get
\[
\mT(\mu)^\top\vd_\mathrm{p} = \begin{pmatrix}(-s_1\eye)\vd_{\vx}+(s_1\eye)\vd_{\vx}\\(-s_1\eye)\vd_{\vx}+(s_1\eye)\vd_{\vx}\end{pmatrix} = \vzero \quad\text{for all }\mu\ge0.
\]
Setting $\phi(\mu)\doteq\vd^\top\mJ(\mu)\vd$ and writing $\vd=(\vd_\mathrm{p},\vzero)$, we expand using the block structure of $\mJ(\mu)$:
\[
\phi(\mu) = \vd_\mathrm{p}^\top \mS(\mu)\vd_\mathrm{p} + 2\,\vd_\mathrm{p}^\top \mT(\mu)\cdot\vzero - \mu\|\vzero\|^2 = \vd_\mathrm{p}^\top \mS(\mu)\vd_\mathrm{p}.
\]
Since $\mT(\mu)^\top\vd_\mathrm{p}=\vzero$, the Schur complement formula \eqref{eqn:singularrule} gives 
\[\vd_\mathrm{p}^\top\tilde{\mJ}(\mu)\vd_\mathrm{p}=\vd_\mathrm{p}^\top \mS(\mu)\vd_\mathrm{p}+\tfrac{1}{\mu}\|\mT(\mu)^\top\vd_\mathrm{p}\|^2=\vd_\mathrm{p}^\top \mS(\mu)\vd_\mathrm{p},\]
so
\[
\phi(\mu) = \vd^\top\mJ(\mu)\vd = \vd_\mathrm{p}^\top \mS(\mu)\vd_\mathrm{p} = \vd_\mathrm{p}^\top\tilde{\mJ}(\mu)\vd_\mathrm{p}
\quad\text{for all }\mu>0,
\]
and by continuity of $\mS(\mu)$ also at $\mu=0$.

We start by showing that
\[\phi(0)<0.\]
Substituting $\mu=0$ into $\mJ(\mu)$ and expanding the quadratic form $\phi(0)=\vd^\top\mJ(0)\vd$ (using that the dual components of $\vd$ are zero and the $\vx$-components repeat across agents) gives
\begin{align*}
\phi(0)&=\vd^\top\mJ(0)\vd\\
&=\begin{bmatrix}
\vd_{\vx} \\ \vd_{\vy}^1 \\ \vd_{\vx} \\ \vd_{\vy}^2 \\ \vd_{\vx} \\ \vd_{\vy}^3 \\ \vzero \\ \vzero\end{bmatrix}^\top
\begin{bmatrix}
\mF_{\vx\vx}^1+2\rho\eye&
\mF_{\vx\vy}^{1}&
-\rho\eye&
\mzero&
-\rho\eye&
\mzero&
-\sqrt{\rho}\eye&
-\sqrt{\rho}\eye
\\
\mF_{\vy\vx}^{1}&
\mF_{\vy\vy}^{1}&
\mzero&
\mzero&
\mzero&
\mzero&
\mzero&
\mzero
\\
-\rho\eye&
\mzero&
\mF_{\vx\vx}^2+\rho\eye&
\mF_{\vx\vy}^{2}&
\mzero&
\mzero&
\sqrt{\rho}\eye&
\mzero
\\
\mzero&
\mzero&
\mF_{\vy\vx}^2&
\mF_{\vy\vy}^2&
\mzero&
\mzero&
\mzero&
\mzero
\\
-\rho\eye&
\mzero&
\mzero&
\mzero&
\mF_{\vx\vx}^3+\rho\eye&
\mF_{\vx\vy}^{3}&
\mzero&
\sqrt{\rho}\eye
\\
\mzero&
\mzero&
\mzero&
\mzero&
\mF_{\vy\vx}^3&
\mF_{\vy\vy}^3&
\mzero&
\mzero
\\
-\sqrt{\rho}\eye&
\mzero&
\sqrt{\rho}\eye&
\mzero&
\mzero&
\mzero&
\mzero&
\mzero
\\
-\sqrt{\rho}\eye&
\mzero&
\mzero&
\mzero&
\sqrt{\rho}\eye&
\mzero&
\mzero&
\mzero
\end{bmatrix}
\begin{bmatrix}
\vd_{\vx} \\ \vd_{\vy}^1 \\ \vd_{\vx} \\ \vd_{\vy}^2 \\ \vd_{\vx} \\ \vd_{\vy}^3 \\ \vzero \\ \vzero\end{bmatrix}
\\
&=\begin{bmatrix}
\vd_{\vx} \\ \vd_{\vy}^1 \\ \vd_{\vx} \\ \vd_{\vy}^2 \\ \vd_{\vx} \\ \vd_{\vy}^3 \end{bmatrix}^\top
\begin{bmatrix}
\mF_{\vx\vx}^1+2\rho\eye&
\mF_{\vx\vy}^{1}&
-\rho\eye&
\mzero&
-\rho\eye&
\mzero
\\
\mF_{\vy\vx}^{1}&
\mF_{\vy\vy}^{1}&
\mzero&
\mzero&
\mzero&
\mzero
\\
-\rho\eye&
\mzero&
\mF_{\vx\vx}^2+\rho\eye&
\mF_{\vx\vy}^{2}&
\mzero&
\mzero
\\
\mzero&
\mzero&
\mF_{\vy\vx}^2&
\mF_{\vy\vy}^2&
\mzero&
\mzero
\\
-\rho\eye&
\mzero&
\mzero&
\mzero&
\mF_{\vx\vx}^3+\rho\eye&
\mF_{\vx\vy}^{3}
\\
\mzero&
\mzero&
\mzero&
\mzero&
\mF_{\vy\vx}^3&
\mF_{\vy\vy}^3
\end{bmatrix}
\begin{bmatrix}
\vd_{\vx} \\ \vd_{\vy}^1 \\ \vd_{\vx} \\ \vd_{\vy}^2 \\ \vd_{\vx} \\ \vd_{\vy}^3 \end{bmatrix}
\\
&=\begin{bmatrix}
\vd_{\vx} \\ \vd_{\vy}^1 \\ \vd_{\vx} \\ \vd_{\vy}^2 \\ \vd_{\vx} \\ \vd_{\vy}^3 \end{bmatrix}^\top
\begin{bmatrix}
\mF_{\vx\vx}^1&
\mF_{\vx\vy}^{1}&
\mzero&
\mzero&
\mzero&
\mzero
\\
\mF_{\vy\vx}^{1}&
\mF_{\vy\vy}^{1}&
\mzero&
\mzero&
\mzero&
\mzero
\\
\mzero&
\mzero&
\mF_{\vx\vx}^2&
\mF_{\vx\vy}^{2}&
\mzero&
\mzero
\\
\mzero&
\mzero&
\mF_{\vy\vx}^2&
\mF_{\vy\vy}^2&
\mzero&
\mzero
\\
\mzero&
\mzero&
\mzero&
\mzero&
\mF_{\vx\vx}^3&
\mF_{\vx\vy}^{3}
\\
\mzero&
\mzero&
\mzero&
\mzero&
\mF_{\vy\vx}^3&
\mF_{\vy\vy}^3
\end{bmatrix}
\begin{bmatrix}
\vd_{\vx} \\ \vd_{\vy}^1 \\ \vd_{\vx} \\ \vd_{\vy}^2 \\ \vd_{\vx} \\ \vd_{\vy}^3 \end{bmatrix}
\\
&=\sum_{i=1}^3 [\vd_{\vx}^\top~\vd_{\vy}^{i\top}]\nabla^2f_i(\vx_i^\star,\vy_i^\star)\begin{bmatrix}
	\vd_{\vx} \\ \vd_{\vy}^i
\end{bmatrix}<0\quad \text{(by \eqref{eqn:second})}.
\end{align*}

Second, we will show that
\[
\lim\limits_{\mu\to\infty}\frac{\phi(\mu)}{\mu}>0.
\]
To compute the limit, divide $\mJ(\mu)$ by $\mu$ and use that $s_1=\sqrt{(1+\mu)\rho}$ and $s_2=\sqrt{1+\mu}$, so the off-diagonal coupling blocks are $O(\sqrt{\mu})$ and therefore vanish after scaling by $1/\mu$, while each diagonal primal block is of the form $\mF_{**}^j+\text{const}\cdot\rho\eye+\mu\mGamma_{**}^j$ and hence converges to $\mGamma_{**}^j$. This yields
\begin{align*}
\lim\limits_{\mu\to\infty}\frac{\phi(\mu)}{\mu}=\vd^\top\left(\lim\limits_{\mu\to\infty}\frac{\mJ(\mu)}{\mu}\right)\vd
&=\begin{bmatrix}
\vd_{\vx} \\ \vd_{\vy}^1 \\ \vd_{\vx} \\ \vd_{\vy}^2 \\ \vd_{\vx} \\ \vd_{\vy}^3 \\ \vzero \\ \vzero\end{bmatrix}^\top
\begin{bmatrix}
\mGamma_{\vx\vx}^1
&&&&&&&\\
&\mGamma_{\vy\vy}^1
&&&&&&
\\
&&\mGamma_{\vx\vx}^2
&&&&&
\\
&&&\mGamma_{\vy\vy}^2
&&&&
\\
&&&&\mGamma_{\vx\vx}^3
&&&
\\
&&&&&\mGamma_{\vy\vy}^3
&&
\\
&&&&&&-\eye &
\\
&&&&&&&-\eye
\end{bmatrix}
\begin{bmatrix}
\vd_{\vx} \\ \vd_{\vy}^1 \\ \vd_{\vx} \\ \vd_{\vy}^2 \\ \vd_{\vx} \\ \vd_{\vy}^3 \\ \vzero \\ \vzero\end{bmatrix}
\\
&=\begin{bmatrix}
\vd_{\vx} \\ \vd_{\vy}^1 \\ \vd_{\vx} \\ \vd_{\vy}^2 \\ \vd_{\vx} \\ \vd_{\vy}^3
\end{bmatrix}^\top
\begin{bmatrix}
\mGamma_{\vx\vx}^1
&&&&&
\\
&\mGamma_{\vy\vy}^1
&&&&
\\
&&\mGamma_{\vx\vx}^2
&&&
\\
&&&\mGamma_{\vy\vy}^2
&&
\\
&&&&\mGamma_{\vx\vx}^3
&
\\
&&&&&\mGamma_{\vy\vy}^3
\end{bmatrix}
\begin{bmatrix}
\vd_{\vx} \\ \vd_{\vy}^1 \\ \vd_{\vx} \\ \vd_{\vy}^2 \\ \vd_{\vx} \\ \vd_{\vy}^3 \end{bmatrix}
>0.
\end{align*}
Here, in view of \eqref{definition:notations:1},  the last line is because
\begin{align*}
\mGamma_{\vx\vx}^1&\doteq 2\rho\eye+\frac{1}{\eta}\mH_{\vx\vx}^1+\mF_{\vx\vx}^1\succ0,
\quad
\mGamma_{\vx\vx}^j\doteq \rho\eye+\frac{1}{\eta}\mH_{\vx\vx}^j+\mF_{\vx\vx}^j\succ0,~~\forall j\ge 2,
\quad
\mGamma_{\vy\vy}^j\doteq \frac{1}{\eta}\mH_{\vy\vy}^j+\mF_{\vy\vy}^j\succ0,~~\forall j\ge 1,
\end{align*}
which follows from relative bi-smoothness, strong bi-convexity of the kernels, and $\eta<\min_j(1/L_j^x,1/L_j^y)$.

Since $\phi(\mu)$ is a continuous real function of $\mu$, the above limit implies $\phi(N)>0$ for some sufficiently large $N$.

Although $\tilde{\mJ}(\mu)$ is defined only for $\mu>0$, the identity above gives $\vd_\mathrm{p}^{\top}\tilde{\mJ}(\mu)\vd_\mathrm{p}=\vd_\mathrm{p}^{\top}\mS(\mu)\vd_\mathrm{p}$ along this test direction, and the right-hand side has a continuous extension to $\mu=0$. Therefore the two cases above translate to statements about $\lambda_{\min}(\tilde{\mJ}(\mu))$:
\begin{itemize}[leftmargin=*]
\item \textit{Small $\mu$}: $\phi(0)<0$ and continuity give $\phi(\mu)<0$ for small $\mu>0$, so $\lambda_{\min}(\tilde{\mJ}(\mu))\le\vd_\mathrm{p}^{\top}\tilde{\mJ}(\mu)\vd_\mathrm{p}/\|\vd_\mathrm{p}\|^2<0$.
\item \textit{Large $\mu$}:
Since $s_1=\sqrt{(1+\mu)\rho}$ and $s_2=\sqrt{1+\mu}$, both coefficients are $O(\sqrt\mu)$.
The entries of $\mT(\mu)$ are multiples of $s_1$, so $\mT(\mu)\mT(\mu)^\top=O(\mu)$ and the term $(1/\mu^2)\mT(\mu)\mT(\mu)^\top=O(1/\mu)\to0$.
For $\mS(\mu)$: its diagonal blocks are $\mF^j_{**}+\text{const}\cdot\rho\eye+\mu\mGamma^j_{**}$, so $\mS(\mu)/\mu$ has diagonal blocks $\to\mGamma^j_{**}$, while each off-diagonal block is $O(\sqrt\mu)$, giving $\mS(\mu)/\mu$ off-diagonal terms of order $O(1/\sqrt\mu)\to0$.
Therefore
\[
\frac{\tilde{\mJ}(\mu)}{\mu}=\frac{\mS(\mu)}{\mu}+\frac{1}{\mu^2}\mT(\mu)\mT(\mu)^\top
\;\xrightarrow{\mu\to\infty}\;
\mathrm{diag}\bigl(\mGamma_{\vx\vx}^1,\mGamma_{\vy\vy}^1,\mGamma_{\vx\vx}^2,\mGamma_{\vy\vy}^2,\mGamma_{\vx\vx}^3,\mGamma_{\vy\vy}^3\bigr)\succ0,
\]
where the positive definiteness is from \eqref{definition:notations:1} and the stepsize condition.
Since the limit is $\succ0$, there exist $\epsilon>0$ and $N>0$ such that $\tilde{\mJ}(\mu)/\mu\succ\epsilon\eye$ for all $\mu>N$, hence $\tilde{\mJ}(\mu)\succ\epsilon\mu\eye\succ0$, and in particular $\lambda_{\min}(\tilde{\mJ}(N))>0$.
\end{itemize}
Since $\tilde{\mJ}(\mu)$ is a continuous real-symmetric matrix for $\mu>0$, its eigenvalues vary continuously with $\mu$ (see \cite[Theorem~5.1]{kato2013perturbation}).  By the intermediate value theorem applied to $\lambda_{\min}(\tilde{\mJ}(\mu))$, there exists $\mu^\star\in(0,N)$ such that $\lambda_{\min}(\tilde{\mJ}(\mu^\star))=0$, i.e., $\tilde{\mJ}(\mu^\star)$ is singular, and by \eqref{eqn:singularrule} so is $\mJ(\mu^\star)$.

The preceding paragraphs complete the detailed $J=3$ instability calculation. We now record why the same argument extends to arbitrary $J\ge2$.
\begin{lemma}[Extension to general $J$]\label{rem:cadmm:generalJ}
Under the hypotheses of \Cref{thm:cadmm}, for any $J\ge 2$:
\begin{enumerate}[leftmargin=*]
\item $\det(Dg(\vz))\neq 0$ for all $\vz\in\Omega$;
\item every strict-saddle KKT point of \eqref{eqn:problemC} in $\Omega$ is an unstable fixed point of $g$.
\end{enumerate}
\end{lemma}
\begin{proof}
The proof uses exactly the same two ingredients as the $J=3$ case.

\emph{Nonsingularity.}
One full iteration decomposes into $2J+1$ elementary maps: one $\vx_j$-update and one $\vy_j$-update for each agent, followed by the dual update. Each Jacobian is block triangular with diagonal blocks given by the same implicit-function-theorem matrices as in the displayed $J=3$ computation, namely $\mGamma_{\vx\vx}^j$, $\mGamma_{\vy\vy}^j$, and identity blocks for untouched variables and dual updates. Since every $\mGamma$ block is positive definite by relative bi-smoothness and the stepsize condition, every elementary Jacobian is nonsingular; hence $\det(Dg)\neq 0$ for arbitrary $J$ by the chain rule.

\emph{Instability.}
The coupling matrix $\mT(\mu)\in\R^{2Jn\times(J-1)n}$ has column $j-1$ ($j=2,\ldots,J$) with $-s_1\eye$ in the $\vx_1$-block and $+s_1\eye$ in the $\vx_j$-block (all other blocks zero). With the test direction $\vd=(\vd_\vx,\vd_\vy^1,\ldots,\vd_\vx,\vd_\vy^J,\vzero,\ldots,\vzero)$ (all $\vx$-components equal to $\vd_\vx$), each column of $\mT(\mu)$ contributes $(-s_1\eye)\vd_\vx+(s_1\eye)\vd_\vx=\vzero$, hence $\mT(\mu)^{\top}\vd_\mathrm{p}=\vzero$.

The general-$J$ structure of the Schur-reduced matrix $\tilde{\mJ}(\mu)=\mS(\mu)+\tfrac{1}{\mu}\mT(\mu)\mT(\mu)^{\top}\in\R^{2Jn\times 2Jn}$ is as follows. Since $\mT(\mu)$ has nonzero entries only in the $\vx$-rows, the additive correction $\tfrac{1}{\mu}\mT(\mu)\mT(\mu)^{\top}$ affects only $\vx$--$\vx$ blocks, with all $\vy$-rows and $\vy$-columns equal to zero. Restricting to the $\vx$-coordinates $(\vx_1,\vx_2,\ldots,\vx_J)$, we obtain
\[
\begin{bmatrix}
\frac{1}{\mu}\mT(\mu)\mT(\mu)^{\top}
\end{bmatrix}_{\vx\text{-blocks},\vx\text{-blocks}}
=\frac{s_1^2}{\mu}
\begin{bmatrix}
(J{-}1)\eye & -\eye & -\eye & \cdots & -\eye\\
-\eye & \eye & \mzero & \cdots & \mzero\\
-\eye & \mzero & \eye & \cdots & \mzero\\
\vdots & \vdots & \vdots & \ddots & \vdots\\
-\eye & \mzero & \mzero & \cdots & \eye
\end{bmatrix}
=\frac{s_1^2}{\mu}\,\mL_{\mathrm{star}}\otimes\eye_n,
\]
where $\mL_{\mathrm{star}}\in\R^{J\times J}$ is the graph Laplacian of the star with hub node~$1$. The constant vector $(\vd_{\vx},\ldots,\vd_{\vx})$ lies in $\ker(\mL_{\mathrm{star}}\otimes\eye_n)$, so $\tfrac{1}{\mu}\mT(\mu)\mT(\mu)^{\top}$ contributes nothing to $\vd_\mathrm{p}^{\top}\tilde{\mJ}(\mu)\vd_\mathrm{p}$ when all $\vx$-components of $\vd$ are equal. Moreover, since $s_1^2=(1+\mu)\rho$, the entries of $\mT(\mu)$ are $O(\sqrt{\mu})$, so $\mT(\mu)\mT(\mu)^{\top}=O(\mu)$ and $\tfrac{1}{\mu}\mT(\mu)\mT(\mu)^{\top}=O(1)$. The cross-agent entries of $\tilde{\mJ}(\mu)$ are therefore at most $O(\sqrt{\mu})$ (from $\mS(\mu)$) plus $O(1)$ (from $\tfrac{1}{\mu}\mT\mT^{\top}$), both of which are $o(\mu)$. Therefore, dividing by $\mu$ and sending $\mu\to\infty$ yields the block-diagonal limit
\[
\frac{\tilde{\mJ}(\mu)}{\mu}\;\xrightarrow{\mu\to\infty}\; \mathrm{diag}\bigl(\mGamma_{\vx\vx}^1,\mGamma_{\vy\vy}^1,\ldots,\mGamma_{\vx\vx}^J,\mGamma_{\vy\vy}^J\bigr)\succ0,
\]
where each $\mGamma$ block is positive definite by relative bi-smoothness and the stepsize condition, exactly as in the $J=3$ case. For the $\rho$-cancellation at $\mu=0$, write $\vd_{\vx,\mathrm{stack}}=(\vd_\vx,\ldots,\vd_\vx)$. The consensus penalty contribution is $\rho\,\vd_{\vx,\mathrm{stack}}^{\top}(\mL_{\mathrm{star}}\otimes\eye_n)\vd_{\vx,\mathrm{stack}}=\rho\sum_{j=2}^{J}\|\vd_\vx-\vd_\vx\|_2^2=0$, because $\vd_{\vx,\mathrm{stack}}$ is in the kernel of the star-graph Laplacian. Combined with the strict-saddle negativity of the corresponding quadratic form at $\mu=0$, this gives the same sign change and hence the same intermediate-value argument for arbitrary $J\ge 2$.
\end{proof}
This completes the proof that strict-saddle KKT points in $\Omega$ are unstable fixed points of the fixed-point map of \Cref{alg:cadmm}.
By \Cref{thm:jason}, the set of initial points in $\Omega$ from which the iteration converges to such a strict-saddle KKT point has Lebesgue measure zero; since random initialization is absolutely continuous with respect to Lebesgue measure, the probability that the iterates converge to such a strict-saddle KKT point is zero.
\end{proof}

\end{document}